\documentclass [12pt]{amsart}
\usepackage[utf8]{inputenc}
\pdfoutput=1
\usepackage{amsmath,amssymb,amsthm,amsfonts}
\usepackage{mathtools}

\usepackage[english]{babel}
\usepackage{comment}
\usepackage{hyperref}
\usepackage{bbm}
\usepackage{mathrsfs}

\usepackage{tikz,graphicx,color}
\usepackage{tikz-cd}
\usepackage{tikz-3dplot}
\usetikzlibrary{calc}
\usetikzlibrary{arrows}
\usetikzlibrary{shapes}
\usetikzlibrary{patterns}
\usetikzlibrary{positioning}
\usetikzlibrary{arrows.meta}
\usetikzlibrary{decorations.markings}
\usetikzlibrary{knots}

\usepackage{epstopdf}

\usepackage[arrow]{xy}
\usepackage{diagbox}
\usepackage{subfig}
\usepackage{arcs}
\usepackage{xcolor}
\usepackage{xspace}

\usepackage[margin=1.0in]{geometry}

\usepackage{enumitem}
\usepackage[b]{esvect}
\usepackage{changepage}
\usepackage{letltxmacro}
\usepackage{thmtools,etoolbox}

\DeclareMathOperator{\newton}{Newton}

\def\myarabic#1{\normalfont(\roman{#1})}
\newlist{theoremlist}{enumerate}{1}
\setlist[theoremlist]{label=\myarabic{theoremlisti},ref={\myarabic{theoremlisti}},itemindent=0pt,labelindent=0pt,
  leftmargin=*,noitemsep}

\makeatletter
\renewcommand{\p@theoremlisti}{\perh@ps{\thetheorem}}
\protected\def\perh@ps#1#2{\textup{#1#2}}
\newcommand{\itemrefperh@ps}[2]{\textup{#2}}
\newcommand{\itemref}[1]{\begingroup\let\perh@ps\itemrefperh@ps\ref{#1}\endgroup}
\makeatother

\usepackage{nameref,hyperref}
\usepackage[capitalize]{cleveref}

\newtheorem{theorem}{Theorem}[section]

\newtheorem{lemma}[theorem]{Lemma}
\newtheorem{proposition}[theorem]{Proposition}
\newtheorem{corollary}[theorem]{Corollary}
\theoremstyle{definition}
\newtheorem{remark}[theorem]{Remark}
\theoremstyle{definition}
\newtheorem{definition}[theorem]{Definition}
\newtheorem{conjecture}[theorem]{Conjecture}

\theoremstyle{definition}

\theoremstyle{definition}
\newtheorem{example}[theorem]{Example}

\crefname{figure}{Figure}{Figures}

\def\figref#1(#2){Figure~\hyperref[#1]{\ref*{#1}(#2)}}

\addtotheorempostheadhook[theorem]{\crefalias{theoremlisti}{theorem}}
\addtotheorempostheadhook[lemma]{\crefalias{theoremlisti}{lemma}}
\addtotheorempostheadhook[proposition]{\crefalias{theoremlisti}{proposition}}
\addtotheorempostheadhook[corollary]{\crefalias{theoremlisti}{corollary}}

% ...

% ...

% ...

\def\<{{\langle}}
\def\>{{\rangle}}

\newcommand{\sm}[6]{\scalebox{1}{$\begin{smallmatrix}
    #1&#2&#3\\#4&#5&#6
   \end{smallmatrix}$}}

\def\Povtp_#1{\Pi_{#1}^{>0}}
\def\Povtnn_#1{\Pi_{#1}^{\geq0}}

\newcounter{todofigure}

\numberwithin{equation}{section}

\makeatletter
\@namedef{subjclassname@2020}{\textup{2020} Mathematics Subject Classification}
\makeatother

\begin{document}
\numberwithin{equation}{section}

\title{Classification of Zamolodchikov Periodic Cluster Algebras}
\author{Ariana Chin}
\address{Department of Mathematics, University of California, Los Angeles, CA 90095, USA}
\email{{\href{mailto:arianagchin@math.ucla.edu}{arianagchin@math.ucla.edu}}}
\thanks{This research did not receive any specific grant from funding agencies in the public, commercial, or not-for-profit sectors.}
\date{October 20, 2025}



\begin{abstract}
    Zamolodchikov periodicity is a property of certain discrete dynamical systems and was one of the primary motivations for the creation of cluster algebras.
    It was first observed by Zamolodchikov in his study of thermodynamic Bethe ansatz, initially for simply-laced Dynkin diagrams.
    It was proved by Keller to hold for tensor products of two Dynkin diagrams, and further shown by Galashin and Pylyavskyy to hold for pairs of commuting simply-laced Cartan matrices of finite type, which Stembridge classified in his study of admissible $W$-cells.
    We prove that the Zamolodchikov periodic cluster algebras are in bijection with pairs of commuting (not necessarily reduced or simply-laced) Cartan matrices of finite type.
    We fully classify all such pairs into 29 infinite families and 14 exceptional types in addition to the 6 infinite families and 11 exceptional types in Stembridge's classification, and show that all of these families can be derived from simply-laced types through two operations preserving Zamolodchikov periodicity, folding and taking transpose.
    Our work holds connections to Kazhdan--Lusztig theory, and our main theorem helps classify all nonnegative $W$-cells for products of two dihedral groups, $W = I_2(p)\times I_2(q)$.
\end{abstract}

\maketitle

\setcounter{tocdepth}{1}
\tableofcontents

\section{Introduction}
\subsection{Zamolodchikov periodicity}
For a generic cluster algebra, repeatedly applying the same nontrivial mutation sequence $\mu_1\mu_2\cdots\mu_k$ will yield infinitely many cluster variables.
However, if $\mathcal{A}$ is a cluster algebra of finite type, any mutation sequence will result in finitely many cluster variables.
In 2003, Fomin and Zelevinsky proved that the finite type cluster algebras are in bijection with finite type Cartan matrices \cite{fomin2}.
In particular, the finite type quiver cluster algebras are in bijection with the simply-laced (type A, D, or E) Cartan matrices, while the finite type cluster algebras with skew-symmetrizable exchange matrices correspond to the (not necessarily ADE) Cartan matrices.

In this paper, we focus on a specific bipartite mutation sequence and analyze which cluster algebras yield finitely many cluster variables under this sequence.

An $n\times n$ matrix $B$ is \textit{bipartite} if there exists a coloring $\epsilon: [n]\rightarrow \{\circ, \bullet\}$ of $[n]$ such that for all $i, j$ of the same color, $b_{ij} = b_{ji} = 0$.
It is well known that mutation is a local move, so mutations $\mu_i$ and $\mu_j$ commute if $b_{ij} = b_{ji} = 0$.
Thus in a bipartite matrix, we may define the \textit{bipartite mutations} $\mu_{\circ}$ and  $\mu_{\bullet}$ as the combined mutation at all white vertices and all black vertices, respectively.
We say a bipartite matrix $B$ is \textit{recurrent} if $\mu_{\circ}(B) = \mu_{\bullet}(B) = -B$. 
In particular, $\mu_{\circ}\mu_{\bullet}(B) = B$. 
A bipartite recurrent matrix $B$ is \textit{Zamolodchikov periodic} if the sequence of mutations $\mu_{\circ}\mu_{\bullet}\mu_{\circ}\mu_{\bullet}\dots$ acting on the seed is periodic.

Zamolodchikov periodicity was first observed by Zamolodchikov in his study \cite{zamolodchikov} of thermodynamic Bethe ansatz, initially focusing on the simply-laced Dynkin diagrams $A_n, D_n$, and $E_n$.
His conjecture was generalized by Ravanini--Valeriani--Tateo \cite{ravanini}, Kuniba--Nakanishi \cite{kuniba}, Kuniba--Nakanishi--Suzuki \cite{kuniba2}, and Fomin--Zelevinsky \cite{fomin3}. 
The conjecture was proved for all tensor products $\Gamma\otimes \Delta$ of Dynkin diagrams by Keller \cite{keller} and later in a different way by Inoue--Iyama--Keller--Kuniba--Nakanishi \cite{inoue, inoue2}.
More recently, Zamolodchikov periodic quivers were fully classified by Galashin--Pylyavskyy in 2019 \cite{pasha}.

\subsection{The main theorem}

The following theorem is the main result of this paper.
\begin{theorem}
\label{mainThm}
    The Zamolodchikov periodic $B$-matrices are in natural bijection with pairs of commuting Cartan matrices $(\Gamma, \Delta)$.
    The period divides $h_{\Gamma} + h_{\Delta}$, where $h_{\Gamma}, h_{\Delta}$ are the Coxeter numbers of $\Gamma, \Delta$.
\end{theorem}
\begin{remark}
    It is worth noting that the different Zamolodchikov periodic $B$-matrices do not necessarily generate distinct cluster algebras.
    For example, there are two distinct Zamolodchikov periodic quivers, $E_6$ and $A_2\otimes A_3$, that produce the same cluster algebra.
    $E_6$ has a period of length 14, while $A_3\otimes A_2$ has a period of length 7.
    Similarly, $D_4$ and $A_2\otimes A_2$ are distinct Zamolodchikov periodic quivers with respective periods 8 and 6 that produce the same cluster algebra.
\end{remark}

Now, we formulate the plan of the paper.
In Section \ref{prelims}, we introduce bipartite dynamics of cluster algebras and their connections to Dynkin diagrams through Stembridge's classification.
Then, the proof of Theorem \ref{mainThm} is broken up into two parts. 
In Section \ref{classification}, we classify all pairs of commuting Cartan matrices as admissible Dynkin biagrams. 
We find that there are 29 infinite families and 14 exceptional types in addition to the 6 infinite families and 11 exceptional types classified in \cite{stembridge}.

In Section \ref{ZP}, we prove that these admissible Dynkin biagrams are in bijection with Zamolodchikov periodic $B$-matrices through a 5-way equivalence whose outline was adapted from \cite{pasha}. 
In particular, we show that the class of Zamolodchikov periodic $B$-matrices is preserved under two operations: folding and taking transpose ($B\mapsto B^T$); see Figures \ref{fig:folding} and \ref{fig:BDCD}. It is a standard fact that folding the $B$-matrix commutes with cluster mutations and thus preserves Zamolodchikov periodicity. On the other hand, taking transpose does not commute with cluster mutations, and thus it is surprising that it preserves the class of Zamolodchikov periodic $B$-matrices.

In Section \ref{$W$-graphs}, we explore connections between admissible Dynkin biagrams and $W$-graphs from Kazhdan--Lusztig theory. In particular, we generalize a result in \cite{stembridge} which connects Zamolodchikov periodic quivers to admissible $W$-cells for $W$ the product of two dihedral groups.

Finally in Section \ref{conjectures}, we conjecture behavior of a variant of tropical mutation on Zamolodchikov periodic $B$-matrices.

\section{Preliminaries}
\label{prelims}
\subsection{Cluster algebras}
Cluster algebras were first introduced in \cite{fomin1} by Fomin and Zelevinsky with applications to dual canonical bases, total positivity in semisimple Lie groups, and Zamolodchikov periodicity for $Y$-systems of finite root systems.

An $n\times n$ matrix $B$ is \textit{skew-symmetrizable} if there exists some $c\in \mathbb{R}_{>0}^n$ such that $\forall i, j\in [n]$, $c_ib_{ij} = -c_jb_{ji}$.
In other words, a matrix $B$ is skew-symmetrizable if there exists some nonzero scaling of the rows $c\in \mathbb{R}_{>0}^n$ that makes the matrix skew-symmetric.

\begin{definition}
    A \textit{labeled seed} in field $\mathcal{F}$ is a pair $(\boldsymbol{x}, B)$ where
    \begin{itemize}
        \item 
            $\boldsymbol{x} = (x_1, \dots, x_n)$ is an $n$-tuple of elements in $\mathcal{F}$ forming a \textit{free generating set}; that is, $x_1, \dots, x_n$ are algebraically independent, and $\mathcal{F} = \mathbb{C}(x_1, \dots, x_n)$,
        \item 
            $B = (b_{ij})$ is an $n\times n$ skew-symmetrizable integer matrix.
    \end{itemize}
\end{definition}
We use the following notation:
\begin{itemize}
    \item 
        $\boldsymbol{x}$ is the (labeled) \textit{cluster} of this seed $(\boldsymbol{x}, B)$,
    \item 
        the elements $x_1, \dots, x_n$ are its \textit{cluster variables},
    \item 
        $B$ is the \textit{exchange matrix}, also known as the $B$-matrix.
\end{itemize}
A \textit{quiver} is a directed graph without loops and cycles of length 2.
When an exchange matrix $B$ is skew-symmetric, the seed can be represented as a quiver $Q$, where $B$ is the adjacency matrix of $Q$.

\begin{definition}
    Let $(\boldsymbol{x}, B)$ be a labeled seed. 
    Let $k\in [n]$.
    The \textit{seed mutation} $\mu_k$ in direction $k$ transforms $(\boldsymbol{x}, B)$ into the new labeled seed $\mu_k(\boldsymbol{x}, B) = (\boldsymbol{x'}, B')$ defined as follows:
    \begin{itemize}
        \item 
            The cluster $\boldsymbol{x'} = (x_1', \dots, x_n')$ is given by $x_j' = x_j$ for all $j\neq k$, whereas $x_k'$ is determined by the \textit{exchange relation}
            \begin{equation}
            \label{mutation}
                x_kx_k' = \prod_{b_{ik} > 0} x_i^{|b_{ik}|} + \prod_{b_{ik} < 0} x_i^{|b_{ik}|}.
            \end{equation}
        \item  
            The exchange matrix $B' = (b_{ij}')$ is given by
            $$b_{ij}' = \begin{cases}
                -b_{ij} & \text{if }i = k \text{ or } j = k,\\
                b_{ij} + b_{ik}b_{kj} & \text{if } b_{ik}, b_{kj} > 0,\\
                b_{ij} - b_{ik}b_{kj} & \text{if } b_{ik}, b_{kj} < 0,\\
                b_{ij} & \text{otherwise.}
            \end{cases}$$
    \end{itemize}
\end{definition}

\begin{definition}
    Let $(\boldsymbol{x}, B)$ be a seed, and let $\mathcal{X}$ be the set of all cluster variables obtained from any sequence of mutations of this seed.
    The \textit{cluster algebra} $\mathcal{A}$ of geometric type over $\mathbb{C}$ generated by the seed $(\boldsymbol{x}, B)$ is the $\mathbb{C}$-subalgebra of the ambient field $\mathbb{C}(\boldsymbol{x})$ generated by all cluster variables: $\mathcal{A} = \mathbb{C}[\mathcal{X}]$.
\end{definition}

In general, $\mathcal{X}$ is an infinite set. When $\mathcal{X}$ is finite, we say the cluster algebra $\mathcal{A}$ is of \textit{finite type}.

\subsection{$T$-systems and Zamolodchikov periodicity}
Given a bipartite recurrent $B$-matrix, one can decompose it as a sum of two skew-symmetrizable matrices $B = \tilde{\Gamma} + \tilde{\Delta}$, $\tilde{\Gamma} = (\tilde{\Gamma}_{ij})$ and $\tilde{\Delta} = (\tilde{\Delta}_{ij})$, as follows:
\begin{align*}
    \tilde{\Gamma}_{ij} \coloneqq \begin{cases}
        b_{ij} & \text{if } b_{ij} > 0, \epsilon_i = \circ, \epsilon_j = \bullet,\\
        b_{ij} & \text{if } b_{ij} < 0, \epsilon_i = \bullet, \epsilon_j = \circ,\\
        0 & \text{otherwise.}
    \end{cases}\qquad
    \tilde{\Delta}_{ij} \coloneqq \begin{cases}
        b_{ij} & \text{if } b_{ij} > 0, \epsilon_i = \bullet, \epsilon_j = \circ,\\
        b_{ij} & \text{if } b_{ij} < 0, \epsilon_i = \circ, \epsilon_j = \bullet,\\
        0 & \text{otherwise.}
    \end{cases}
\end{align*}
Throughout the paper, we often use these matrices as unsigned matrices, so we define the matrices $\Gamma = (\Gamma_{ij})$ and $\Delta = (\Delta_{ij})$ by
$$\Gamma_{ij} \coloneqq |\tilde{\Gamma}_{ij}|, \hspace{5mm} \Delta_{ij}\coloneqq |\tilde{\Delta}_{ij}|$$
for all $i,j\in[n]$.
One can also associate to $B = \tilde{\Gamma} + \tilde{\Delta}$ a \textit{$T$-system}, an infinite system of algebraic equations, defined as follows.
Let $\boldsymbol{x} = (x_1, \dots, x_n)$, and let $\mathbb{Q}(\boldsymbol{x})$ be the field of rational functions in these variables. The \textit{$T$-system} associated with $B$ is a family $T_k(t)$ of elements of $\mathbb{Q}(\boldsymbol{x})$ satisfying the following relations for all $k\in [n]$ and all $t\in \mathbb{Z}$:
\begin{align*}
    T_k(t+1)T_k(t-1) = \prod_{i} T_i(t)^{\Gamma_{ik}} + \prod_j T_j(t)^{\Delta_{jk}}.
\end{align*}

At each time step, when $\epsilon_k = \circ$, $T_k(t+1)$ is only dependent on black indices ($\epsilon_i = \bullet$).
Similarly when $\epsilon_k = \bullet$, $T_k(t+1)$ is only dependent on white indices ($\epsilon_i = \circ$).
Because of this, the $T$-system associated with $B$ splits into two independent components.
For our purposes, we consider only one component.
From now on, assume that the $T$-system is defined only for $k\in [n]$ and $t\in \mathbb{Z}$ such that
\begin{equation}
\label{splitTsystem}
    \epsilon_k = \circ\quad\text{and}\quad t\equiv 0\hspace{-2mm}\pmod{2}\qquad \text{or} \qquad \epsilon_k = \bullet\quad \text{and} \quad t\equiv 1\hspace{-2mm}\pmod{2}.
\end{equation}
The $T$-system is set to the following initial conditions:
\begin{align*}
    T_k(0) = x_k\quad \text{and}\quad T_k(1) = x_k, \quad\text{for all } k \text{ satisfying } (\ref{splitTsystem}).
\end{align*}
Notice that the $T$-system relations are exactly the same as the exchange relations (\ref{mutation}) for cluster mutation in a bipartite recurrent setting. In particular, this initialization will result in the same sequence of cluster variables as the mutation sequence $\mu_{\circ}\mu_{\bullet}\mu_{\circ}\mu_{\bullet}\dots$ acting on the initial seed $(\boldsymbol{x}, B)$.
We say the $T$-system is \textit{periodic} if there exists some positive integer $N$ such that $T_k(t + 2N) = T_k(t)$ for all $k\in[n], t\in\mathbb{Z}$. 
The \textit{period} is the smallest such $N$.

\begin{example}
Consider the following bipartite recurrent $5\times 5$ $B$-matrix $B = \tilde{\Gamma} + \tilde{\Delta}$ with bipartition $\epsilon_1 = \epsilon_3 = \epsilon_4 = \bullet$ and $\epsilon_2 = \epsilon_5 = \circ$.
\begin{align*}
    B = \begin{bmatrix}
        0 & -1 & 0 & 0 & 1\\
        1 & 0 & 1 & -1 & 0\\
        0 & -2 & 0 & 0 & 2\\
        0 & 3 & 0 & 0 & -3\\
        -1 & 0 & -1 & 1 & 0
    \end{bmatrix} = \begin{bmatrix}
        0 & -1 & 0 & 0 & 0\\
        1 & 0 & 1 & 0 & 0\\
        0 & -2 & 0 & 0 & 0\\
        0 & 0 & 0 & 0 & -3\\
        0 & 0 & 0 & 1 & 0
    \end{bmatrix} + \begin{bmatrix}
        0 & 0 & 0 & 0 & 1\\
        0 & 0 & 0 & -1 & 0\\
        0 & 0 & 0 & 0 & 2\\
        0 & 3 & 0 & 0 & 0\\
        -1 & 0 & -1 & 0 & 0
    \end{bmatrix}.
\end{align*}
Let us compute the first few elements of the $T$-system associated with $B$.
\begin{align*}
    T_2(0) &= x_2, \hspace{5mm} T_5(0) = x_5, \hspace{5mm} T_1(1) = x_1, \hspace{5mm} T_3(1) = x_3, \hspace{5mm} T_4(1) = x_4,\\
    T_2(2) &= \frac{x_1x_3^2 + x_4^3}{x_2}, \hspace{5mm} T_5(2) = \frac{x_1x_3^2 + x_4^3}{x_5},\\
    T_1(3) &= \frac{x_1x_3^2 + x_4^3}{x_1x_2} + \frac{x_1x_3^2 + x_4^3}{x_1x_5}, \hspace{5mm} T_3(3) = \frac{x_1x_3^2 + x_4^3}{x_3x_2} + \frac{x_1x_3^2 + x_4^3}{x_3x_5}.
\end{align*}
\end{example}

\begin{definition}
    A bipartite recurrent $B$-matrix is \textit{Zamolodchikov periodic} if its associated $T$-system is periodic.
\end{definition}

\subsection{Tropical $T$-systems and $Y$-systems}
Given a bipartite recurrent $B$-matrix, we define an analogous system of algebraic equations.

Let $\lambda\in\mathbb{R}^n$ be a labeling of the $n$ indices.
The \textit{tropical $T$-system} $\boldsymbol{t}^{\lambda}$ associated with $B$ is a family $\boldsymbol{t}_k^{\lambda}(t)\in\mathbb{R}$ of real numbers satisfying the following  relations for all $k\in [n]$ and all $t\in\mathbb{Z}$:
\begin{align*}
    \boldsymbol{t}_k^{\lambda}(t + 1) + \boldsymbol{t}_k^{\lambda}(t - 1) = \max\left(\sum_i \Gamma_{ik}\boldsymbol{t}_i^{\lambda}(t), \sum_j \Delta_{jk}\boldsymbol{t}_j^{\lambda}(t)\right).
\end{align*}
Again, we only consider the subsystem defined for $k\in [n]$ and $t\in\mathbb{Z}$ satisfying (\ref{splitTsystem}).
The tropical $T$-system is set to the following initial conditions:
\begin{align*}
    \boldsymbol{t}_k^{\lambda}(0) = \lambda_k\quad\text{and}\quad \boldsymbol{t}_k^{\lambda}(1) &= \lambda_k, \quad \text{for all } k \text{ satisfying } (\ref{splitTsystem}).
\end{align*}
In other words, this system is the \textit{tropicalization} of the $T$-system associated with $B$ (see Section \ref{polytopes}).

We say the associated tropical $T$-system $\boldsymbol{t}^{\lambda}$ is \textit{periodic} if there exists a positive integer $N$ such that $\boldsymbol{t}_k^{\lambda}(t + 2N) = \boldsymbol{t}_k^{\lambda}(t)$ for all $k\in[n], t\in\mathbb{Z}$.

\begin{example}
    Let us continue working with the same $5\times 5$ matrix $B$ from our previous example with $\epsilon_1 = \epsilon_3 = \epsilon_4 = \bullet$ and $\epsilon_2 = \epsilon_5 = \circ$.
    Let $\lambda = e_5 = [0, 0, 0, 0, 1]^{\top}$.
    Then, the initial values of $\boldsymbol{t}_k^{\lambda}(t)$ are 
    \begin{align*}
        \boldsymbol{t}_2^{\lambda}(0) = 0, \hspace{5mm} \boldsymbol{t}_5^{\lambda}(0) = 1, \hspace{5mm} \boldsymbol{t}_1^{\lambda}(1) = 0, \hspace{5mm} \boldsymbol{t}_3^{\lambda}(1) = 0, \hspace{5mm}
        \boldsymbol{t}_4^{\lambda}(1) = 0.
    \end{align*}
    The values of the tropical $T$-system $\boldsymbol{t}^{\lambda}$ are given in Table \ref{table:tropicalex}.
    Notice that $\boldsymbol{t}_k^{\lambda}(t + 12) = \boldsymbol{t}_k^{\lambda}(t)$ for all $k\in [6]$.
    This is a special case of the tropical version of Zamolodchikov periodicity where the period is $N = 6$, the Coxeter number of $B_3$ and $G_2$ (see Table \ref{table:1}).
\end{example}
\begin{table}
\tiny
    \begin{tabular}{|c|ccccc|}\hline
    \backslashbox{$t$}{$k$}     &  1  & 2  & 3 & 4 & 5   \\\hline
     $0$      &       & 0     &      &    &  1  \\\hline
     $1$      & 0     &       &   0  &  0 &     \\\hline
     $2$      &       & 0     &      &    &  -1 \\\hline
     $3$      & 0     &       &   0  &  0 &     \\\hline
     $4$      &       & 0     &      &    &  1  \\\hline
     $5$      & 1     &       & 1    &  1 &     \\\hline
     $6$      &       & 3     &      &    &  2  \\\hline
     $7$      & 2     &       & 2    &  2 &     \\\hline
     $8$      &       & 3     &      &    &  4  \\\hline
     $9$      & 2     &       &   2  &  2 &     \\\hline
     $10$     &       & 3     &      &    &  2  \\\hline
     $11$     & 1     &       &   1  &  1 &     \\\hline
     $12$     &       & 0     &      &    &  1  \\\hline
     $13$     & 0     &       &   0  &  0 &     \\\hline
    \end{tabular}
    \caption{The evolution of values of the tropical $T$-system $\boldsymbol{t}_k^{\lambda}(t)$ for $B_3\bowtie_1 G_2$.}
    \label{table:tropicalex}
\end{table}

\subsection{ADE bigraphs}
In \cite{stembridge}, Stembridge studies admissible $W$-cells in the case of $W = I(p)\times I(q)$, direct products of two dihedral groups. He classifies all such $W$-cells which are in 4-to-1 correspondence with \textit{admissible ADE bigraphs}, defined as follows.

An \textit{ADE bigraph} is an ordered pair of simple undirected graphs $(G_1, G_2)$ which do not share edges such that $G_1\cup G_2$ is bipartite and each connected component of $G_1$, $G_2$ is an ADE Dynkin diagram. 
The \textit{dual} of an ADE bigraph $(G_1, G_2)$, denoted $(G_1, G_2)^*$, is defined to be $(G_2, G_1)$.
An ADE bigraph $(G_1, G_2)$ is \textit{admissible} if the adjacency matrices $A_{G_1}$, $A_{G_2}$ commute.
We can depict an ADE bigraph as a single graph with red and blue edges, where the red components are associated to $G_1$ and the blue components are associated to $G_2$.

Notice that the adjacency matrices here are closely related to the Cartan matrices for $G_1$ and $G_2$, $C_{G_1}$ and $C_{G_2}$. In particular, $A_{G_1} = 2I - C_{G_1}$ and $A_{G_2} = 2I - C_{G_2}$, where $A_G$ is the \textit{Coxeter adjacency matrix} of $G$. If $B$ is a skew-symmetric matrix such that the element-wise absolute value $|B|$ is the associated Coxeter adjacency matrix $A_G$, we call $B$ a \textit{signed Coxeter adjacency matrix} of $G$.

\begin{remark}
    One can also visualize admissibility of an ADE bigraph. 
    An ADE bigraph is admissible if and only if the number of length two red-blue paths from $i$ to $j$ is the same as the number of length two blue-red paths from $i$ to $j$ for every pair of vertices $i, j$.
    An example of admissible and nonadmissible ADE bigraphs is shown in Figure \ref{fig:admissibleADE}.
\end{remark}

\begin{figure}
\scalebox{0.5}{
\begin{tikzpicture}

\coordinate (v0x0) at (0.00,0.00);
\coordinate (v0x1) at (0.00,2.00);
\coordinate (v0x2) at (-1.00,4.00);
\coordinate (v0x3) at (2.00,3.00);
\coordinate (v0x4) at (5.00,0.00);
\coordinate (v0x5) at (5.00,2.00);
\coordinate (v0x6) at (5.00,4.00);
\coordinate (v0x7) at (13.00,0.00);
\coordinate (v0x8) at (13.00,2.00);
\coordinate (v0x9) at (12.00,4.00);
\coordinate (v0x10) at (15.00,3.00);
\coordinate (v0x11) at (18.00,0.50);
\coordinate (v0x12) at (18.00,2.50);
\coordinate (v0x13) at (20.00,3.50);
\coordinate (v0x14) at (21.50,1.50);
\coordinate (v0x15) at (21.50,-0.50);

\draw[color=red,line width=0.75mm] (v0x1) to[] (v0x0);
\draw[color=red,line width=0.75mm] (v0x2) to[] (v0x1);
\draw[color=red,line width=0.75mm] (v0x3) to[] (v0x1);
\draw[color=blue,line width=0.75mm] (v0x4) to[] (v0x0);
\draw[color=red,line width=0.75mm] (v0x5) to[] (v0x4);
\draw[color=blue,line width=0.75mm] (v0x5) to[] (v0x1);
\draw[color=red,line width=0.75mm] (v0x6) to[] (v0x5);
\draw[color=blue,line width=0.75mm] (v0x6) to[] (v0x2);
\draw[color=blue,line width=0.75mm] (v0x6) to[] (v0x3);

\draw[color=red,line width=0.75mm] (v0x8) to[] (v0x7);
\draw[color=red,line width=0.75mm] (v0x9) to[] (v0x8);
\draw[color=red,line width=0.75mm] (v0x10) to[] (v0x8);
\draw[color=blue,line width=0.75mm] (v0x11) to[] (v0x7);
\draw[color=blue,line width=0.75mm] (v0x15) to[] (v0x7);
\draw[color=red,line width=0.75mm] (v0x12) to[] (v0x11);
\draw[color=blue,line width=0.75mm] (v0x12) to[] (v0x8);
\draw[color=blue,line width=0.75mm] (v0x14) to[] (v0x8);
\draw[color=red,line width=0.75mm] (v0x13) to[] (v0x12);
\draw[color=blue,line width=0.75mm] (v0x13) to[] (v0x9);
\draw[color=blue,line width=0.75mm] (v0x13) to[] (v0x10);
\draw[color=red,line width=0.75mm] (v0x13) to[] (v0x14);
\draw[color=red,line width=0.75mm] (v0x14) to[] (v0x15);

\draw[fill=white] (v0x0.center) circle (0.2);
\draw[fill=black!75!white] (v0x1.center) circle (0.2);
\draw (-1.25, 2.25) node [anchor=north west][inner sep=0.75pt] {\LARGE{$v_1$}};
\draw[fill=white] (v0x2.center) circle (0.2);
\draw[fill=white] (v0x3.center) circle (0.2);
\draw[fill=black!75!white] (v0x4.center) circle (0.2);
\draw[fill=white] (v0x5.center) circle (0.2);
\draw[fill=black!75!white] (v0x6.center) circle (0.2);
\draw (5.50, 4.25) node [anchor=north west][inner sep=0.75pt] {\LARGE{$v_2$}};
\draw[fill=white] (v0x7.center) circle (0.2);
\draw[fill=black!75!white] (v0x8.center) circle (0.2);
\draw[fill=white] (v0x9.center) circle (0.2);
\draw[fill=white] (v0x10.center) circle (0.2);
\draw[fill=black!75!white] (v0x11.center) circle (0.2);
\draw[fill=white] (v0x12.center) circle (0.2);
\draw[fill=black!75!white] (v0x13.center) circle (0.2);
\draw[fill=white] (v0x14.center) circle (0.2);
\draw[fill=black!75!white] (v0x15.center) circle (0.2);
\end{tikzpicture}}
\caption{\label{fig:admissibleADE} A nonadmissible ADE bigraph (left) and an admissible ADE bigraph (right). The vertices $v_1$ and $v_2$ form a nonadmissible pair.}
\end{figure}
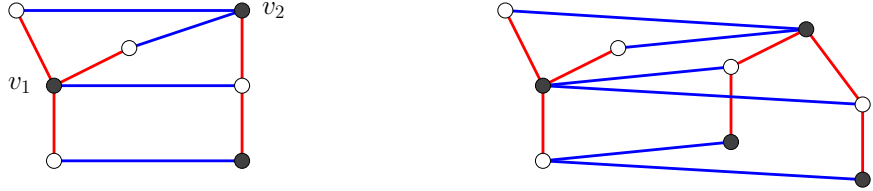

Given an ADE bigraph $(G_1, G_2)$ with bipartition $\epsilon: [n]\rightarrow \{\circ, \bullet\}$, there is a unique associated bipartite quiver $Q$ defined as follows:
\begin{itemize}
    \item 
        If $(u, v)$ is an edge in $G_1$, direct it so that it points from $\circ$ to $\bullet$.
    \item 
        If $(u, v)$ is an edge in $G_2$, direct it so that it points from $\bullet$ to $\circ$.
\end{itemize}
Let $\tilde{\Gamma}, \tilde{\Delta}$ be the resulting signed Coxeter adjacency matrices for $G_1$ and $G_2$, respectively.
Then, $B = \tilde{\Gamma} + \tilde{\Delta}$ is the adjacency matrix of the quiver $Q$.

Stembridge's classification of admissible ADE bigraphs includes 6 infinite families and 11 exceptional types. 
Below, we list a few of these types.

\textbf{Tensor products}.
Let $\Lambda$ and $\Lambda'$ be two Dynkin diagrams of rank $n$ and $m$, respectively.
The \textit{tensor product}, or \textit{box product}, of $\Lambda$ and $\Lambda'$ is denoted $\Lambda\otimes \Lambda'$. 
It consists of $m$ red copies of $\Lambda$ connected to each other with blue edges.
The red and blue Coxeter adjacency matrices $\Gamma$ and $\Delta$ are Kronecker products of the form $A_{\Lambda}\otimes I_m$ and $I_n\otimes A_{\Lambda'}$, respectively.
From this, it is straightforward to see that tensor products are admissible.
An example of the tensor product $D_4\otimes A_4$ is given in Figure \ref{fig:tensor}.

\begin{figure}
\scalebox{0.5}{
\begin{tikzpicture}

\foreach \i in {0, 1, 2, 3}
{
    \coordinate (v\i x0) at (4 * \i + 0.00,0.00);
    \coordinate (v\i x1) at (4 * \i + 0.00,2.00);
    \coordinate (v\i x2) at (4 * \i + -1.00,4.00);
    \coordinate (v\i x3) at (4 * \i + 2.00,3.00);
    \draw[color=red,line width=0.75mm] (v\i x1) to[] (v\i x0);
    \draw[color=red,line width=0.75mm] (v\i x2) to[] (v\i x1);
    \draw[color=red,line width=0.75mm] (v\i x3) to[] (v\i x1);
}
\foreach \i in {0, 1, 2, 3}
{
    \draw[color=blue,line width=0.75mm] (v0x\i) to[] (v1x\i);
    \draw[color=blue,line width=0.75mm] (v1x\i) to[] (v2x\i);
    \draw[color=blue,line width=0.75mm] (v2x\i) to[] (v3x\i);
}

\foreach \i in {0, 2}
{
    \draw[fill=black!75!white] (v\i x0.center) circle (0.2);
    \draw[fill=white] (v\i x1.center) circle (0.2);
    \draw[fill=black!75!white] (v\i x2.center) circle (0.2);
    \draw[fill=black!75!white] (v\i x3.center) circle (0.2);
}

\foreach \i in {1, 3}
{
    \draw[fill=white] (v\i x0.center) circle (0.2);
    \draw[fill=black!75!white] (v\i x1.center) circle (0.2);
    \draw[fill=white] (v\i x2.center) circle (0.2);
    \draw[fill=white] (v\i x3.center) circle (0.2);
}
\end{tikzpicture}}
\caption{\label{fig:tensor} The tensor product $D_4\otimes A_4$.}
\end{figure}
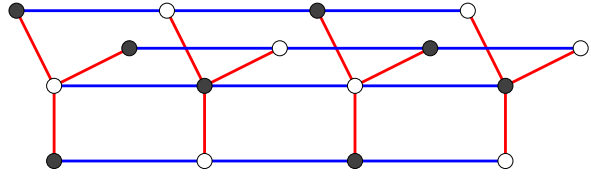

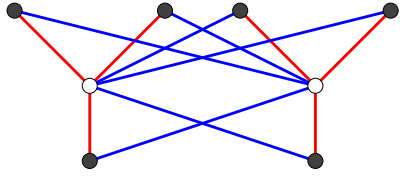
\begin{figure}
\scalebox{0.5}{
\begin{tikzpicture}

\foreach \i in {0, 1}
{
    \coordinate (v\i x0) at (6 * \i + 0.00,0.00);
    \coordinate (v\i x1) at (6 * \i + 0.00,2.00);
    \coordinate (v\i x2) at (6 * \i + -2.00,4.00);
    \coordinate (v\i x3) at (6 * \i + 2.00,4.00);
    \draw[color=red,line width=0.75mm] (v\i x1) to[] (v\i x0);
    \draw[color=red,line width=0.75mm] (v\i x2) to[] (v\i x1);
    \draw[color=red,line width=0.75mm] (v\i x3) to[] (v\i x1);
}
\draw[color=blue,line width=0.75mm] (v0x1) to[] (v1x0);
\draw[color=blue,line width=0.75mm] (v0x0) to[] (v1x1);
\draw[color=blue,line width=0.75mm] (v0x1) to[] (v1x2);
\draw[color=blue,line width=0.75mm] (v0x1) to[] (v1x3);
\draw[color=blue,line width=0.75mm] (v0x2) to[] (v1x1);
\draw[color=blue,line width=0.75mm] (v0x3) to[] (v1x1);
\foreach \i in {0, 1}
{
    \draw[fill=black!75!white] (v\i x0.center) circle (0.2);
    \draw[fill=white] (v\i x1.center) circle (0.2);
    \draw[fill=black!75!white] (v\i x2.center) circle (0.2);
    \draw[fill=black!75!white] (v\i x3.center) circle (0.2);
}
\end{tikzpicture}}
\caption{\label{fig:twist} The twist $D_4\times D_4$.}
\end{figure}

\begin{figure}
\scalebox{0.5}{
\begin{tikzpicture}

\coordinate (v0x0) at (0.00,0.00);
\coordinate (v0x1) at (0.00,2.00);
\coordinate (v0x2) at (2.00,3.00);
\coordinate (v0x3) at (3.50,1.00);
\coordinate (v0x4) at (3.50,-1.00);
\coordinate (v0x5) at (2.00,5.00);

\coordinate (v1x0) at (6.00,5.50);
\coordinate (v1x1) at (6.00,3.50);
\coordinate (v1x2) at (7.50,1.50);
\coordinate (v1x3) at (9.50,2.50);
\coordinate (v1x4) at (9.50,4.50);
\coordinate (v1x5) at (7.50,-0.50);

\draw[color=blue,line width=0.75mm] (v0x0) to[] (v1x5);
\draw[color=blue,line width=0.75mm] (v0x4) to[] (v1x5);
\draw[color=blue,line width=0.75mm] (v0x1) to[] (v1x2);
\draw[color=blue,line width=0.75mm] (v0x3) to[] (v1x2);
\draw[color=blue,line width=0.75mm] (v0x2) to[] (v1x1);
\draw[color=blue,line width=0.75mm] (v0x2) to[] (v1x3);
\draw[color=blue,line width=0.75mm] (v0x5) to[] (v1x0);
\draw[color=blue,line width=0.75mm] (v0x5) to[] (v1x4);
\foreach \i in {0, 1}
{
    \draw[color=red,line width=0.75mm] (v\i x0) to[] (v\i x1);
    \draw[color=red,line width=0.75mm] (v\i x1) to[] (v\i x2);
    \draw[color=red,line width=0.75mm] (v\i x2) to[] (v\i x3);
    \draw[color=red,line width=0.75mm] (v\i x3) to[] (v\i x4);
    \draw[color=red,line width=0.75mm] (v\i x2) to[] (v\i x5);
    \draw[fill=black!75!white] (v\i x0.center) circle (0.2);
    \draw[fill=white] (v\i x1.center) circle (0.2);
    \draw[fill=black!75!white] (v\i x2.center) circle (0.2);
    \draw[fill=white] (v\i x3.center) circle (0.2);
    \draw[fill=black!75!white] (v\i x4.center) circle (0.2);
    \draw[fill=white] (v\i x5.center) circle (0.2);
}
\end{tikzpicture}}
\caption{\label{fig:E6E6} The binding $E_6\ast E_6$.}
\end{figure}
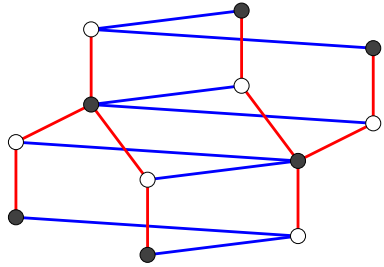

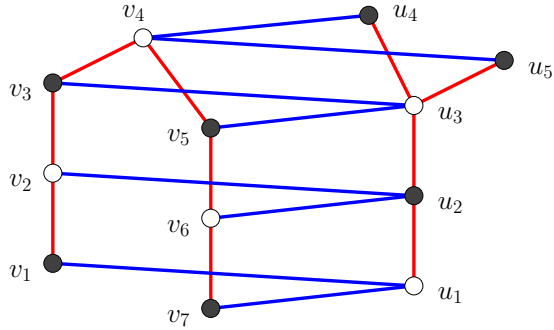
\begin{figure}
\scalebox{0.6}{
\begin{tikzpicture}
\coordinate (v0x0) at (0.00,0.00);
\draw (-1.00, 0.00) node [anchor=north west][inner sep=0.75pt] {\Large{$v_1$}};
\coordinate (v0x1) at (0.00,2.00);
\draw (-1.00, 2.00) node [anchor=north west][inner sep=0.75pt] {\Large{$v_2$}};
\coordinate (v0x2) at (0.00,4.00);
\draw (-1.00, 4.00) node [anchor=north west][inner sep=0.75pt] {\Large{$v_3$}};
\coordinate (v0x3) at (2.00,5.00);
\draw (1.50, 5.75) node [anchor=north west][inner sep=0.75pt] {\Large{$v_4$}};
\coordinate (v0x4) at (3.50,3.00);
\draw (2.50, 3.00) node [anchor=north west][inner sep=0.75pt] {\Large{$v_5$}};
\coordinate (v0x5) at (3.50,1.00);
\draw (2.50, 1.00) node [anchor=north west][inner sep=0.75pt] {\Large{$v_6$}};
\coordinate (v0x6) at (3.50,-1.00);
\draw (2.50, -1.00) node [anchor=north west][inner sep=0.75pt] {\Large{$v_7$}};
\coordinate (v1x0) at (8.00,-0.50);
\draw (8.50, -0.50) node [anchor=north west][inner sep=0.75pt] {\Large{$u_1$}};
\coordinate (v1x1) at (8.00,1.50);
\draw (8.50, 1.50) node [anchor=north west][inner sep=0.75pt] {\Large{$u_2$}};
\coordinate (v1x2) at (8.00,3.50);
\draw (8.50, 3.50) node [anchor=north west][inner sep=0.75pt] {\Large{$u_3$}};
\coordinate (v1x3) at (7.00,5.50);
\draw (7.50, 5.75) node [anchor=north west][inner sep=0.75pt] {\Large{$u_4$}};
\coordinate (v1x4) at (10.00,4.50);
\draw (10.50, 4.50) node [anchor=north west][inner sep=0.75pt] {\Large{$u_5$}};
\foreach \i in {0, 1, 2, 3, 4, 5}
{
    \pgfmathtruncatemacro{\j}{\i + 1};
    \draw[color=red,line width=0.75mm] (v0x\i) to[] (v0x\j);
}
\draw[color=red,line width=0.75mm] (v1x0) to[] (v1x1);
\draw[color=red,line width=0.75mm] (v1x1) to[] (v1x2);
\draw[color=red,line width=0.75mm] (v1x2) to[] (v1x3);
\draw[color=red,line width=0.75mm] (v1x2) to[] (v1x4);
\foreach \i in {0, 1, 2}
{
    \pgfmathtruncatemacro{\j}{6 - \i};
    \draw[color=blue,line width=0.75mm] (v0x\i) to[] (v1x\i);
    \draw[color=blue,line width=0.75mm] (v0x\j) to[] (v1x\i);
}
\draw[color=blue,line width=0.75mm] (v0x3) to[] (v1x3);
\draw[color=blue,line width=0.75mm] (v0x3) to[] (v1x4);

\draw[fill=black!75!white] (v0x0.center) circle (0.2);
\draw[fill=white] (v0x1.center) circle (0.2);
\draw[fill=black!75!white] (v0x2.center) circle (0.2);
\draw[fill=white] (v0x3.center) circle (0.2);
\draw[fill=black!75!white] (v0x4.center) circle (0.2);
\draw[fill=white] (v0x5.center) circle (0.2);
\draw[fill=black!75!white] (v0x6.center) circle (0.2);
\draw[fill=white] (v1x0.center) circle (0.2);
\draw[fill=black!75!white] (v1x1.center) circle (0.2);
\draw[fill=white] (v1x2.center) circle (0.2);
\draw[fill=black!75!white] (v1x3.center) circle (0.2);
\draw[fill=black!75!white] (v1x4.center) circle (0.2);
\end{tikzpicture}}
\caption{\label{fig:AnDn} The binding $A_7\ast D_5$.}
\end{figure}

\textbf{Twists}.
Twists are denoted $\Lambda\times\Lambda$, and consist of two red copies of $\Lambda$ connected to each other with blue edges.
The red and blue Coxeter adjacency matrices are of the form
$$
\Gamma = \begin{bmatrix}
    A_{\Lambda} & 0\\
    0 & A_{\Lambda}
\end{bmatrix}, \qquad \Delta = \begin{bmatrix}
    0 & A_{\Lambda}\\
    A_{\Lambda} & 0
\end{bmatrix},
$$
where $A_{\Lambda}$ is the Coxeter adjacency matrix of $\Lambda$, so it is clear that twists are admissible.
An example of the twist $D_4\times D_4$ is given in Figure \ref{fig:twist}.

$\boldsymbol{E_6\ast E_6}$.
This exceptional bigraph has two red copies of $E_6$, which are connected to each other with blue edges as shown in Figure \ref{fig:E6E6}.

$\boldsymbol{D_5\boxtimes A_7}$.
This exceptional bigraph has one red copy of $D_5$ and one red copy of $A_7$, which are connected to each other with blue edges as shown in Figure \ref{fig:BCBCexceptional}. This family is self-dual.

$\boldsymbol{A_{2n-1}\ast D_{n+1}}$.
Each bigraph in this family has two red components, $A_{2n-1}$ and $D_{n+1}$, which are connected with blue edges as depicted in Figure \ref{fig:AnDn}.
In particular, let $A_{2n-1}$ be labeled with vertices $v_1, \dots, v_{2n-1}$, and let $D_{n+1}$ be labeled with vertices $u_1, \dots, u_{n}, u_{n+1}$ as shown in Figure \ref{fig:AnDn}.
Then, $u_i$ is connected to $v_i$ and $v_{2n-i}$ with a blue edge for all $i\in \{1, 2, \dots, n-1\}$, while $u_n$ and $u_{n+1}$ are connected to $v_n$ with blue edges.
This is a specific instance of the infinite families listed below.

\textbf{The} $\boldsymbol{(A^{m-1}D)_n}$ \textbf{family}.
Each bigraph in this family is obtained by taking $A_{2n-1}\otimes A_{m}$ and then attaching the final (red) copy of $A_{2n-1}$ to an additional red component $D_{n+1}$ with blue edges, $A_{2n-1}\ast D_{n+1}$. 
This family is self-dual; when one swaps the blue and red components, one obtains $(A^{n-1}D)_m$.
In \cite{stembridge}, this is denoted as
\begin{align*}
    (A^{m-1}D)_n &:= (A_{2n-1}\equiv A_{2n-1}\equiv \cdots \equiv A_{2n-1}\ast D_{n+1}).
\end{align*}

\textbf{The} $\boldsymbol{(AD^{m-1})_n}$ \textbf{family}.
Each bigraph in this family is obtained by taking $D_{n+1}\otimes A_{m-1}$ and then attaching the first (red) copy of $D_{n+1}$ to an additional red component $A_{2n-1}$ with blue edges, $A_{2n-1}\ast D_{n+1}$. 
This family is self-dual; when one swaps the blue and red components, one obtains $(AD^{n-1})_m$.
In \cite{stembridge}, this is denoted as
\begin{align*}
    (AD^{m-1})_n &:= (A_{2n-1}\ast D_{n+1}\equiv D_{n+1}\equiv \cdots \equiv D_{n+1}).
\end{align*}


In \cite{pasha}, Galashin--Pylyavskyy proved that the Zamolodchikov periodic quivers are in natural bijection with admissible ADE bigraphs.
For the classification of Zamolodchikov periodic cluster algebras, we generalize the idea of an ADE bigraph to include all Dynkin diagrams.

\subsection{Dynkin biagrams}

A \textit{Dynkin biagram} is a tuple $(\Gamma, \Delta)$ of two (not necessarily irreducible) $n\times n$ Coxeter adjacency matrices such that their sum $B = \Gamma + \Delta$ is bipartite and $\Gamma, \Delta$ share no nonzero entries. 
The \textit{dual} of a Dynkin biagram $(\Gamma, \Delta)$, denoted $(\Gamma, \Delta)^*$, is defined to be $(\Delta, \Gamma)$. 
Finally, a Dynkin biagram $(\Gamma, \Delta)$ is \textit{admissible} if the associated Cartan matrices $C_{\Gamma} = 2I - \Gamma$ and $C_{\Delta} = 2I - \Delta$ commute, or equivalently if $\Gamma$ and $\Delta$ commute.

In this paper, we depict Dynkin biagrams as a union of red Dynkin diagram components associated to $\Gamma$ connected by blue Dynkin diagram components associated to $\Delta$.
An example of admissible and nonadmissible Dynkin biagrams is shown in Figure \ref{fig:admissible}.
\begin{remark}
    The definitions of tensor products and twists extend naturally to include all Dynkin diagrams $\Lambda$, not just ADE types. 
    Because the definition is identical, it is still clear that both tensor products and twists are admissible Dynkin biagrams.
\end{remark}
\begin{remark}
    We can also visualize a Dynkin biagram as a directed graph with red and blue edges, where $\Gamma$ corresponds to red edges and $\Delta$ corresponds to blue edges. 
    We say an edge of a Dynkin diagram is \textit{simple} if it is undirected, and \textit{nonsimple} if it is a directed multi-edge.
    Every ADE Dynkin diagram is made up of simple edges, and every non-ADE Dynkin diagram has a nonsimple edge.
    
    Each simple edge $(u,v)$ can be replaced by two directed edges $(u, v), (v, u)$. A nonsimple edge with corresponding matrix entries $b_{uv}\neq b_{vu}$ can be replaced with $b_{uv}$ directed edges $(u, v)$ and $b_{vu}$ directed edges $(v, u)$. 
    Then, a Dynkin biagram is an arrangement of these edges on a shared bipartite vertex set $[n]$ such that $\Gamma$ and $\Delta$ share no edges.

    Under this characterization, $(\Gamma\Delta)_{ij}$ is the number of red-blue paths from $i$ to $j$.
    Similarly, $(\Delta\Gamma)_{ij}$ is the number of blue-red paths from $i$ to $j$.
    Thus a Dynkin biagram $(\Gamma, \Delta)$ is admissible if the number of red-blue paths from $i$ to $j$ is the same as the number of blue-red paths from $i$ to $j$ for every pair of vertices $i,j\in [n]$.
\end{remark}
\begin{figure}
\scalebox{0.5}{
\begin{tikzpicture}

\coordinate (v0x0) at (0.00,0.00);
\coordinate (v0x1) at (0.00,2.00);
\coordinate (v0x2) at (-1.00,4.00);
\coordinate (v0x3) at (2.00,3.00);
\coordinate (v0x4) at (5.00,0.00);
\coordinate (v0x5) at (5.00,2.00);
\coordinate (v0x6) at (5.00,4.00);

\coordinate (v0x7) at (13.00,0.00);
\coordinate (v0x8) at (13.00,2.00);
\coordinate (v0x9) at (13.00,4.00);
\coordinate (v0x11) at (17.00,0.50);
\coordinate (v0x12) at (17.00,2.50);
\coordinate (v0x13) at (19.00,3.50);
\coordinate (v0x14) at (20.50,1.50);
\coordinate (v0x15) at (20.50,-0.50);

\draw[color=red, line width=0.75mm] (5.0, 3.11) to[] (4.78, 2.89) ;
\draw[color=red, line width=0.75mm] (5.0, 3.11) to[] (5.22, 2.89) ;
\draw[color=red, line width=0.75mm] (4.93, 2.0) to[] (4.93, 4.0) ;
\draw[color=red, line width=0.75mm] (5.07, 2.0) to[] (5.07, 4.0) ;
\draw[color=red,line width=0.75mm] (v0x1) to[] (v0x0);
\draw[color=red,line width=0.75mm] (v0x2) to[] (v0x1);
\draw[color=red,line width=0.75mm] (v0x3) to[] (v0x1);
\draw[color=blue,line width=0.75mm] (v0x4) to[] (v0x0);
\draw[color=red,line width=0.75mm] (v0x5) to[] (v0x4);
\draw[color=blue,line width=0.75mm] (v0x5) to[] (v0x1);
\draw[color=blue,line width=0.75mm] (v0x6) to[] (v0x2);
\draw[color=blue,line width=0.75mm] (v0x6) to[] (v0x3);

\draw[color=red, line width=0.75mm] (13.0, 3.11) to[] (12.78, 2.89) ;
\draw[color=red, line width=0.75mm] (13.0, 3.11) to[] (13.22, 2.89) ;
\draw[color=red, line width=0.75mm] (12.93, 2.0) to[] (12.93, 4.0) ;
\draw[color=red, line width=0.75mm] (13.07, 2.0) to[] (13.07, 4.0) ;
\draw[color=blue, line width=0.75mm] (15.890379966593063, 3.7591350027839114) to[] (16.091350027839116, 3.521624930402215) ;
\draw[color=blue, line width=0.75mm] (15.890379966593063, 3.7591350027839114) to[] (16.12789003897476, 3.960105064029962) ;
\draw[color=blue, line width=0.75mm] (18.994186816410238, 3.4302417969228585) to[] (12.994186816410238, 3.9302417969228585) ;
\draw[color=blue, line width=0.75mm] (19.005813183589762, 3.5697582030771415) to[] (13.005813183589762, 4.069758203077142) ;
\draw[color=red,line width=0.75mm] (v0x8) to[] (v0x7);
\draw[color=blue,line width=0.75mm] (v0x11) to[] (v0x7);
\draw[color=blue,line width=0.75mm] (v0x15) to[] (v0x7);
\draw[color=red,line width=0.75mm] (v0x12) to[] (v0x11);
\draw[color=blue,line width=0.75mm] (v0x12) to[] (v0x8);
\draw[color=blue,line width=0.75mm] (v0x14) to[] (v0x8);
\draw[color=red,line width=0.75mm] (v0x13) to[] (v0x12);
\draw[color=red,line width=0.75mm] (v0x13) to[] (v0x14);
\draw[color=red,line width=0.75mm] (v0x14) to[] (v0x15);

\draw[fill=white] (v0x0.center) circle (0.2);
\draw[fill=black!75!white] (v0x1.center) circle (0.2);
\draw (-1.25, 2.25) node [anchor=north west][inner sep=0.75pt] {\LARGE{$v_1$}};
\draw[fill=white] (v0x2.center) circle (0.2);
\draw[fill=white] (v0x3.center) circle (0.2);
\draw[fill=black!75!white] (v0x4.center) circle (0.2);
\draw[fill=white] (v0x5.center) circle (0.2);
\draw[fill=black!75!white] (v0x6.center) circle (0.2);
\draw (5.50, 4.25) node [anchor=north west][inner sep=0.75pt] {\LARGE{$v_2$}};
\draw[fill=white] (v0x7.center) circle (0.2);
\draw[fill=black!75!white] (v0x8.center) circle (0.2);
\draw[fill=white] (v0x9.center) circle (0.2);
\draw[fill=black!75!white] (v0x11.center) circle (0.2);
\draw[fill=white] (v0x12.center) circle (0.2);
\draw[fill=black!75!white] (v0x13.center) circle (0.2);
\draw[fill=white] (v0x14.center) circle (0.2);
\draw[fill=black!75!white] (v0x15.center) circle (0.2);
\end{tikzpicture}}
\caption{\label{fig:admissible} A nonadmissible Dynkin biagram (left) and an admissible Dynkin biagram (right). The vertices $v_1$ and $v_2$ form a nonadmissible pair.}
\end{figure}

Given a Dynkin biagram $(\Gamma, \Delta)$ on $n$ vertices with bipartition $\epsilon: [n]\rightarrow \{\circ, \bullet\}$, there is an associated bipartite $B$-matrix $B = \tilde{\Gamma} + \tilde{\Delta}$ defined as follows:
\begin{align*}
    \tilde{\Gamma}_{ij} = \begin{cases}
        \Gamma_{ij} & \text{if }\epsilon_i = \circ, \epsilon_j = \bullet,\\
        -\Gamma_{ij} & \text{if }\epsilon_i = \bullet, \epsilon_j = \circ,\\
        0 & \text{otherwise.}
    \end{cases}\qquad 
    \tilde{\Delta}_{ij} = \begin{cases}
        \Delta_{ij} & \text{if }\epsilon_i = \bullet, \epsilon_j = \circ,\\
        -\Delta_{ij} & \text{if }\epsilon_i = \circ, \epsilon_j = \bullet,\\
        0 & \text{otherwise.}
    \end{cases}
\end{align*}
Notice this is consistent with directing red edges $\circ\rightarrow \bullet$, and blue edges $\bullet\rightarrow \circ$.

\begin{example}
Let us find the skew-symmetrizable matrices $\tilde{\Gamma}, \tilde{\Delta}$ associated with the admissible Dynkin biagram $B_3\bowtie_1 G_2$, defined in Section \ref{bindings} and depicted in Figure \ref{fig:BGCG}.
This Dynkin biagram has two red components, $B_3$ and $G_2$, and two blue components, $B_3$ and $G_2$.
The set of vertices is $\{1, 2, 3, 4, 5\}$, where $\epsilon_1 = \epsilon_3 = \epsilon_4 = \bullet$ and $\epsilon_2 = \epsilon_5 = \circ$.
The associated Coxeter adjacency matrices are
\begin{align*}
    \Gamma = \begin{bmatrix}
        0 & 1 & 0 & 0 & 0\\
        1 & 0 & 1 & 0 & 0\\
        0 & 2 & 0 & 0 & 0\\
        0 & 0 & 0 & 0 & 3\\
        0 & 0 & 0 & 1 & 0
    \end{bmatrix}\text{ and } \Delta = \begin{bmatrix}
        0 & 0 & 0 & 0 & 1\\
        0 & 0 & 0 & 1 & 0\\
        0 & 0 & 0 & 0 & 2\\
        0 & 3 & 0 & 0 & 0\\
        1 & 0 & 1 & 0 & 0
    \end{bmatrix}.
\end{align*}
From the bipartite coloring, we see the associated signed Coxeter adjacency matrices are
\begin{align*}
    \tilde{\Gamma} = \begin{bmatrix}
        0 & -1 & 0 & 0 & 0\\
        1 & 0 & 1 & 0 & 0\\
        0 & -2 & 0 & 0 & 0\\
        0 & 0 & 0 & 0 & -3\\
        0 & 0 & 0 & 1 & 0
    \end{bmatrix}\text{ and } \tilde{\Delta} = \begin{bmatrix}
        0 & 0 & 0 & 0 & 1\\
        0 & 0 & 0 & -1 & 0\\
        0 & 0 & 0 & 0 & 2\\
        0 & 3 & 0 & 0 & 0\\
        -1 & 0 & -1 & 0 & 0
    \end{bmatrix}.
\end{align*}
Notice that these are the same matrices as those from our running example.
\end{example}

\section{Classification of commuting Cartan matrices as admissible Dynkin biagrams}
\label{classification}
\subsection{Bindings and folding}
\label{bindings}

\begin{definition}
    Let $(\Gamma, \Delta)$  be a Dynkin biagram. If $\Gamma = \Gamma_1\sqcup \Gamma_2$ and every edge of $\Delta$ connects $\Gamma_1$ to $\Gamma_2$, then $(\Gamma, \Delta)$ is a \textit{binding} of $\Gamma_1$ and $\Gamma_2$.
\end{definition}

For instance, the tensor product of $\Gamma$ with $A_2$, $B_2$, or $G_2$ is a binding of two copies of $\Gamma$. 
We use the notation $$(\Gamma\equiv \Gamma) \coloneqq \Gamma\otimes A_2,\quad (\Gamma\equiv_2 \Gamma) \coloneqq \Gamma\otimes B_2, \quad (\Gamma\equiv_3 \Gamma) \coloneqq \Gamma\otimes G_2,$$ and refer to these as \textit{parallel bindings}.

In the special case that $(\Gamma, \Delta)$ and $(\Delta, \Gamma)$ are both bindings, we refer to this as a \textit{double binding}.
For example, the twist $\Gamma\times \Gamma$ is a double binding of two copies of $\Gamma$.

\begin{definition}
    A \textit{bicolored automorphism} $f: [n] \rightarrow [n]$ is an automorphism that satisfies the following.
    \begin{itemize}
        \item[(i)]
            $\epsilon_{f(i)} = \epsilon_i$ (preserves bipartite coloring of vertices).
        \item[(ii)]
            For all $i_1, i_2$ in the same $f$-orbit $I$, and for all $j\in[n]$, $b_{i_1j}b_{i_2j} \geq 0$ (preserves edge colors).
        \item[(iii)]
            For all $i, j\in [n]$, $b_{f(i)f(j)} = b_{ij}$.
        \item[(iv)]
            For all $i, j$ in the same $f$-orbit $I$, $b_{ij} = b_{ji} = 0$.
    \end{itemize}
\end{definition}

\begin{definition}
    Let $B$ be a bipartite skew-symmetrizable matrix, and let $f$ be a bicolored automorphism on $B$. 
    Then, one may fold $B$ to get $f(B)$, another bipartite skew-symmetrizable matrix with rows and columns indexed by $f$-orbits, defined as 
    $$f(B)_{IJ} \coloneqq \sum_{i\in I}b_{ij}, \text{ for any } j\in J.$$
    Moreover, $f(B)$ splits naturally into $f(B) = f(\Gamma) + f(\Delta)$, where 
    $$f(\Gamma)_{IJ} \coloneqq \sum_{i\in I}\Gamma_{ij} \text{ and } f(\Delta)_{IJ} \coloneqq \sum_{i\in I}\Delta_{ij}, \text{ for some fixed } j\in J.$$
    In particular, the result does not depend on $j\in J$.
\end{definition}

With a bicolored automorphism, one can fold bindings to obtain other bindings.
For instance, there exists a bicolored automorphism that folds the exceptional binding $D_5\boxtimes A_7$ into an exceptional binding of $B_4$ and $C_4$, denoted $B_4\boxtimes C_4$ (see Figure \ref{fig:BCBCexceptional}).

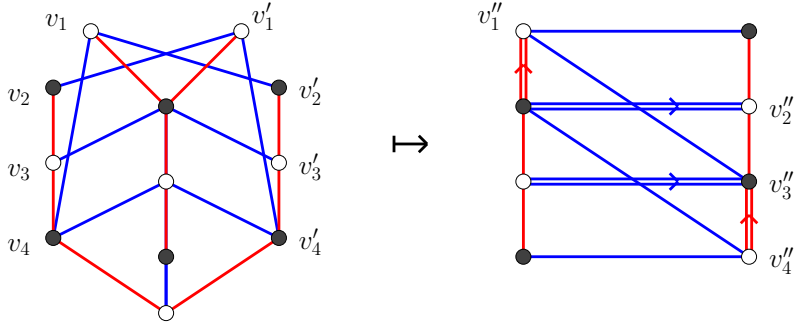
\begin{figure}[h]
\scalebox{0.5}{
\begin{tikzpicture}
\coordinate (u0x0) at (0.00, 0.00);
\coordinate (u0x1) at (0.00, 2.00);
\coordinate (u0x2) at (0.00, 4.00);
\coordinate (u0x3) at (-1.00, 6.00);
\coordinate (u0x4) at (2.00, 5.00);

\coordinate (u1x0) at (4.00, 6.00);
\coordinate (u1x1) at (4.00, 4.00);
\coordinate (u1x2) at (4.00, 2.00);
\coordinate (u1x3) at (5.50, 0.00);
\coordinate (u1x4) at (7.50, 1.00);
\coordinate (u1x5) at (7.50, 3.00);
\coordinate (u1x6) at (7.50, 5.00);

\draw (-1.75, 6.00) node {\LARGE{$v_1$}};
\draw (1.5, 5.50) node {\LARGE{$v_1'$}};

\draw (4.75, 6.00) node {\LARGE{$v_2$}};
\draw (8.25, 5.00) node {\LARGE{$v_2'$}};
\draw (4.75, 4.00) node {\LARGE{$v_3$}};
\draw (8.25, 3.00) node {\LARGE{$v_3'$}};
\draw (4.75, 2.00) node {\LARGE{$v_4$}};
\draw (8.25, 1.00) node {\LARGE{$v_4'$}};

\draw (10.75, 6.00) node [anchor=south west][inner sep=0.75pt] {\LARGE{$v_1''$}};
\draw (18.50, 5.50) node [anchor=south west][inner sep=0.75pt] {\LARGE{$v_2''$}};
\draw (18.50, 3.50) node [anchor=south west][inner sep=0.75pt] {\LARGE{$v_3''$}};
\draw (18.50, 1.50) node [anchor=south west][inner sep=0.75pt] {\LARGE{$v_4''$}};

\draw[color=blue,line width=0.75mm] (u0x0) to[] (u1x3);
\draw[color=blue,line width=0.75mm] (u0x1) to[] (u1x2);
\draw[color=blue,line width=0.75mm] (u0x1) to[] (u1x4);
\draw[color=blue,line width=0.75mm] (u0x2) to[] (u1x1);
\draw[color=blue,line width=0.75mm] (u0x2) to[] (u1x5);
\draw[color=blue,line width=0.75mm] (u0x2) to[] (u1x3);
\draw[color=blue,line width=0.75mm] (u0x3) to[] (u1x0);
\draw[color=blue,line width=0.75mm] (u0x3) to[] (u1x4);
\draw[color=blue,line width=0.75mm] (u0x4) to[] (u1x6);
\draw[color=blue,line width=0.75mm] (u0x4) to[] (u1x2);

\foreach \i in {0, 1, 2, 3, 4, 5}{
    \pgfmathtruncatemacro{\j}{\i + 1};
    \draw[color=red,line width=0.75mm] (u1x\i) to[] (u1x\j);
}
\draw[color=red,line width=0.75mm] (u0x0) to[] (u0x1);
\draw[color=red,line width=0.75mm] (u0x1) to[] (u0x2);
\draw[color=red,line width=0.75mm] (u0x2) to[] (u0x3);
\draw[color=red,line width=0.75mm] (u0x2) to[] (u0x4);
\foreach \i in {0, 2, 4, 6}{
    \draw[fill=black!75!white] (u1x\i .center) circle (0.2);
}
\foreach \i in {1, 3, 5}{
    \draw[fill=white] (u1x\i .center) circle (0.2);
}
\draw[fill=black!75!white] (u0x0.center) circle (0.2);
\draw[fill=white] (u0x1.center) circle (0.2);
\draw[fill=black!75!white] (u0x2.center) circle (0.2);
\draw[fill=white] (u0x3.center) circle (0.2);
\draw[fill=white] (u0x4.center) circle (0.2);

\foreach \i in {0, 1}
{
\coordinate (v\i x0) at (12 + 6 * \i, 0.00);
\coordinate (v\i x1) at (12 + 6 * \i, 2.00);
\coordinate (v\i x2) at (12 + 6 * \i, 4.00);
\coordinate (v\i x3) at (12 + 6 * \i, 6.00);
\draw[color=red,line width=0.75mm] (v\i x1) to[] (v\i x2);
}
\draw[color=red,line width=0.75mm] (v0x0) to[] (v0x1);
\draw[color=red,line width=0.75mm] (v1x2) to[] (v1x3);
\draw[color=red, line width=0.75mm] (12 + 0.0, 5.11) to[] (12 + -0.22, 4.89) ;
\draw[color=red, line width=0.75mm] (12 + 0.0, 5.11) to[] (12 + 0.22, 4.89) ;
\draw[color=red, line width=0.75mm] (12 + -0.07, 4.0) to[] (12 + -0.07, 6.0) ;
\draw[color=red, line width=0.75mm] (12 + 0.07, 4.0) to[] (12 + 0.07, 6.0) ;
\draw[color=red, line width=0.75mm] (18 + 0.0, -4 + 5.11) to[] (18 + -0.22, -4 + 4.89) ;
\draw[color=red, line width=0.75mm] (18 + 0.0, -4 + 5.11) to[] (18 + 0.22, -4 + 4.89) ;
\draw[color=red, line width=0.75mm] (18 + -0.07, -4 + 4.0) to[] (18 + -0.07, -4 + 6.0) ;
\draw[color=red, line width=0.75mm] (18 + 0.07, -4 + 4.0) to[] (18 + 0.07, -4 + 6.0) ;

\draw[color=blue,line width=0.75mm] (v0x0) to[] (v1x0);
\draw[color=blue,line width=0.75mm] (v0x2) to[] (v1x0);
\draw[color=blue,line width=0.75mm] (v0x3) to[] (v1x1);
\draw[color=blue,line width=0.75mm] (v0x3) to[] (v1x3);
\draw[color=blue, line width=0.75mm] (12 + 4.11, 2.0) to[] (12 + 3.89, 2.22) ;
\draw[color=blue, line width=0.75mm] (12 + 4.11, 2.0) to[] (12 + 3.89, 1.78) ;
\draw[color=blue, line width=0.75mm] (12 + 0.0, 2.07) to[] (12 + 6.0, 2.07) ;
\draw[color=blue, line width=0.75mm] (12 + 0.0, 1.93) to[] (12 + 6.0, 1.93) ;
\draw[color=blue, line width=0.75mm] (12 + 4.11, 4.0) to[] (12 + 3.89, 4.22) ;
\draw[color=blue, line width=0.75mm] (12 + 4.11, 4.0) to[] (12 + 3.89, 3.78) ;
\draw[color=blue, line width=0.75mm] (12 + 0.0, 4.07) to[] (12 + 6.0, 4.07) ;
\draw[color=blue, line width=0.75mm] (12 + 0.0, 3.93) to[] (12 + 6.0, 3.93) ;

\draw[fill=black!75!white] (v0x0.center) circle (0.2);
\draw[fill=white] (v0x1.center) circle (0.2);
\draw[fill=black!75!white] (v0x2.center) circle (0.2);
\draw[fill=white] (v0x3.center) circle (0.2);
\draw[fill=white] (v1x0.center) circle (0.2);
\draw[fill=black!75!white] (v1x1.center) circle (0.2);
\draw[fill=white] (v1x2.center) circle (0.2);
\draw[fill=black!75!white] (v1x3.center) circle (0.2);

\draw[color=black, line width=0.75mm] (8.00 + 1.56, 3.00 + -0.22) to[] (8.00 + 1.56, 3.00 + 0.22) ;
\draw[color=black, line width=0.75mm] (8.00 + 2.44, 3.00 + 0.0) to[] (8.00 + 1.56, 3.00 + 0.00) ;
\draw[color=black, line width=0.75mm] (8.00 + 2.44, 3.00 + 0.0) to[] (8.00 + 2.22, 3.00 + 0.22) ;
\draw[color=black, line width=0.75mm] (8.00 + 2.44, 3.00 + 0.0) to[] (8.00 + 2.22, 3.00 + -0.22) ;

\end{tikzpicture}}
\caption{\label{fig:BCBCexceptional} The folding of exceptional type $D_5\boxtimes A_7$ to $B_4\boxtimes C_4$. The vertices $v_i$ and $v_i'$ form the $f$-orbit $v_i''$ for each $i\in \{1, 2, 3, 4\}$. This folds $D_5$ to $B_4$ and $A_7$ to $C_4$.}
\end{figure}

In addition, with a valid bicolored automorphism that folds Dynkin diagram $\Gamma$ to Dynkin diagram $\Lambda$, one can fold the twist $\Gamma\times \Gamma$ into another binding.
We use the notation $\Lambda \ltimes \Gamma$ and $\Gamma\rtimes \Lambda$ to denote a binding of $\Lambda$ and $\Gamma$ obtained by folding one side of a twist, and $\Lambda_1\bowtie \Lambda_2$ to indicate a binding of $\Lambda_1$ and $\Lambda_2$ obtained by folding both sides of the twist $\Gamma\times\Gamma$.
An example of a sequence of folds that lead to a binding of the form $\Lambda_1\bowtie \Lambda_2$ is depicted in Figure \ref{fig:folding}.
All bindings obtained by folding twists are depicted in Figures \ref{fig:BDCD}, \ref{fig:BGCG}, \ref{fig:BC}, and \ref{fig:GD}.

\begin{remark}
    In general, $C_n$ is only defined for $n\geq 3$.
    However in this paper, we sometimes use $C_2$ to denote the opposite orientation of $B_2$.
    This notation is useful in distinguishing between the bindings $B_n\ltimes D_{n+1}$ and $C_n\ltimes D_{n+1}$ when $n = 2$ (see Figure \ref{fig:BDCD}).
\end{remark}

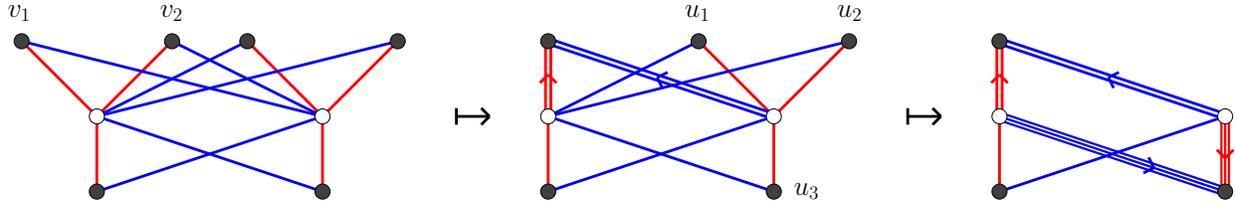
\begin{figure}[h]
\scalebox{0.5}{
\begin{tikzpicture}
\draw[color=black, line width=0.75mm] (9.00 + 0.56, 2.00 + -0.22) to[] (9.00 + 0.56, 2.00 + 0.22) ;
\draw[color=black, line width=0.75mm] (9.00 + 1.44, 2.00 + 0.0) to[] (9.00 + 0.56, 2.00 + 0.00) ;
\draw[color=black, line width=0.75mm] (9.00 + 1.44, 2.00 + 0.0) to[] (9.00 + 1.22, 2.00 + 0.22) ;
\draw[color=black, line width=0.75mm] (9.00 + 1.44, 2.00 + 0.0) to[] (9.00 + 1.22, 2.00 + -0.22) ;
\draw[color=black, line width=0.75mm] (21.00 + 0.56, 2.00 + -0.22) to[] (21.00 + 0.56, 2.00 + 0.22) ;
\draw[color=black, line width=0.75mm] (21.00 + 1.44, 2.00 + 0.0) to[] (21.00 + 0.56, 2.00 + 0.00) ;
\draw[color=black, line width=0.75mm] (21.00 + 1.44, 2.00 + 0.0) to[] (21.00 + 1.22, 2.00 + 0.22) ;
\draw[color=black, line width=0.75mm] (21.00 + 1.44, 2.00 + 0.0) to[] (21.00 + 1.22, 2.00 + -0.22) ;
\foreach \i in {0, 1}
{
    \coordinate (v\i x0) at (6 * \i + 0.00,0.00);
    \coordinate (v\i x1) at (6 * \i + 0.00,2.00);
    \coordinate (v\i x2) at (6 * \i + -2.00,4.00);
    \coordinate (v\i x3) at (6 * \i + 2.00,4.00);
    \draw[color=red,line width=0.75mm] (v\i x1) to[] (v\i x0);
    \draw[color=red,line width=0.75mm] (v\i x2) to[] (v\i x1);
    \draw[color=red,line width=0.75mm] (v\i x3) to[] (v\i x1);
}
\draw[color=blue,line width=0.75mm] (v0x1) to[] (v1x0);
\draw[color=blue,line width=0.75mm] (v0x0) to[] (v1x1);
\draw[color=blue,line width=0.75mm] (v0x1) to[] (v1x2);
\draw[color=blue,line width=0.75mm] (v0x1) to[] (v1x3);
\draw[color=blue,line width=0.75mm] (v0x2) to[] (v1x1);
\draw[color=blue,line width=0.75mm] (v0x3) to[] (v1x1);
\foreach \i in {0, 1}
{
    \draw[fill=black!75!white] (v\i x0.center) circle (0.2);
    \draw[fill=white] (v\i x1.center) circle (0.2);
}
\draw[fill=black!75!white] (v0x2.center) circle (0.2);
\draw (-2.40, 5.00) node [anchor=north west][inner sep=0.75pt] {\LARGE{$v_1$}};
\draw[fill=black!75!white] (v0x3.center) circle (0.2);
\draw (1.65, 5.00) node [anchor=north west][inner sep=0.75pt] {\LARGE{$v_2$}};
\draw[fill=black!75!white] (v1x2.center) circle (0.2);
\draw[fill=black!75!white] (v1x3.center) circle (0.2);

\foreach \i in {1, 2}
{
    \draw[color=red, line width=0.75mm] (12 * \i + 0.0, 3.11) to[] (12 * \i + -0.22, 2.89) ;
    \draw[color=red, line width=0.75mm] (12 * \i + 0.0, 3.11) to[] (12 * \i + 0.22, 2.89) ;
    \draw[color=blue, line width=0.75mm] (12 * \i + 2.8956448372144434, 3.034785054261852) to[] (12 * \i + 3.034785054261852, 2.7565046201670347) ;
    \draw[color=blue, line width=0.75mm] (12 * \i + 2.8956448372144434, 3.034785054261852) to[] (12 * \i + 3.173925271309261, 3.173925271309261) ;
}
\foreach \i in {1}
{
    \coordinate (v\i x0) at (12 * \i,0.00);
    \coordinate (v\i x1) at (12 * \i,2.00);
    \coordinate (v\i x2) at (12 * \i,4.00);
    \coordinate (v\i x3) at (12 * \i + 6.00,0.00);
    \coordinate (v\i x4) at (12 * \i + 6.00,2.00);
    \coordinate (v\i x5) at (12 * \i + 4.00,4.00);
    \coordinate (v\i x6) at (12 * \i + 8.00,4.00);
    \draw[color=red,line width=0.75mm] (v\i x0) to[] (v\i x1);
    \draw[color=red,line width=0.75mm] (v\i x3) to[] (v\i x4);
    \draw[color=red,line width=0.75mm] (v\i x4) to[] (v\i x5);
    \draw[color=red,line width=0.75mm] (v\i x4) to[] (v\i x6);
    \draw[color=blue,line width=0.75mm] (v\i x0) to[] (v\i x4);
    \draw[color=blue,line width=0.75mm] (v\i x1) to[] (v\i x3);
    \draw[color=blue,line width=0.75mm] (v\i x1) to[] (v\i x5);
    \draw[color=blue,line width=0.75mm] (v\i x1) to[] (v\i x6);
    \draw[color=red, line width=0.75mm] (12 * \i + -0.07, 2.0) to[] (12 * \i + -0.07, 4.0) ;
    \draw[color=red, line width=0.75mm] (12 * \i + 0.07, 2.0) to[] (12 * \i + 0.07, 4.0) ;
    \draw[color=blue, line width=0.75mm] (12 * \i + 5.977864056378821, 1.933592169136464) to[] (12 * \i + -0.02213594362117866, 3.933592169136464) ;
    \draw[color=blue, line width=0.75mm] (12 * \i + 6.022135943621179, 2.066407830863536) to[] (12 * \i + 0.02213594362117866, 4.066407830863536) ;
    \draw[fill=black!75!white] (v\i x0.center) circle (0.2);
    \draw[fill=white] (v\i x1.center) circle (0.2);
    \draw[fill=black!75!white] (v\i x2.center) circle (0.2);
    \draw[fill=black!75!white] (v\i x3.center) circle (0.2);
    \draw (15.60, 5.00) node [anchor=north west][inner sep=0.75pt] {\LARGE{$u_1$}};
    \draw[fill=white] (v\i x4.center) circle (0.2);
    \draw[fill=black!75!white] (v\i x5.center) circle (0.2);
    \draw (19.65, 5.00) node [anchor=north west][inner sep=0.75pt] {\LARGE{$u_2$}};
    \draw[fill=black!75!white] (v\i x6.center) circle (0.2);
    \draw (18.5, 0.25) node [anchor=north west][inner sep=0.75pt] {\LARGE{$u_3$}};
}

\foreach \i in {2}
{
    \coordinate (v\i x0) at (12 * \i,0.00);
    \coordinate (v\i x1) at (12 * \i,2.00);
    \coordinate (v\i x2) at (12 * \i,4.00);
    \coordinate (v\i x3) at (12 * \i + 6.00,0.00);
    \coordinate (v\i x4) at (12 * \i + 6.00,2.00);
    \draw[color=red,line width=0.75mm] (v\i x0) to[] (v\i x1);
    \draw[color=blue,line width=0.75mm] (v\i x0) to[] (v\i x4);
    \draw[color=blue, line width=0.75mm] (12 * \i + 4.104355162785557, 0.6318816124048144) to[] (12 * \i + 3.965214945738148, 0.9101620464996318) ;
    \draw[color=blue, line width=0.75mm] (12 * \i + 4.104355162785557, 0.6318816124048144) to[] (12 * \i + 3.826074728690739, 0.4927413953574058) ;
    \draw[color=blue, line width=0.6mm] (12 * \i, 2.0) to[] (12 * \i + 6.0, 0.0) ;
    \draw[color=blue, line width=0.6mm] (12 * \i + 0.0316227766016838, 2.094868329805051) to[] (12 * \i + 6.031622776601684, 0.09486832980505139) ;
    \draw[color=blue, line width=0.6mm] (12 * \i + -0.0316227766016838, 1.9051316701949486) to[] (12 * \i + 5.968377223398316, -0.09486832980505139) ;
    \draw[color=red, line width=0.75mm] (12 * \i + -0.07, 2.0) to[] (12 * \i + -0.07, 4.0) ;
    \draw[color=red, line width=0.75mm] (12 * \i + 0.07, 2.0) to[] (12 * \i + 0.07, 4.0) ;
    \draw[color=red, line width=0.75mm] (12 * \i + 6.0, 0.89) to[] (12 * \i + 6.22, 1.11) ;
    \draw[color=red, line width=0.75mm] (12 * \i + 6.0, 0.89) to[] (12 * \i + 5.78, 1.11) ;
    \draw[color=red, line width=0.6mm] (12 * \i + 6.0, 2.0) to[] (12 * \i + 6.0, 0.0) ;
    \draw[color=red, line width=0.6mm] (12 * \i + 6.1, 2.0) to[] (12 * \i + 6.1, 0.0) ;
    \draw[color=red, line width=0.6mm] (12 * \i + 5.9, 2.0) to[] (12 * \i + 5.9, 0.0) ;
    \draw[color=blue, line width=0.75mm] (12 * \i + 5.977864056378821, 1.933592169136464) to[] (12 * \i + -0.02213594362117866, 3.933592169136464) ;
    \draw[color=blue, line width=0.75mm] (12 * \i + 6.022135943621179, 2.066407830863536) to[] (12 * \i + 0.02213594362117866, 4.066407830863536) ;
    \draw[fill=black!75!white] (v\i x0.center) circle (0.2);
    \draw[fill=white] (v\i x1.center) circle (0.2);
    \draw[fill=black!75!white] (v\i x2.center) circle (0.2);
    \draw[fill=black!75!white] (v\i x3.center) circle (0.2);
    \draw[fill=white] (v\i x4.center) circle (0.2);
}
\end{tikzpicture}}
\caption{\label{fig:folding} A sequence of folds from $D_4\times D_4$ to $B_3\ltimes D_4$ to $B_3\bowtie_1 G_2$. The vertices $v_1, v_2$ and $u_1, u_2, u_3$ indicate the nontrivial orbits in each bicolored automorphism.}
\end{figure}

\begin{figure}[h]
        \scalebox{0.5}{
        \begin{tikzpicture}
        \draw[color=red, line width=0.75mm] (0.0, 3.11) to[] (-0.22, 2.89) ;
        \draw[color=red, line width=0.75mm] (0.0, 3.11) to[] (0.22, 2.89) ;
        \draw[color=red, line width=0.75mm] (12.0, 2.89) to[] (12.22, 3.11) ;
        \draw[color=red, line width=0.75mm] (12.0, 2.89) to[] (11.78, 3.11) ;
        \draw[color=blue, line width=0.75mm] (2.8956448372144434, 3.034785054261852) to[] (3.034785054261852, 2.7565046201670347) ;
        \draw[color=blue, line width=0.75mm] (2.8956448372144434, 3.034785054261852) to[] (3.173925271309261, 3.173925271309261) ;
        \draw[color=blue, line width=0.75mm] (15.104355162785556, 2.965214945738148) to[] (14.965214945738149, 3.2434953798329653) ;
        \draw[color=blue, line width=0.75mm] (15.104355162785556, 2.965214945738148) to[] (14.82607472869074, 2.826074728690739) ;
        \foreach \i in {0, 1}
        {
            \coordinate (v\i x0) at (12 * \i,0.00);
            \coordinate (v\i x1) at (12 * \i,2.00);
            \coordinate (v\i x2) at (12 * \i,4.00);
            \coordinate (v\i x3) at (12 * \i + 6.00,0.00);
            \coordinate (v\i x4) at (12 * \i + 6.00,2.00);
            \coordinate (v\i x5) at (12 * \i + 4.00,4.00);
            \coordinate (v\i x6) at (12 * \i + 8.00,4.00);
            \draw[color=red,line width=0.75mm] (v\i x0) to[] (v\i x1);
            \draw[color=red,line width=0.75mm] (v\i x3) to[] (v\i x4);
            \draw[color=red,line width=0.75mm] (v\i x4) to[] (v\i x5);
            \draw[color=red,line width=0.75mm] (v\i x4) to[] (v\i x6);
            \draw[color=blue,line width=0.75mm] (v\i x0) to[] (v\i x4);
            \draw[color=blue,line width=0.75mm] (v\i x1) to[] (v\i x3);
            \draw[color=blue,line width=0.75mm] (v\i x1) to[] (v\i x5);
            \draw[color=blue,line width=0.75mm] (v\i x1) to[] (v\i x6);
            \draw[color=red, line width=0.75mm] (12 * \i + -0.07, 2.0) to[] (12 * \i + -0.07, 4.0) ;
            \draw[color=red, line width=0.75mm] (12 * \i + 0.07, 2.0) to[] (12 * \i + 0.07, 4.0) ;
            \draw[color=blue, line width=0.75mm] (12 * \i + 5.977864056378821, 1.933592169136464) to[] (12 * \i + -0.02213594362117866, 3.933592169136464) ;
            \draw[color=blue, line width=0.75mm] (12 * \i + 6.022135943621179, 2.066407830863536) to[] (12 * \i + 0.02213594362117866, 4.066407830863536) ;
            \draw[fill=black!75!white] (v\i x0.center) circle (0.2);
            \draw[fill=white] (v\i x1.center) circle (0.2);
            \draw[fill=black!75!white] (v\i x2.center) circle (0.2);
            \draw[fill=black!75!white] (v\i x3.center) circle (0.2);
            \draw[fill=white] (v\i x4.center) circle (0.2);
            \draw[fill=black!75!white] (v\i x5.center) circle (0.2);
            \draw[fill=black!75!white] (v\i x6.center) circle (0.2);
        }
        \end{tikzpicture}}
        \caption{\label{fig:BDCD} The double bindings $B_3\ltimes D_4$ and $C_3\ltimes D_4$.}
        \end{figure}
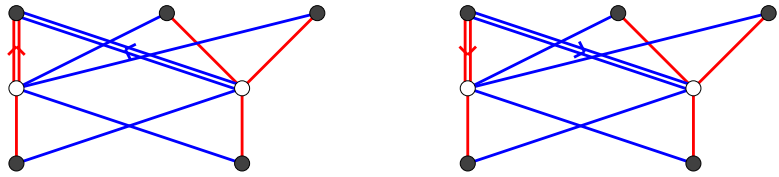

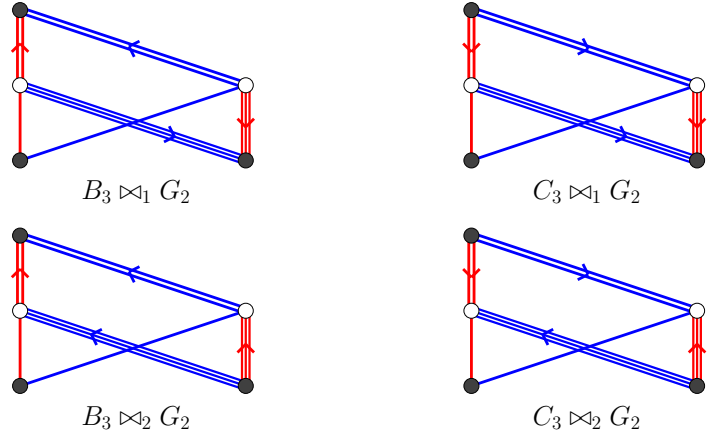
\begin{figure}[h]
        \scalebox{0.5}{
        \begin{tikzpicture}
        \foreach \i in {0, 1}
        {
            \coordinate (v\i x0) at (12 * \i,0.00);
            \coordinate (v\i x1) at (12 * \i,2.00);
            \coordinate (v\i x2) at (12 * \i,4.00);
            \coordinate (v\i x3) at (12 * \i + 6.00,0.00);
            \coordinate (v\i x4) at (12 * \i + 6.00,2.00);
            \draw[color=red,line width=0.75mm] (v\i x0) to[] (v\i x1);
            \draw[color=blue,line width=0.75mm] (v\i x0) to[] (v\i x4);
            \draw[color=blue, line width=0.75mm] (12 * \i + 4.104355162785557, 0.6318816124048144) to[] (12 * \i + 3.965214945738148, 0.9101620464996318) ;
            \draw[color=blue, line width=0.75mm] (12 * \i + 4.104355162785557, 0.6318816124048144) to[] (12 * \i + 3.826074728690739, 0.4927413953574058) ;
            \draw[color=blue, line width=0.6mm] (12 * \i, 2.0) to[] (12 * \i + 6.0, 0.0) ;
            \draw[color=blue, line width=0.6mm] (12 * \i + 0.0316227766016838, 2.094868329805051) to[] (12 * \i + 6.031622776601684, 0.09486832980505139) ;
            \draw[color=blue, line width=0.6mm] (12 * \i + -0.0316227766016838, 1.9051316701949486) to[] (12 * \i + 5.968377223398316, -0.09486832980505139) ;
            \draw[color=red, line width=0.75mm] (12 * \i + -0.07, 2.0) to[] (12 * \i + -0.07, 4.0) ;
            \draw[color=red, line width=0.75mm] (12 * \i + 0.07, 2.0) to[] (12 * \i + 0.07, 4.0) ;
            \draw[color=red, line width=0.75mm] (12 * \i + 6.0, 0.89) to[] (12 * \i + 6.22, 1.11) ;
            \draw[color=red, line width=0.75mm] (12 * \i + 6.0, 0.89) to[] (12 * \i + 5.78, 1.11) ;
            \draw[color=red, line width=0.6mm] (12 * \i + 6.0, 2.0) to[] (12 * \i + 6.0, 0.0) ;
            \draw[color=red, line width=0.6mm] (12 * \i + 6.1, 2.0) to[] (12 * \i + 6.1, 0.0) ;
            \draw[color=red, line width=0.6mm] (12 * \i + 5.9, 2.0) to[] (12 * \i + 5.9, 0.0) ;
            \draw[color=blue, line width=0.75mm] (12 * \i + 5.977864056378821, 1.933592169136464) to[] (12 * \i + -0.02213594362117866, 3.933592169136464) ;
            \draw[color=blue, line width=0.75mm] (12 * \i + 6.022135943621179, 2.066407830863536) to[] (12 * \i + 0.02213594362117866, 4.066407830863536) ;
            \draw[fill=black!75!white] (v\i x0.center) circle (0.2);
            \draw[fill=white] (v\i x1.center) circle (0.2);
            \draw[fill=black!75!white] (v\i x2.center) circle (0.2);
            \draw[fill=black!75!white] (v\i x3.center) circle (0.2);
            \draw[fill=white] (v\i x4.center) circle (0.2);
        }
        \foreach \i in {0, 1}
        {
            \draw[color=red, line width=0.75mm] (0.0, -6.00 * \i + 3.11) to[] (-0.22, -6.00 * \i + 2.89) ;
            \draw[color=red, line width=0.75mm] (0.0, -6.00 * \i + 3.11) to[] (0.22, -6.00 * \i + 2.89) ;
            \draw[color=red, line width=0.75mm] (12.0, -6.00 * \i + 2.89) to[] (12.22, -6.00 * \i + 3.11) ;
            \draw[color=red, line width=0.75mm] (12.0, -6.00 * \i + 2.89) to[] (11.78, -6.00 * \i + 3.11) ;
            \draw[color=blue, line width=0.75mm] (2.8956448372144434, -6.00 * \i + 3.034785054261852) to[] (3.034785054261852, -6.00 * \i + 2.7565046201670347) ;
            \draw[color=blue, line width=0.75mm] (2.8956448372144434, -6.00 * \i + 3.034785054261852) to[] (3.173925271309261, -6.00 * \i + 3.173925271309261) ;
            \draw[color=blue, line width=0.75mm] (15.104355162785556, -6.00 * \i + 2.965214945738148) to[] (14.965214945738149, -6.00 * \i + 3.2434953798329653) ;
            \draw[color=blue, line width=0.75mm] (15.104355162785556, -6.00 * \i + 2.965214945738148) to[] (14.82607472869074, -6.00 * \i + 2.826074728690739) ;
            \coordinate (w\i x0) at (12 * \i,-6.00);
            \coordinate (w\i x1) at (12 * \i,-6.00 + 2.00);
            \coordinate (w\i x2) at (12 * \i,-6.00 + 4.00);
            \coordinate (w\i x3) at (12 * \i + 6.00,-6.00);
            \coordinate (w\i x4) at (12 * \i + 6.00,-6.00 + 2.00);
            \draw[color=red,line width=0.75mm] (w\i x0) to[] (w\i x1);
            \draw[color=blue,line width=0.75mm] (w\i x0) to[] (w\i x4);
            \draw[color=blue, line width=0.75mm] (12 * \i + 1.8956448372144434, -6.00 + 1.3681183875951854) to[] (12 * \i + 2.034785054261852, -6.00 + 1.089837953500368) ;
            \draw[color=blue, line width=0.75mm] (12 * \i + 1.8956448372144434, -6.00 + 1.3681183875951854) to[] (12 * \i + 2.173925271309261, -6.00 + 1.5072586046425942) ;
            \draw[color=blue, line width=0.6mm] (12 * \i, -6.00 + 2.0) to[] (12 * \i + 6.0, -6.00) ;
            \draw[color=blue, line width=0.6mm] (12 * \i + 0.0316227766016838, -6.00 + 2.094868329805051) to[] (12 * \i + 6.031622776601684, -6.00 + 0.09486832980505139) ;
            \draw[color=blue, line width=0.6mm] (12 * \i + -0.0316227766016838, -6.00 + 1.9051316701949486) to[] (12 * \i + 5.968377223398316, -6.00 + -0.09486832980505139) ;
            \draw[color=red, line width=0.75mm] (12 * \i + -0.07, -6.00 + 2.0) to[] (12 * \i + -0.07, -6.00 + 4.0) ;
            \draw[color=red, line width=0.75mm] (12 * \i + 0.07, -6.00 + 2.0) to[] (12 * \i + 0.07, -6.00 + 4.0) ;
            \draw[color=red, line width=0.75mm] (12 * \i + 6.0, -6.00 + 1.11) to[] (12 * \i + 5.78, -6.00 + 0.89) ;
            \draw[color=red, line width=0.75mm] (12 * \i + 6.0, -6.00 + 1.11) to[] (12 * \i + 6.22, -6.00 + 0.89) ;
            \draw[color=red, line width=0.6mm] (12 * \i + 6.0, -6.00 + 2.0) to[] (12 * \i + 6.0, -6.00 + 0.0) ;
            \draw[color=red, line width=0.6mm] (12 * \i + 6.1, -6.00 + 2.0) to[] (12 * \i + 6.1, -6.00 + 0.0) ;
            \draw[color=red, line width=0.6mm] (12 * \i + 5.9, -6.00 + 2.0) to[] (12 * \i + 5.9, -6.00 + 0.0) ;
            \draw[color=blue, line width=0.75mm] (12 * \i + 5.977864056378821, -6.00 + 1.933592169136464) to[] (12 * \i + -0.02213594362117866, -6.00 + 3.933592169136464) ;
            \draw[color=blue, line width=0.75mm] (12 * \i + 6.022135943621179, -6.00 + 2.066407830863536) to[] (12 * \i + 0.02213594362117866, -6.00 + 4.066407830863536) ;
            \draw[fill=black!75!white] (w\i x0.center) circle (0.2);
            \draw[fill=white] (w\i x1.center) circle (0.2);
            \draw[fill=black!75!white] (w\i x2.center) circle (0.2);
            \draw[fill=black!75!white] (w\i x3.center) circle (0.2);
            \draw[fill=white] (w\i x4.center) circle (0.2);
        }
        \draw (1.60, -0.50) node [anchor=north west][inner sep=0.75pt] {\LARGE{$B_3\bowtie_1 G_2$}};
        \draw (1.60, -6.50) node [anchor=north west][inner sep=0.75pt] {\LARGE{$B_3\bowtie_2 G_2$}};
        \draw (13.60, -0.50) node [anchor=north west][inner sep=0.75pt] {\LARGE{$C_3\bowtie_1 G_2$}};
        \draw (13.60, -6.50) node [anchor=north west][inner sep=0.75pt] {\LARGE{$C_3\bowtie_2 G_2$}};
        \end{tikzpicture}}
        \caption{\label{fig:BGCG} The double bindings $B_3\bowtie_{1,2} G_2$ and $C_3\bowtie_{1,2} G_2$.}
    \end{figure}

\begin{figure}[h]
        \scalebox{0.5}{
        \begin{tikzpicture}
        \foreach \i in {0, 1, 2, 3, 4}
        {
            \coordinate (v0x\i) at (2 * \i, 0.00);
            \coordinate (v1x\i) at (2 * \i, 2.00);
        }
        \draw[color=red, line width=0.75mm] (7.11, 0.0) to[] (6.89, 0.22) ;
        \draw[color=red, line width=0.75mm] (7.11, 0.0) to[] (6.89, -0.22) ;
        \draw[color=red, line width=0.75mm] (6.0, 0.07) to[] (8.0, 0.07) ;
        \draw[color=red, line width=0.75mm] (6.0, -0.07) to[] (8.0, -0.07) ;
        \draw[color=red, line width=0.75mm] (6.89, 2.0) to[] (7.11, 1.78) ;
        \draw[color=red, line width=0.75mm] (6.89, 2.0) to[] (7.11, 2.22) ;
        \draw[color=red, line width=0.75mm] (8.0, 1.93) to[] (6.0, 1.93) ;
        \draw[color=red, line width=0.75mm] (8.0, 2.07) to[] (6.0, 2.07) ;
        \draw[color=blue, line width=0.75mm] (7.42221825406948, 1.4222182540694799) to[] (7.733345237791561, 1.4222182540694797) ;
        \draw[color=blue, line width=0.75mm] (7.42221825406948, 1.4222182540694799) to[] (7.422218254069479, 1.7333452377915606) ;
        \draw[color=blue, line width=0.75mm] (8.049497474683058, 1.9505025253169417) to[] (6.049497474683059, -0.049497474683058325) ;
        \draw[color=blue, line width=0.75mm] (7.950502525316941, 2.0494974746830583) to[] (5.950502525316941, 0.049497474683058325) ;
        \draw[color=blue, line width=0.75mm] (7.57778174593052, 0.42221825406947977) to[] (7.577781745930521, 0.7333452377915607) ;
        \draw[color=blue, line width=0.75mm] (7.57778174593052, 0.42221825406947977) to[] (7.266654762208439, 0.42221825406947977) ;
        \draw[color=blue, line width=0.75mm] (6.049497474683059, 2.0494974746830583) to[] (8.049497474683058, 0.049497474683058325) ;
        \draw[color=blue, line width=0.75mm] (5.950502525316941, 1.9505025253169417) to[] (7.950502525316941, -0.049497474683058325) ;
        \foreach \i in {0, 1, 2}
        {
            \pgfmathtruncatemacro{\j}{\i + 1};
            \draw[color=red,line width=0.75mm] (v0x\i) to[] (v0x\j);
            \draw[color=red,line width=0.75mm] (v1x\i) to[] (v1x\j);
            \draw[color=blue,line width=0.75mm] (v0x\i) to[] (v1x\j);
            \draw[color=blue,line width=0.75mm] (v1x\i) to[] (v0x\j);
        }
        
        \foreach \i in {0, 2, 4}
        {
            \draw[fill=black!75!white] (v0x\i.center) circle (0.2);
            \draw[fill=black!75!white] (v1x\i.center) circle (0.2);
        }
        \foreach \i in {1, 3}
        {
            \draw[fill=white] (v1x\i.center) circle (0.2);
            \draw[fill=white] (v0x\i.center) circle (0.2);
        }
        \end{tikzpicture}}
        \caption{\label{fig:BC} The double binding $B_5\bowtie C_5$.}
        \end{figure}
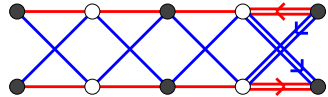
    
    \begin{figure}[h]
        \scalebox{0.5}{
        \begin{tikzpicture}
        \draw[color=red, line width=0.75mm] (0.0, 0.89) to[] (0.22, 1.11) ;
        \draw[color=red, line width=0.75mm] (0.0, 0.89) to[] (-0.22, 1.11) ;
        \draw[color=red, line width=0.75mm] (12.0, 1.11) to[] (11.78, 0.89) ;
        \draw[color=red, line width=0.75mm] (12.0, 1.11) to[] (12.22, 0.89) ;
        \draw[color=blue, line width=0.75mm] (1.8956448372144434, 0.6318816124048144) to[] (2.173925271309261, 0.4927413953574058) ;
        \draw[color=blue, line width=0.75mm] (1.8956448372144434, 0.6318816124048144) to[] (2.034785054261852, 0.9101620464996318) ;
        \draw[color=blue, line width=0.75mm] (16.104355162785557, 1.3681183875951854) to[] (15.826074728690739, 1.5072586046425942) ;
        \draw[color=blue, line width=0.75mm] (16.104355162785557, 1.3681183875951854) to[] (15.965214945738148, 1.089837953500368) ;
        \foreach \i in {0, 1}
        {
            \coordinate (v\i x0) at (12 * \i,0.00);
            \coordinate (v\i x1) at (12 * \i,2.00);
            \coordinate (v\i x3) at (12 * \i + 6.00,0.00);
            \coordinate (v\i x4) at (12 * \i + 6.00,2.00);
            \coordinate (v\i x5) at (12 * \i + 4.00,4.00);
            \coordinate (v\i x6) at (12 * \i + 8.00,4.00);
            \draw[color=red,line width=0.75mm] (v\i x3) to[] (v\i x4);
            \draw[color=red,line width=0.75mm] (v\i x4) to[] (v\i x5);
            \draw[color=red,line width=0.75mm] (v\i x4) to[] (v\i x6);
            \draw[color=blue,line width=0.75mm] (v\i x1) to[] (v\i x3);
            \draw[color=blue,line width=0.75mm] (v\i x1) to[] (v\i x5);
            \draw[color=blue,line width=0.75mm] (v\i x1) to[] (v\i x6);
            \draw[color=red, line width=0.6mm] (12 * \i + 0.0, 2.0) to[] (12 * \i + 0.0, 0.0) ;
            \draw[color=red, line width=0.6mm] (12 * \i + 0.1, 2.0) to[] (12 * \i + 0.1, 0.0) ;
            \draw[color=red, line width=0.6mm] (12 * \i + -0.1, 2.0) to[] (12 * \i + -0.1, 0.0) ;
            \draw[color=blue, line width=0.6mm] (12 * \i + 6.0, 2.0) to[] (12 * \i + 0.0, 0.0) ;
            \draw[color=blue, line width=0.6mm] (12 * \i + 6.031622776601684, 1.9051316701949486) to[] (12 * \i + 0.0316227766016838, -0.09486832980505139) ;
            \draw[color=blue, line width=0.6mm] (12 * \i + 5.968377223398316, 2.094868329805051) to[] (12 * \i + -0.0316227766016838, 0.09486832980505139) ;
            \draw[fill=black!75!white] (v\i x0.center) circle (0.2);
            \draw[fill=white] (v\i x1.center) circle (0.2);
            \draw[fill=black!75!white] (v\i x3.center) circle (0.2);
            \draw[fill=white] (v\i x4.center) circle (0.2);
            \draw[fill=black!75!white] (v\i x5.center) circle (0.2);
            \draw[fill=black!75!white] (v\i x6.center) circle (0.2);
        }
        \end{tikzpicture}}
        \caption{\label{fig:GD} The double bindings $G_2\ltimes_1 D_4$ and $G_2\ltimes_2 D_4$.}
        \end{figure}
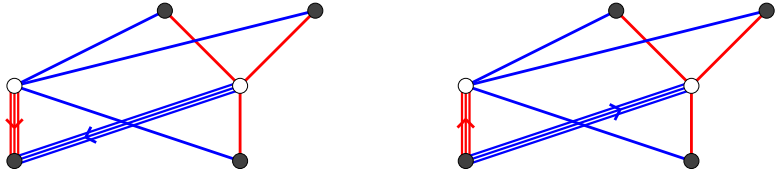

\newpage
\begin{proposition}[\cite{stembridge}]
\label{StembridgeAdmissibleBindings}
    A Dynkin biagram is admissible if and only if all of its bindings are admissible.
\end{proposition}

By definition, all parallel bindings are tensor products, which are admissible.
We also know that twists are admissible.
Using Proposition \ref{StembridgeAdmissibleBindings}, we can string together various sequences of parallel bindings and twists to form admissible Dynkin biagrams. For example, 
\begin{align*}
    B_n\ast C_n &:= (B_2\equiv B_2\equiv \cdots \equiv B_2\times B_2)^* \text{ and}\\
    F_4\ast F_4 &:= (B_2\equiv B_2\times B_2\equiv B_2)^*
\end{align*}
are both admissible bindings (see Figure \ref{fig:BCFF}).

We can construct even more bindings by dualizing strings of parallel bindings and folded twists as follows:
\begin{align*}
    A_{2n-1}\ast B_n &:= (D_3\equiv D_3\equiv \cdots \equiv D_3\rtimes B_2)^*,\\
    A_{2n-1}\ast C_n &:= (D_3\equiv D_3\equiv \cdots \equiv D_3\rtimes C_2)^*,\\
    B_n\ast D_{n+1} &:= (C_2\equiv C_2\equiv \cdots \equiv C_2\ltimes D_3)^*,\\
    C_n\ast D_{n+1} &:= (B_2\equiv B_2\equiv \cdots \equiv B_2\ltimes D_3)^*,\\
    E_6\ast_{1} F_4 &:= (B_2\equiv B_2\ltimes D_3\equiv D_3)^*,\\
    E_6\ast_{2} F_4 &:= (C_2\equiv C_2\ltimes D_3\equiv D_3)^*.
\end{align*}
Each of these bindings are depicted in Figures \ref{fig:BDCDBinding}, \ref{fig:ABAC}, and \ref{fig:EF}.

\begin{figure}
\scalebox{0.5}{
\begin{tikzpicture}
\foreach \i in {0, 1, 2, 3, 4}
{
    \coordinate (v0x\i) at (2 * \i, 0.00);
    \coordinate (v1x\i) at (2 * \i, 2.00);
}
\foreach \i in {0, 4}
{
    \draw[color=red, line width=0.75mm] (2 * \i + 7.11, 0.0) to[] (2 * \i + 6.89, 0.22) ;
    \draw[color=red, line width=0.75mm] (2 * \i + 7.11, 0.0) to[] (2 * \i + 6.89, -0.22) ;
    \draw[color=red, line width=0.75mm] (2 * \i + 6.0, 0.07) to[] (2 * \i + 8.0, 0.07) ;
    \draw[color=red, line width=0.75mm] (2 * \i + 6.0, -0.07) to[] (2 * \i + 8.0, -0.07) ;
    \draw[color=red, line width=0.75mm] (2 * \i + 6.89, 2.0) to[] (2 * \i + 7.11, 1.78) ;
    \draw[color=red, line width=0.75mm] (2 * \i + 6.89, 2.0) to[] (2 * \i + 7.11, 2.22) ;
    \draw[color=red, line width=0.75mm] (2 * \i + 8.0, 1.93) to[] (2 * \i + 6.0, 1.93) ;
    \draw[color=red, line width=0.75mm] (2 * \i + 8.0, 2.07) to[] (2 * \i + 6.0, 2.07) ;
}
\foreach \i in {0, 1, 2, 3, 6, 7}
{
    \pgfmathtruncatemacro{\j}{\i + 5};
    \coordinate (v0x\j) at (2 * \i + 12, 0.00);
    \coordinate (v1x\j) at (2 * \i + 12, 2.00);
    \draw[color=blue, line width=0.75mm] (2 * \i, 1.11) to[] (2 * \i + -0.22, 0.89) ;
    \draw[color=blue, line width=0.75mm] (2 * \i, 1.11) to[] (2 * \i + 0.22, 0.89) ;
    \draw[color=blue, line width=0.75mm] (2 * \i + -0.07, 0.0) to[] (2 * \i + -0.07, 2.0) ;
    \draw[color=blue, line width=0.75mm] (2 * \i + 0.07, 0.0) to[] (2 * \i + 0.07, 2.0) ;
}
\foreach \i in {0, 1, 2, 5, 7}
{
    \pgfmathtruncatemacro{\j}{\i + 1};
    \draw[color=red,line width=0.75mm] (v0x\i) to[] (v0x\j);
    \draw[color=red,line width=0.75mm] (v1x\i) to[] (v1x\j);
}
\foreach \i in {4, 8, 9}
{
    \draw[color=blue, line width=0.75mm] (2 * \i, 0.89) to[] (2 * \i + 0.22, 1.11) ;
    \draw[color=blue, line width=0.75mm] (2 * \i, 0.89) to[] (2 * \i + -0.22, 1.11) ;
    \draw[color=blue, line width=0.75mm] (2 * \i + 0.07, 2.0) to[] (2 * \i + 0.07, 0.0) ;
    \draw[color=blue, line width=0.75mm] (2 * \i + -0.07, 2.0) to[] (2 * \i + -0.07, 0.0) ;
}
\foreach \i in {0, 2, 4, 5, 7}
{
    \draw[fill=black!75!white] (v0x\i.center) circle (0.2);
    \draw[fill=white] (v1x\i.center) circle (0.2);
}
\foreach \i in {1, 3, 6, 8}
{
    \draw[fill=black!75!white] (v1x\i.center) circle (0.2);
    \draw[fill=white] (v0x\i.center) circle (0.2);
}
\end{tikzpicture}}
\caption{\label{fig:BCFF} The bindings $B_5\ast C_5$ and $F_4\ast F_4$.}
\end{figure}
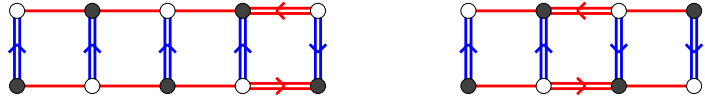

\begin{figure}[h]
\scalebox{0.5}{
\begin{tikzpicture}
\draw[color=red, line width=0.75mm] (0.0, 3.11) to[] (-0.22, 2.89) ;
\draw[color=red, line width=0.75mm] (0.0, 3.11) to[] (0.22, 2.89) ;
\draw[color=red, line width=0.75mm] (14.0, 2.89) to[] (14.22, 3.11) ;
\draw[color=red, line width=0.75mm] (14.0, 2.89) to[] (13.78, 3.11) ;
\draw[color=blue, line width=0.75mm] (2.11, 0.0) to[] (1.89, 0.22) ;
\draw[color=blue, line width=0.75mm] (2.11, 0.0) to[] (1.89, -0.22) ;
\draw[color=blue, line width=0.75mm] (2.11, 2.0) to[] (1.89, 2.22) ;
\draw[color=blue, line width=0.75mm] (2.11, 2.0) to[] (1.89, 1.78) ;
\draw[color=blue, line width=0.75mm] (15.89, 0.0) to[] (16.11, -0.22) ;
\draw[color=blue, line width=0.75mm] (15.89, 0.0) to[] (16.11, 0.22) ;
\draw[color=blue, line width=0.75mm] (15.89, 2.0) to[] (16.11, 1.78) ;
\draw[color=blue, line width=0.75mm] (15.89, 2.0) to[] (16.11, 2.22) ;
\foreach \i in {0, 1}
{
    \coordinate (v\i x0) at (14 * \i,0.00);
    \coordinate (v\i x1) at (14 * \i,2.00);
    \coordinate (v\i x2) at (14 * \i,4.00);
    \coordinate (v\i x3) at (14 * \i + 4.00,0.00);
    \coordinate (v\i x4) at (14 * \i + 4.00,2.00);
    \coordinate (v\i x5) at (14 * \i + 3.00,4.00);
    \coordinate (v\i x6) at (14 * \i + 6.00,3.00);
    \draw[color=red,line width=0.75mm] (v\i x0) to[] (v\i x1);
    \draw[color=red,line width=0.75mm] (v\i x3) to[] (v\i x4);
    \draw[color=red,line width=0.75mm] (v\i x4) to[] (v\i x5);
    \draw[color=red,line width=0.75mm] (v\i x4) to[] (v\i x6);
    \draw[color=blue,line width=0.75mm] (v\i x2) to[] (v\i x5);
    \draw[color=blue,line width=0.75mm] (v\i x2) to[] (v\i x6);
    \draw[color=blue, line width=0.75mm] (14 * \i + 0.0, 0.07) to[] (14 * \i + 4.0, 0.07) ;
    \draw[color=blue, line width=0.75mm] (14 * \i + 0.0, -0.07) to[] (14 * \i + 4.0, -0.07) ;
    \draw[color=blue, line width=0.75mm] (14 * \i + 0.0, 2.07) to[] (14 * \i + 4.0, 2.07) ;
    \draw[color=blue, line width=0.75mm] (14 * \i + 0.0, 1.93) to[] (14 * \i + 4.0, 1.93) ;
    \draw[color=red, line width=0.75mm] (14 * \i + -0.07, 2.0) to[] (14 * \i + -0.07, 4.0) ;
    \draw[color=red, line width=0.75mm] (14 * \i + 0.07, 2.0) to[] (14 * \i + 0.07, 4.0) ;
    \draw[fill=black!75!white] (v\i x0.center) circle (0.2);
    \draw[fill=white] (v\i x1.center) circle (0.2);
    \draw[fill=black!75!white] (v\i x2.center) circle (0.2);
    \draw[fill=white] (v\i x3.center) circle (0.2);
    \draw[fill=black!75!white] (v\i x4.center) circle (0.2);
    \draw[fill=white] (v\i x5.center) circle (0.2);
    \draw[fill=white] (v\i x6.center) circle (0.2);
}
\end{tikzpicture}}
\caption{\label{fig:BDCDBinding} The bindings $B_3\ast D_4$ and $C_3\ast D_4$.}
\end{figure}
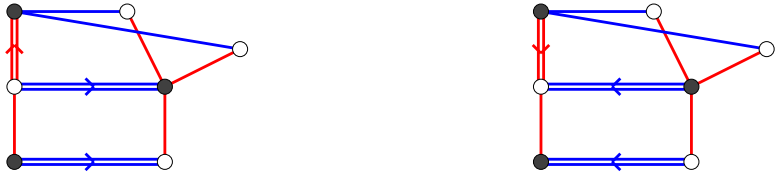

\begin{figure}
\scalebox{0.5}{
\begin{tikzpicture}
\draw[color=red, line width=0.75mm] (8.0, 4.61) to[] (7.78, 4.39) ;
\draw[color=red, line width=0.75mm] (8.0, 4.61) to[] (8.22, 4.39) ;
\draw[color=red, line width=0.75mm] (12 + 8.0, 4.39) to[] (12 + 8.22, 4.61) ;
\draw[color=red, line width=0.75mm] (12 + 8.0, 4.39) to[] (12 + 7.78, 4.61) ;
\draw[color=blue, line width=0.75mm] (5.109620033406936, 5.259135002783911) to[] (4.872109961025241, 5.460105064029962) ;
\draw[color=blue, line width=0.75mm] (5.109620033406936, 5.259135002783911) to[] (4.9086499721608865, 5.021624930402215) ;
\draw[color=blue, line width=0.75mm] (12 + 4.890379966593064, 5.240864997216089) to[] (12 + 5.127890038974759, 5.039894935970038) ;
\draw[color=blue, line width=0.75mm] (12 + 4.890379966593064, 5.240864997216089) to[] (12 + 5.0913500278391135, 5.478375069597785) ;
\foreach \k in {0, 1}
{
\coordinate (\k v0x0) at (12 * \k + 0.00,0.00);
\coordinate (\k v0x1) at (12 * \k + 0.00,2.00);
\coordinate (\k v0x2) at (12 * \k + 0.00,4.00);
\coordinate (\k v0x3) at (12 * \k + 2.00,5.00);
\coordinate (\k v0x4) at (12 * \k + 3.50,3.00);
\coordinate (\k v0x5) at (12 * \k + 3.50,1.00);
\coordinate (\k v0x6) at (12 * \k + 3.50,-1.00);
\coordinate (\k v1x0) at (12 * \k + 8.00,-0.50);
\coordinate (\k v1x1) at (12 * \k + 8.00,1.50);
\coordinate (\k v1x2) at (12 * \k + 8.00,3.50);
\coordinate (\k v1x3) at (12 * \k + 8.00,5.50);
\foreach \i in {0, 1, 2, 3, 4, 5}
{
    \pgfmathtruncatemacro{\j}{\i + 1};
    \draw[color=red,line width=0.75mm] (\k v0x\i) to[] (\k v0x\j);
}
\draw[color=red,line width=0.75mm] (\k v1x0) to[] (\k v1x1);
\draw[color=red,line width=0.75mm] (\k v1x1) to[] (\k v1x2);
\draw[color=red, line width=0.75mm] (12 * \k + 7.93, 3.5) to[] (12 * \k + 7.93, 5.5) ;
\draw[color=red, line width=0.75mm] (12 * \k + 8.07, 3.5) to[] (12 * \k + 8.07, 5.5) ;
\foreach \i in {0, 1, 2}
{
    \pgfmathtruncatemacro{\j}{6 - \i};
    \draw[color=blue,line width=0.75mm] (\k v0x\i) to[] (\k v1x\i);
    \draw[color=blue,line width=0.75mm] (\k v0x\j) to[] (\k v1x\i);
}
\draw[color=blue, line width=0.75mm] (12 * \k + 1.994186816410238, 5.069758203077142) to[] (12 * \k + 7.994186816410238, 5.569758203077142) ;
\draw[color=blue, line width=0.75mm] (12 * \k + 2.005813183589762, 4.930241796922858) to[] (12 * \k + 8.005813183589762, 5.430241796922858) ;
\foreach \i in {0, 2, 4, 6}
{
    \draw[fill=black!75!white] (\k v0x\i.center) circle (0.2);
}
\foreach \i in {1, 3, 5}
{
    \draw[fill=white] (\k v0x\i.center) circle (0.2);
}
\draw[fill=white] (\k v1x0.center) circle (0.2);
\draw[fill=black!75!white] (\k v1x1.center) circle (0.2);
\draw[fill=white] (\k v1x2.center) circle (0.2);
\draw[fill=black!75!white] (\k v1x3.center) circle (0.2);
}
\end{tikzpicture}}
\caption{\label{fig:ABAC} The bindings $A_7\ast B_4$ and $A_7\ast C_4$.}
\end{figure}
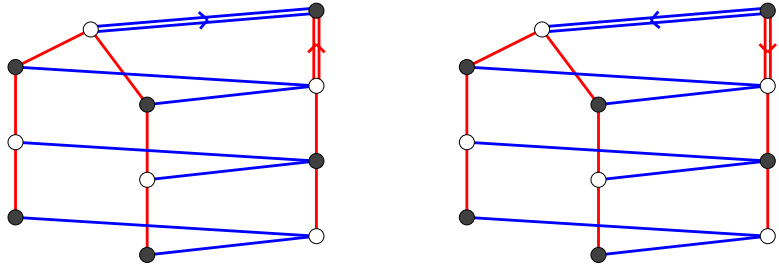

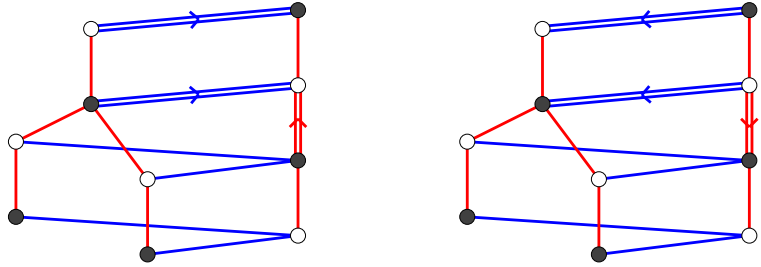
\begin{figure}
\scalebox{0.5}{
\begin{tikzpicture}
\draw[color=red, line width=0.75mm] (7.5, 2.61) to[] (7.28, 2.39) ;
\draw[color=red, line width=0.75mm] (7.5, 2.61) to[] (7.72, 2.39) ;
\draw[color=blue, line width=0.75mm] (4.8595482527114475, 3.259958932064677) to[] (4.620533883159198, 3.459137573358218) ;
\draw[color=blue, line width=0.75mm] (4.8595482527114475, 3.259958932064677) to[] (4.660369611417907, 3.020944562512428) ;
\draw[color=blue, line width=0.75mm] (4.8595482527114475, 5.259958932064677) to[] (4.620533883159198, 5.459137573358218) ;
\draw[color=blue, line width=0.75mm] (4.8595482527114475, 5.259958932064677) to[] (4.660369611417907, 5.020944562512428) ;

\draw[color=red, line width=0.75mm] (12 + 7.5, 2.39) to[] (12 + 7.72, 2.61) ;
\draw[color=red, line width=0.75mm] (12 + 7.5, 2.39) to[] (12 + 7.28, 2.61) ;
\draw[color=blue, line width=0.75mm] (12 + 4.6404517472885525, 3.240041067935323) to[] (12 + 4.879466116840802, 3.040862426641782) ;
\draw[color=blue, line width=0.75mm] (12 + 4.6404517472885525, 3.240041067935323) to[] (12 + 4.839630388582093, 3.479055437487572) ;
\draw[color=blue, line width=0.75mm] (12 + 4.6404517472885525, 5.240041067935323) to[] (12 + 4.879466116840802, 5.040862426641782) ;
\draw[color=blue, line width=0.75mm] (12 + 4.6404517472885525, 5.240041067935323) to[] (12 + 4.839630388582093, 5.479055437487572) ;

\foreach \k in {0, 1}
{
\coordinate (\k v0x0) at (12 * \k + 0.00,0.00);
\coordinate (\k v0x1) at (12 * \k + 0.00,2.00);
\coordinate (\k v0x2) at (12 * \k + 2.00,3.00);
\coordinate (\k v0x3) at (12 * \k + 3.50,1.00);
\coordinate (\k v0x4) at (12 * \k + 3.50,-1.00);
\coordinate (\k v0x5) at (12 * \k + 2.00,5.00);

\coordinate (\k v1x0) at (12 * \k + 7.50,5.50);
\coordinate (\k v1x1) at (12 * \k + 7.50,3.50);
\coordinate (\k v1x2) at (12 * \k + 7.50,1.50);
\coordinate (\k v1x5) at (12 * \k + 7.50,-0.50);

\draw[color=red, line width=0.75mm] (12 * \k + 7.43, 1.5) to[] (12 * \k + 7.43, 3.5) ;
\draw[color=red, line width=0.75mm] (12 * \k + 7.57, 1.5) to[] (12 * \k + 7.57, 3.5) ;
\draw[color=blue, line width=0.75mm] (12 * \k + 1.9936624977770236, 3.0697125244527395) to[] (12 * \k + 7.493662497777024, 3.5697125244527395) ;
\draw[color=blue, line width=0.75mm] (12 * \k + 2.0063375022229764, 2.9302874755472605) to[] (12 * \k + 7.506337502222976, 3.4302874755472605) ;
\draw[color=blue, line width=0.75mm] (12 * \k + 1.9936624977770236, 5.069712524452739) to[] (12 * \k + 7.493662497777024, 5.569712524452739) ;
\draw[color=blue, line width=0.75mm] (12 * \k + 2.0063375022229764, 4.930287475547261) to[] (12 * \k + 7.506337502222976, 5.430287475547261) ;

\draw[color=blue,line width=0.75mm] (\k v0x0) to[] (\k v1x5);
\draw[color=blue,line width=0.75mm] (\k v0x4) to[] (\k v1x5);
\draw[color=blue,line width=0.75mm] (\k v0x1) to[] (\k v1x2);
\draw[color=blue,line width=0.75mm] (\k v0x3) to[] (\k v1x2);
\draw[color=red,line width=0.75mm] (\k v0x1) to[] (\k v0x2);
\draw[color=red,line width=0.75mm] (\k v0x2) to[] (\k v0x3);
\draw[color=red,line width=0.75mm] (\k v0x3) to[] (\k v0x4);
\foreach \i in {0, 1}
{
    \draw[color=red,line width=0.75mm] (\k v\i x0) to[] (\k v\i x1);
    \draw[color=red,line width=0.75mm] (\k v\i x2) to[] (\k v\i x5);
    \draw[fill=black!75!white] (\k v\i x0.center) circle (0.2);
    \draw[fill=white] (\k v\i x1.center) circle (0.2);
    \draw[fill=black!75!white] (\k v\i x2.center) circle (0.2);
    \draw[fill=white] (\k v\i x5.center) circle (0.2);
}
\draw[fill=white] (\k v0x3.center) circle (0.2);
\draw[fill=black!75!white] (\k v0x4.center) circle (0.2);
}
\end{tikzpicture}}
\caption{\label{fig:EF} The bindings $E_6\ast_1 F_4$ and $E_6\ast_2 F_4$.}
\end{figure}

Now, we are able to describe the classification of admissible non-ADE Dynkin biagrams in terms of the bindings defined above.

\newpage

\begin{theorem}
\label{classThm}
    The connected, admissible non-ADE Dynkin biagrams are:
    \begin{itemize}
        \item [(a)]
            tensor products $\Gamma\otimes \Delta$ for $\Gamma$ or $\Delta$ non-ADE, twists $\Gamma\times \Gamma$ for $\Gamma$ non-ADE, and the double bindings $B_n\ltimes D_{n+1}$, $C_n\ltimes D_{n+1}$, $B_n\bowtie C_n$, $G_2\ltimes_{1, 2} D_4$, $B_3\bowtie_{1,2} G_2$, $C_3\bowtie_{1,2} G_2$, and $B_4\boxtimes C_4$,
        \item [(b)]
            bindings $B_n\ast A_{2n-1}$, $C_n\ast A_{2n-1}$, $B_n\ast C_n$, $B_n\ast D_{n+1}$, $C_n\ast D_{n+1}$, $F_4\ast_{1,2} E_6$, $F_4\ast F_4$,
        \item [(c)]
            $B_2\times B_2\equiv \dots \equiv B_2$, \hspace{2mm}$B_2\equiv B_2\times B_2\equiv B_2$,
        \item [(d)]
            $(BA^{m-1})_n\coloneqq (B_n\ast A_{2n-1}\equiv\cdots\equiv A_{2n-1}), $ \hspace{2mm}$(B^{m-1}A)_n\coloneqq (B_n\equiv\cdots\equiv B_n\ast A_{2n-1})$, 
        \item [(e)]
            $B_n\equiv B_n \ast A_{2n-1} \equiv A_{2n-1}$,
        \item [(f)]
            $(CA^{m-1})_n\coloneqq (C_n\ast A_{2n-1}\equiv\cdots\equiv A_{2n-1}), $ \hspace{2mm}$(C^{m-1}A)_n\coloneqq (C_n\equiv\cdots\equiv C_n\ast A_{2n-1})$,
        \item [(g)]
            $C_n\equiv C_n \ast A_{2n-1} \equiv A_{2n-1}$,
        \item [(h)]
            $(BC^{m-1})_n\coloneqq (B_n\ast C_n\equiv\cdots\equiv C_n), $ \hspace{2mm}$(B^{m-1}C)_n\coloneqq (B_n\equiv\cdots\equiv B_n\ast C_n)$,
        \item [(i)]
            $B_n\equiv B_n\ast C_n\equiv C_n$,
        \item [(j)]
            $(BD^{m-1})_n\coloneqq (B_n\ast D_{n+1}\equiv\cdots\equiv D_{n+1}), $ \hspace{2mm}$(B^{m-1}D)_n\coloneqq (B_n\equiv\cdots\equiv B_n\ast D_{n+1})$,
        \item [(k)]
            $B_n\equiv B_n \ast D_{n+1} \equiv D_{n+1}$,
        \item [(l)]
            $(CD^{m-1})_n\coloneqq (C_n\ast D_{n+1}\equiv\cdots\equiv D_{n+1}), $ \hspace{2mm}$(C^{m-1}D)_n\coloneqq (C_n\equiv\cdots\equiv C_n\ast D_{n+1})$,
        \item [(m)]
            $C_n\equiv C_n \ast D_{n+1} \equiv D_{n+1}$,
        \item [(n)]
            $(F_4E_6^{m-1})_{1,2}\coloneqq (F_4\ast_{1, 2} E_6\equiv\cdots\equiv E_6), $ \hspace{2mm}$(F_4^{m-1}E_6)_{1, 2}\coloneqq (F_4\equiv\cdots\equiv F_4\ast_{1, 2} E_6)$,
        \item [(o)]
            $F_4\equiv F_4 \ast_{1} E_6 \equiv E_6$, \hspace{2mm}$F_4\equiv F_4 \ast_{2} E_6 \equiv E_6$,
        \item [(p)]
            $F_4\ast F_4\equiv \dots \equiv F_4$, \hspace{2mm}$F_4\equiv F_4\ast F_4\equiv F_4$.
    \end{itemize}
    In each of the families of the form $(\Lambda_1\Lambda_2^{m-1})$ or $(\Lambda_1^{m-1}\Lambda_2)$, there are $m$ factors.
\end{theorem}
\subsection{Operations preserving admissibility}

\begin{lemma}[Folding biagrams]
\label{foldingLem}
    Let $(\Gamma, \Delta)$ be an admissible Dynkin biagram and let $f$ be a bicolored automorphism. Then folding along $f$ produces another admissible Dynkin biagram $(\Gamma', \Delta')$.
\end{lemma}

\begin{proof}
    For any two vertices $i, j$ in $(\Gamma, \Delta)$, we know that
    \begin{align*}
        (\Delta\Gamma)_{ij} = \sum_k \Delta_{ik}\Gamma_{kj} = \sum_k \Gamma_{ik}\Delta_{kj} = (\Gamma\Delta)_{ij}.
    \end{align*}

    Now, fold along $f$ to produce $(f(\Gamma), f(\Delta))$. Let $I, J$ be disjoint orbits of $f$. 
    It suffices to show that the number of blue-red paths $I\rightarrow J$ is the same as the number of red-blue paths $I\rightarrow J$.
    \begin{align*}
        \sum_K f(\Delta)_{IK}f(\Gamma)_{KJ} &= \sum_K \left(\sum_{i\in I}\Delta_{ik'}\right)\left(\sum_{k\in K}\Gamma_{kj}\right)\quad\text{for arbitrary }k'\in K, j\in J\\
        &= \sum_K \sum_{k\in K} \Gamma_{kj}\left(\sum_{i\in I} \Delta_{ik}\right)
        = \sum_{i\in I} \sum_k \Delta_{ik}\Gamma_{kj}\\
        &= \sum_{i\in I} \sum_k \Gamma_{ik}\Delta_{kj}
        = \sum_K \sum_{k\in K} \Delta_{kj}\left(\sum_{i\in I} \Gamma_{ik}\right)\\
        &= \sum_K \left(\sum_{i\in I}\Gamma_{ik'}\right)\left(\sum_{k\in K}\Delta_{kj}\right)
        = \sum_K f(\Gamma)_{IK}f(\Delta)_{KJ}.
    \end{align*}
    So, $\big(f(\Delta)f(\Gamma)\big)_{IJ} = \big(f(\Gamma)f(\Delta)\big)_{IJ}$.
\end{proof}

The following operation plays an important role in the proof of the bijection between admissible Dynkin biagrams and Zamolodchikov periodic cluster algebras in Section \ref{ZP}.
\begin{definition}[Global Flip]
\label{flippingLem}
    Given a Dynkin biagram $(\Gamma, \Delta)$, we can produce another Dynkin biagram by taking the transpose $(\Gamma^{\top}, \Delta^{\top})$. We call this operation a \textit{global flip} of a Dynkin biagram.
\end{definition}
Notice that if $(\Gamma, \Delta)$ is a pair of commuting Cartan matrices, $(\Gamma^{\top}, \Delta^{\top})$ is also a pair of commuting Cartan matrices. 
In other words, admissibility is preserved under global flips. 
This operation has no effect on ADE types, and changes the direction of any nonsimple edges in non-ADE types. 
Figures \ref{fig:BDCD} and \ref{fig:GD} are examples of global flips.
\begin{remark}
    This operation is equivalent to taking the \textit{Langlands dual} of a group.
    In particular, if $G$ is a reductive algebraic group with Dynkin diagram $\Lambda$, its Langlands dual $^LG$ has Dynkin diagram $\Lambda'$, where $\Lambda'$ is the global flip of $\Lambda$.
\end{remark}

\subsection{Classification of admissible double bindings}
Much of the work of classification was in classifying the non-ADE double bindings by analyzing potential pairings of Dynkin diagrams and dominant eigenvectors of their corresponding Cartan matrices.
\begin{theorem}
\label{doubleBindingsThm}
    The admissible non-ADE double bindings are
    \begin{itemize}
        \item [(a)] 
            parallel bindings $\Gamma\equiv \Gamma$, $\Gamma\equiv_2 \Gamma$, and $\Gamma\equiv_3 \Gamma$ for $\Gamma\in\{A_2, B_2, G_2\},$
        \item [(b)]
            twists $\Gamma\times \Gamma$ for non-ADE $\Gamma$,
        \item [(c)]
            $B_n\ltimes D_{n+1}$, $C_n\ltimes D_{n+1}$, $B_n\bowtie C_n$, $G_2\ltimes_{1, 2} D_4$, $B_3\bowtie_{1,2} G_2$, $C_3\bowtie_{1,2} G_2$, and
        \item [(d)]
            the exceptional binding $B_4\boxtimes C_4$.
    \end{itemize}
\end{theorem}

\begin{lemma}
\label{admissibleDouble}
    All double bindings listed in Theorem \ref{doubleBindingsThm} are admissible.
\end{lemma}

\begin{proof}
    We have already discussed how parallel bindings and twists are admissible.
    
    All of the remaining double bindings listed in Theorem \ref{doubleBindingsThm} can be obtained from an admissible ADE bigraph, classified in \cite{stembridge}, through a series of global flips and foldings along bicolored automorphisms as follows:
    \begin{align*}
        D_{n+1}\times D_{n+1} &\mapsto B_n\ltimes D_{n+1},\\
        \left(B_n\ltimes D_{n+1}\right)^{\top} &= C_n\ltimes D_{n+1},\\
        C_n\ltimes D_{n+1} &\mapsto C_n\bowtie B_n,\\
        D_4\times D_4 &\mapsto G_2\ltimes_1 D_4,\\
        \left(G_2\ltimes_1 D_4\right)^{\top} &= G_2\ltimes_2 D_4,\\
        G_2\ltimes_1 D_4 &\mapsto G_2\bowtie_1 B_3,\\ 
        G_2\ltimes_2 D_4 &\mapsto G_2\bowtie_2 B_3,\\
        D_4\rtimes C_3 &\mapsto G_2\bowtie_1 C_3,\\ 
        \left(G_2\bowtie_1 B_3\right)^{\top} &= G_2\bowtie_2 C_3,\\
        D_5\boxtimes A_7 &\mapsto B_4\boxtimes C_4.
    \end{align*}
    By Lemma \ref{foldingLem} and Lemma \ref{flippingLem}, these are all admissible bindings.
\end{proof}

\begin{corollary}
\label{admissibleBindings}
    Each of the bindings $A_{2n-1}\ast B_n$, $A_{2n-1}\ast C_n$, $B_n\ast C_n$, $B_n\ast D_{n+1}$, $C_n\ast D_{n+1}$, $E_6\ast_{1,2} F_4$, $F_4\ast F_4$ are admissible.
\end{corollary}

The following proposition lists the dominant eigenvectors for each Dynkin diagram's Coxeter adjacency matrix. A calculation for each ADE type is in \cite[Appendix A]{stembridge}, and the non-ADE types can be obtained via a standard folding argument or a direct computation.
\begin{proposition}
\label{eigenvectors}
    The following are the dominant eigenvectors for each Dynkin diagram's Coxeter adjacency matrix:
    \begin{itemize}
        \item[(i)]
            If the vertices of $A_n$ are numbered $1-2-\cdots-n$, then
            $$\vv{v}_{A} = \left[\sin\left(\frac{\pi}{n+1}\right), \sin\left(\frac{2\pi}{n+1}\right), \dots, \sin\left(\frac{n\pi}{n+1}\right)\right].$$
        \item[(ii)]
            If the vertices of $B_n$ are numbered 
            \[1\hspace{-1mm}=\hspace{-1mm}\Leftarrow\hspace{-1mm} 2-3-\cdots -n,\]
            then the left and right dominant eigenvectors are
            \begin{align*}
                \vv{v}_{B}^{L} &= \left[\frac{1}{2}, \cos\left(\frac{\pi}{2n}\right), \cos\left(\frac{2\pi}{2n}\right), \dots, \cos \left(\frac{(n-1)\pi}{2n}\right)\right] \text{ and}\\
                \vv{v}_{B}^{R} &= \left[1, \cos\left(\frac{\pi}{2n}\right), \cos\left(\frac{2\pi}{2n}\right), \dots, \cos \left(\frac{(n-1)\pi}{2n}\right)\right].
            \end{align*}
        \item[(iii)]
            If the vertices of $C_n$ are numbered 
            \[1\hspace{-1mm}\Rightarrow\hspace{-1mm}=\hspace{-1mm} 2-3-\cdots -n,\]
            then the left  and right dominant eigenvectors are
            \begin{align*}
                \vv{v}_{C}^{L} &= \left[1, \cos\left(\frac{\pi}{2n}\right), \cos\left(\frac{2\pi}{2n}\right), \dots, \cos \left(\frac{(n-1)\pi}{2n}\right)\right] \text{ and}\\
                \vv{v}_{C}^{R} &= \left[\frac{1}{2}, \cos\left(\frac{\pi}{2n}\right), \cos\left(\frac{2\pi}{2n}\right), \dots, \cos \left(\frac{(n-1)\pi}{2n}\right)\right].
            \end{align*}
        \item[(iv)]
            If the vertices of $D_n$ are numbered
            \begin{align*}
                &2\\[-2mm]
                &\hspace{1mm}\rule{0.7pt}{1.6ex}\\[-2mm]
                1\hspace{1mm}-\hspace{1mm}&3-4-\cdots -n,
            \end{align*}
            then
            $$\vv{v}_{D} = \left[\frac{1}{2}, \frac{1}{2}, \cos\left(\frac{\pi}{2n - 2}\right), \cos\left(\frac{2\pi}{2n - 2}\right), \dots, \cos \left(\frac{(n-2)\pi}{2n - 2}\right)\right].$$
        \item[(v)]
            If the vertices of $E_n$ $(n = 6, 7, 8)$ are numbered
            \begin{align*}
                &2\\[-2mm]
                &\hspace{1mm}\rule{0.7pt}{1.6ex}\\[-2mm]
                1-3\hspace{1mm}-\hspace{1mm}&4-5-\cdots -n,
            \end{align*}
            and $\theta = \frac{\pi}{h}$, where $h = 12, 18, 30$ is the Coxeter number, then
            $$\vv{v}_{E} = \left[\frac{\sin{\theta}}{\sin{3\theta}}, \frac{\sin{\theta}}{\sin{2\theta}}, \frac{\sin{2\theta}}{\sin{3\theta}}, 1, \frac{\sin{(n-4)\theta}}{\sin{(n-3)\theta}}, \dots, \frac{\sin{\theta}}{\sin{(n-3)\theta}}\right].$$
        \item[(vi)]
            If the vertices of $F_4$ are numbered $1-2\hspace{-1mm}\Rightarrow\hspace{-1mm}=\hspace{-1mm} 3-4$, then the left and right dominant eigenvectors are
            $$\vv{v}_{F}^{L} = \left[1, \frac{2}{\sqrt{6} - \sqrt{2}}, \frac{1}{\sqrt{3} - 1}, \frac{\sqrt{2}}{2}\right] \quad \text{and} \quad \vv{v}_{F}^{R} = \left[1, \frac{2}{\sqrt{6} - \sqrt{2}}, \frac{2}{\sqrt{3} - 1}, \sqrt{2}\right].$$
        \item[(vii)]
            If the vertices of $G_2$ are numbered $1\hspace{-1mm}\equiv\hspace{-1mm}\Lleftarrow\hspace{-1mm} 2$, then the left and right dominant eigenvectors are
            $$\vv{v}_{G}^{L} = \left[1, \sqrt{3}\right] \quad \text{and}\quad \vv{v}_{G}^{R} = \left[1, \frac{\sqrt{3}}{3}\right].$$
    \end{itemize}
    
\end{proposition}

We now present three lemmas that will be useful for identifying double bindings in the proof of Theorem \ref{doubleBindingsThm}.

\begin{lemma}
\label{commonEigVec}
    In a connected admissible Dynkin biagram $(\Gamma, \Delta)$, the Coxeter adjacency matrices $\Gamma = 2I - C_{\Gamma}$ and $\Delta = 2I - C_{\Delta}$ have a common dominant eigenvector $\vv{v}\in \mathbb{R}^n_{>0}$. 
\end{lemma}
\begin{proof}
    Recall that if $B$ is the adjacency matrix of a strongly connected directed graph, then the Perron--Frobenius Theorem implies that $B$ has a dominant eigenvalue with left and right eigenspaces of dimension 1, the corresponding eigenvectors have positive coordinates, and these are the only positive eigenvectors of $B$.

    Notice that every connected Dynkin biagram corresponds to a pair of nonnegative adjacency matrices $\Gamma, \Delta$ that together describe a strongly connected directed graph. 
    Let $B = \Gamma + \Delta$ be the resulting adjacency matrix.
    Then,
    \begin{align*}
        B\Gamma &= \Gamma^2 + \Delta\Gamma = \Gamma^2 + \Gamma\Delta = \Gamma B.
    \end{align*}
    If $\vv{v}$ is a dominant eigenvector of $B$, we have
    $$B\Gamma\vv{v} = \Gamma B\vv{v} = \lambda_{B}\Gamma\vv{v}.$$
    So, $\Gamma\vv{v}$ is in the $\lambda_B$-eigenspace of $B$. By Perron--Frobenius, this eigenspace is of dimension 1, hence $\Gamma\vv{v} = c\cdot \vv{v}$ is some scalar multiple of $\vv{v}$. Thus $\vv{v}$ is an eigenvector of $\Gamma$. By a similar argument, $\vv{v}$ is also an eigenvector of $\Delta$.

    Since $\vv{v}$ is a dominant eigenvector of $B$, by the Perron--Frobenius Theorem, $\vv{v}$ has all positive values. 
    Let $\Gamma$ have $k$ components: $\Gamma = \Gamma_1\oplus \dots \oplus \Gamma_k$.
    $\vv{v}$ can be decomposed into a sum of positive eigenvectors supported on each $\Gamma$-component: $\vv{v} = \vv{v}_1 + \dots + \vv{v}_k$.
    By Perron--Frobenius, each $\vv{v}_i$ is a dominant eigenvector for $\Gamma_i$.
    So, $\vv{v}$ is a sum of dominant eigenvectors for $\Gamma_1, \dots, \Gamma_k$, making it a dominant eigenvector for $\Gamma$. 
    By an analogous argument, $\vv{v}$ is a dominant eigenvector for $\Delta$.
\end{proof}

\begin{corollary}
\label{coxeterNum}
    In a connected admissible Dynkin biagram $(\Gamma, \Delta)$, all irreducible components of $\Gamma$ (resp., $\Delta$) have the same Coxeter number $h_{\Gamma}$ (resp., $h_{\Delta}$).
\end{corollary}

\begin{proof}
    Let $\Lambda$ be a Dynkin diagram. From Lemma 2 of \cite{eigenvalues}, 
    $$\lambda = 2\cos\left(\frac{\pi}{h}\right),$$
    where $\lambda$ is the dominant eigenvalue of the Coxeter adjacency matrix $A_{D}$ and $h$ is the Coxeter number of $\Lambda$. 
    
    Given a connected admissible Dynkin biagram $(\Gamma, \Delta)$, let $\vv{v}$ be the common dominant eigenvector of $\Gamma, \Delta$ made up of all positive coordinates. Recall from the proof of Lemma \ref{commonEigVec} that $\vv{v} = \vv{v}_1 + \dots + \vv{v}_k$ is a sum of dominant eigenvectors supported on each component $\Gamma_1, \dots, \Gamma_k$. Thus, 
    $$\Gamma\vv{v} = \Gamma_1\vv{v}_1 + \dots + \Gamma_k \vv{v}_k = \lambda\vv{v}_1 + \dots + \lambda \vv{v}_k,$$
    so each component $\Gamma_i$ has the same dominant eigenvalue. From \cite{eigenvalues}, this means each $\Gamma$-component has the same Coxeter number $h_{\Gamma}$.
    Similarly, each component of $\Delta$ has the same Coxeter number $h_{\Delta}$.
\end{proof}

\begin{corollary}
\label{uniqueEigenvector}
    Let $\Lambda_1, \Lambda_2$ be two Dynkin diagrams, and let $\vv{v}_1, \vv{v}_2$ be the corresponding dominant eigenvectors of their Coxeter adjacency matrices. If $\vv{v}_1 = c\cdot \vv{v}_2$ for some scalar $c\in \mathbb{R}_{>0}$, then $\Lambda_1 \cong \Lambda_2$.
\end{corollary}

Define the \textit{color type} of a Dynkin diagram $\Lambda$ to be the multiset $\{a, b\}$, where $a$ and $b$ denote the number of vertices of each color in the bipartite coloring of $\Lambda$. Table \ref{table:1} lists the color types of each of the Dynkin diagrams.

\begin{lemma}[\cite{stembridge}]
\label{colorType}
    If $(\Gamma, \Delta)$ is an admissible double binding and the two components of $\Gamma$ have color types $\{a, b\}$ and $\{c, d\}$, then the two components of $\Delta$ have color types $\{a, d\}$ and $\{c, b\}$ or $\{a, c\}$ and $\{b, d\}$.
\end{lemma}
\begin{proof}
    As $(\Gamma, \Delta)$ is bipartite, since every edge in $\Gamma$ connects $\Delta_1$ to $\Delta_2$ and $\Gamma_1$ is connected, $(\Gamma_1\cap \Delta_1)\sqcup (\Gamma_1\cap \Delta_2)$ is the bipartite partition of the vertices of $\Gamma_1$.
    Similarly, $(\Gamma_2\cap \Delta_1)\sqcup (\Gamma_2\cap \Delta_2)$ is the bipartite partition of the vertices of $\Gamma_2$, and $(\Gamma_1\cap \Delta_i)\sqcup (\Gamma_2\cap \Delta_i)$ is the bipartite partition of the vertices of $\Delta_i$, for $i = 1, 2$.

    Thus, the bipartite partition of $(\Gamma, \Delta)$ is 
    $$\big((\Gamma_1\cap \Delta_1)\cup (\Gamma_2\cap \Delta_2)\big) \sqcup \big((\Gamma_1\cap\Delta_2)\cup(\Gamma_2\cap\Delta_1)\big).$$
    The two components of $\Delta$ have color types $\{a, d\}$ and $\{c, b\}$ or $\{a, c\}$ and $\{b, d\}$.
\end{proof}


\begin{lemma}
\label{non-ADEDoubleBindingComponent}
    If $(\Delta, \Gamma)$ is an admissible non-ADE double binding, then at least two components of $\{\Gamma_1, \Gamma_2, \Delta_1, \Delta_2\}$ must be non-ADE.
\end{lemma}
\begin{proof}
    Suppose there is only one non-ADE component. Without loss of generality, let it be $\Gamma_1$. Let $i, j$ be two vertices in $\Gamma_1$ connected by  a nonsimple $\Gamma$-edge, and let $k$ be a vertex adjacent in $\Delta$ to $j$. Then, $\Gamma_{ij} \neq \Gamma_{ji}$. Notice that no Dynkin diagram has a vertex adjacent to two nonsimple edges, so $(i,j)$ is the only nonsimple $\Gamma$-edge adjacent to $i$.

    Now, we count $(\Gamma\Delta)_{ik}$. 
    All such paths will necessarily be made up of simple edges except for ones that trace $(i, j), (j, k)$.
    Tracing these simple paths backwards to produce blue-red paths $k\rightarrow i$ will yield the same number of paths.
    However, the number of paths $(i, j), (j, k)$ is $$\Gamma_{ij}\Delta_{jk} \neq \Delta_{kj}\Gamma_{ji},$$
    so the total number of red-blue paths $(\Gamma\Delta)_{ik}$ will not be the same as the total number of blue-red paths $(\Delta\Gamma)_{ki}$. 
    
    On the other hand, the number of blue-red paths $i\rightarrow k$ will be the same as the number of red-blue paths $k\rightarrow i$, as these paths are all made up of simple edges.
    Admissibility gives
    $$(\Gamma\Delta)_{ik} = (\Delta\Gamma)_{ik} = (\Gamma\Delta)_{ki} = (\Delta\Gamma)_{ki},$$
    which is a contradiction.
\end{proof}

\begin{table}
\tiny
    \begin{tabular}{|c|c|c|}
        \hline
        Dynkin Diagram & Coxeter Number $h$ & Color Type\\
        \hline
        $A_{2n}$ & $2n + 1$ & $\{n, n\}$\\
        $A_{2n + 1}$ & $2n + 2$ & $\{n + 1, n\}$\\
        $B_{2n}$ & $4n$ & $\{n, n\}$\\
        $B_{2n + 1}$ & $4n + 2$ & $\{n + 1, n\}$\\
        $C_{2n}$ & $4n$ & $\{n, n\}$\\
        $C_{2n + 1}$ & $4n + 2$ & $\{n + 1, n\}$\\
        $D_{2n}$ & $4n - 2$ & $\{n + 1, n - 1\}$\\
        $D_{2n + 1}$ & $4n$ & $\{n + 1, n\}$\\
        $E_6$ & 12 & $\{3, 3\}$\\
        $E_7$ & 18 & $\{4, 3\}$\\
        $E_8$ & 30 & $\{4, 4\}$\\
        $F_4$ & 12 & $\{2, 2\}$\\
        $G_2$ & 6 & $\{1, 1\}$\\
        \hline
    \end{tabular}
    \caption{Coxeter numbers and color types of each Dynkin diagram.}
    \label{table:1}
\end{table}

We are now ready to present the proof of Theorem \ref{doubleBindingsThm}.
We break up the proof into two parts.
This proof structure was adapted from \cite{stembridge}. 
First, we examine admissible double bindings where $\Gamma_1\not\cong \Gamma_2$.
Then, we classify all admissible double bindings with isomorphic components. 

\begin{lemma}
\label{nonisom}
    If $(\Gamma, \Delta)$ is an admissible double binding such that the two components of $\Gamma$ are not isomorphic, then $(\Gamma, \Delta)$ is either $B_n\ltimes D_{n+1}$, $C_n\ltimes D_{n+1}$, $B_n\bowtie C_n$, $G_2\ltimes_{1, 2} D_4$, $B_3\bowtie_{1,2} G_2$, $C_3\bowtie_{1,2} G_2$, or $B_4\boxtimes C_4$.
\end{lemma}

\begin{proof}
    Using Corollary \ref{coxeterNum}, we may examine all possible non-isomorphic pairings $\{\Gamma_1, \Gamma_2\}$ with the same Coxeter number. As all admissible ADE double bindings were classified by Stembridge \cite{stembridge}, we assume $\Gamma_1$ must be non-ADE. This gives us possibilities\\

    \begin{tabular}{l l l l l}
        $\{B_n(C_n), A_{2n-1}\}$, & $\{B_n(C_n), D_{n+1}\}$, & $\{B_6(C_6), E_6\}$, & $\{B_9(C_9), E_7\}$, & $\{B_{15}(C_{15}), E_8\}$,\\
        $\{B_6(C_6), F_4\}$, & $\{B_3(C_3), G_2\}$, & $\{B_n, C_n\}$, & $\{F_4, A_{11}\}$, & $\{F_4, D_7\}$,\\
        $\{F_4, E_6\}$, & $\{G_2, A_5\}$,& $\{G_2, D_4\}$. & &
    \end{tabular}\\

    Observe from Table \ref{table:1} that $\{a, b\}$ is a color type of a Dynkin diagram if and only if $|a-b| \leq 2$. Using this and Lemma \ref{colorType}, valid color types for $\Delta$ reduce the above list to\\

    \begin{tabular}{l m{4cm} l l}
        $\{B_{2n}, C_{2n}\}$ & $\big\{\{n, n\}, \{n, n\}\big\}$, & $\{B_{2n+1}, C_{2n+1}\}$ & $\big\{\{n+1, n\}, \{n+1, n\}\big\}$,\\
        $\{B_{2n}, D_{2n+1}\}$ & $\big\{\{n, n\}, \{n+1, n\}\big\}$, & $\{C_{2n}, D_{2n+1}\}$ & $\big\{\{n, n\}, \{n+1, n\}\big\}$,\\
        $\{B_3(C_3), A_5\}$ & $\big\{\{2, 1\}, \{3, 2\}\big\}$, & $\{B_{2n+1}(C_{2n+1}), D_{2n+2}\}$ & $\big\{\{n+1, n\}, \{n+2, n\}\big\}$,\\
        $\{B_4(C_4), A_7\}$ & $\big\{\{2, 2\}, \{4, 3\}\big\}$, & $\{B_5(C_5), A_9\}$ & $\big\{\{3, 2\}, \{5, 4\}\big\}$,\\
        $\{B_6(C_6), E_6\}$ & $\big\{\{3, 3\}, \{3, 3\}\big\}$, & $\{B_9(C_9), E_7\}$ & $\big\{\{5, 4\}, \{4, 3\}\big\}$,\\
        $\{B_6(C_6), F_4\}$ & $\big\{\{3, 3\}, \{2, 2\}\big\}$, & $\{B_3(C_3), G_2\}$ & $\big\{\{2, 1\}, \{1, 1\}\big\}$,\\
        $\{F_4, D_7\}$ & $\big\{\{2, 2\}, \{4, 3\}\big\}$, & $\{F_4, E_6\}$ & $\big\{\{2, 2\}, \{3, 3\}\big\}$,\\
        $\{G_2, A_5\}$ & $\big\{\{1, 1\}, \{3, 2\}\big\}$, & $\{G_2, D_4\}$ & $\big\{\{1, 1\}, \{3, 1\}\big\}.$
    \end{tabular}\\
    
    Notice Lemma \ref{colorType} asserts that if $\{a, b\}, \{c, d\}$ is a listed color type, then either $\{a, c\}, \{b, d\}$ or $\{a, d\}, \{b, c\}$ must be a valid color type.
    This combined with Lemma \ref{non-ADEDoubleBindingComponent} eliminates $\{B_5(C_5), A_9\}$ from the list.
    
    We first check that the following color types not listed in Lemma \ref{admissibleDouble} are incompatible double bindings.\\

    \begin{tabular}{l m{4cm} l l}
        $\{B_3(C_3), A_5\}$ & $\big\{\{2, 1\}, \{3, 2\}\big\},$ & $\{B_4(C_4), A_7\}$ & $\big\{\{2, 2\}, \{4, 3\}\big\},$\\
        $\{B_6(C_6), E_6\}$ & $\big\{\{3, 3\}, \{3, 3\}\big\},$ & $\{B_9(C_9), E_7\}$ & $\big\{\{5, 4\}, \{4, 3\}\big\},$\\
        $\{B_6(C_6), F_4\}$ & $\big\{\{3, 3\}, \{2, 2\}\big\},$ & $\{F_4, D_7\}$ & $\big\{\{2, 2\}, \{4, 3\}\big\},$\\
        $\{F_4, E_6\}$ & $\big\{\{2, 2\}, \{3, 3\}\big\},$ & $\{G_2, A_5\}$ & $\big\{\{1, 1\}, \{3, 2\}\big\}.$
    \end{tabular}
    
    Throughout this portion of the proof, we heavily use Lemma \ref{commonEigVec} to compare dominant eigenvectors of components in the double bindings. The color set projections of all dominant eigenvectors relevant to this proof are listed in Table \ref{table:2}.
    By Lemma \ref{commonEigVec}, we know that $\Gamma$ and $\Delta$ share a common (left/right) eigenvector $\vv{v}$ that is some linear combination of the dominant eigenvectors of $\Gamma_1, \Gamma_2$ and that of $\Delta_1, \Delta_2$:
    \begin{align*}
        \vv{v} &= c_1\vv{v}_{\Gamma_1}\oplus c_2\vv{v}_{\Gamma_2} = d_1\vv{v}_{\Delta_1} \oplus d_2\vv{v}_{\Delta_2}.
    \end{align*}
    So if we write 
    \begin{align*}
        A = \begin{pmatrix}
            \vv{v}_{\Gamma_1} & \vv{v}_{\Gamma_2}
        \end{pmatrix} = 
        \begin{pmatrix}
            \vv{v}_{\Gamma_1\cap \Delta_1} & \vv{v}_{\Gamma_2\cap \Delta_1}\\
            \vv{v}_{\Gamma_1\cap \Delta_2} & \vv{v}_{\Gamma_2\cap \Delta_2}
        \end{pmatrix},
    \end{align*}
    there is some scaling of the columns of $A$ which result in the rows being dominant eigenvectors for $\Delta_1, \Delta_2$. 
    Moreover, note that non-ADE Dynkin diagrams have distinct left and right dominant eigenvectors, which we call \textit{chiral} eigenvectors. 
    Each chiral eigenvector has at least one \textit{chiral} vertex, an index whose value changes between the left and right eigenvector.
    So, there must exist a valid scaling for the left common dominant eigenvectors and the right common dominant eigenvectors. Recall from the proof of Lemma \ref{colorType} that this 2x2 partition gives the bipartite coloring of $(\Gamma, \Delta)$.
    In addition, as observed in Proposition \ref{eigenvectors}, the right dominant eigenvector for $C_n$ is the left dominant eigenvector of $B_n$ and vice versa, so we may investigate them both in the same case.

    \begin{table}
    \tiny
    \begin{tabular}{|c|l|l|}
        \hline
        Type & Color Set $\circ$ & Color Set $\bullet$\\
        \hline
        $A_3$ & \{1.000, 1.000\}& \{1.000\}\\
        $A_4$ & $\{1.000, 1.618\}$ & $\{1.000, 1.618\}$\\
        $A_5$ & $\{1.000, 1.000, 2.000\}$ & \{1.000, 1.000\}\\
        $A_7$ & \{1.000, 1.000, 2.414, 2.414\} & \{1.000, 1.000, 1.414\}\\
        $B_2$ & \{1.000\} & \{\textbf{1.000}\}\\
        & \{1.000\} & \{\textbf{1.000}\}\\
        $B_3$ & \{1.000, \textbf{1.000}\} & \{1.000\}\\
        & \{1.000, \textbf{2.000}\} & \{1.000\}\\
        $B_4$ & \{1.000, 2.414\} & \{\textbf{1.000}, 1.414\}\\
        & \{1.000, 2.414\} & \{1.000, \textbf{1.414}\}\\
        $B_5$ & \{1.000, \textbf{1.618}, 2.618\} & \{1.000, 1.618\}\\
        & \{1.000, 2.618, \textbf{3.236}\} & \{1.000, 1.618\}\\
        $B_6$ & \{1.000, 2.732, 3.732\} & \{1.000, \textbf{1.000}, 1.732\}\\
        & \{1.000, 2.732, 3.732\} & \{1.000, 1.732, \textbf{2.000}\}\\
        $B_7$ & \{1.000, \textbf{2.247}, 2.802, 4.049\} & \{1.000, 1.802, 2.247\}\\
        & \{1.000, 2.802, 4.049, \textbf{4.494}\} & \{1.000, 1.802, 2.247\}\\
        $B_8$ & \{1.000, 2.848, 4.262, 5.027\} & \{1.000, \textbf{1.307}, 1.848, 2.414\}\\
        & \{1.000, 2.848, 4.262, 5.027\} & \{1.000, 1.848, 2.414, \textbf{2.613}\}\\
        $B_9$ & \{1.000, 2.879, \textbf{2.879}, 4.411, 5.411\} & \{1.000, 1.879, 2.532, 2.879\}\\
        & \{1.000, 2.879, 4.411, 5.411, \textbf{5.759}\} & \{1.000, 1.879, 2.532, 2.879\}\\
        $C_3$ & \{1.000, \textbf{2.000}\} & \{1.000\}\\
        & \{1.000, \textbf{1.000}\} & \{1.000\}\\
        $C_4$ & \{1.000, 2.414\} & \{1.000, \textbf{1.414}\}\\
        & \{1.000, 2.414\} & \{\textbf{1.000}, 1.414\}\\
        $C_5$ & \{1.000, 2.618, \textbf{3.236}\} & \{1.000, 1.618\}\\
        & \{1.000, \textbf{1.618}, 2.618\} & \{1.000, 1.618\}\\
        $C_6$ & \{1.000, 2.732, 3.732\} & \{1.000, 1.732, \textbf{2.000}\}\\
        & \{1.000, 2.732, 3.732\} & \{1.000, \textbf{1.000}, 1.732\}\\
        $C_7$ & \{1.000, 2.802, 4.049, \textbf{4.494}\} & \{1.000, 1.802, 2.247\}\\
        & \{1.000, \textbf{2.247}, 2.802, 4.049\} & \{1.000, 1.802, 2.247\}\\
        $C_8$ & \{1.000, 2.848, 4.262, 5.027\} & \{1.000, 1.848, 2.414, \textbf{2.613}\}\\
        & \{1.000, 2.848, 4.262, 5.027\} & \{1.000, \textbf{1.307}, 1.848, 2.414\}\\
        $C_9$ & \{1.000, 2.879, 4.411, 5.411, \textbf{5.759}\} & \{1.000, 1.879, 2.532, 2.879\}\\
        & \{1.000, 2.879, \textbf{2.879}, 4.411, 5.411\} & \{1.000, 1.879, 2.532, 2.879\}\\
        $D_4$ & \{1.000, 1.000, 1.000\} & \{1.000\}\\
        $D_5$ & \{1.000, 1.000, 1.414\} & \{1.000, 2.414\}\\
        $D_6$ & \{1.000, 1.618, 1.618, 2.618\} & \{1.000, 1.618\}\\
        $D_7$ & \{1.000, 1.000, 1.000, 1.732\} & \{1.000, 2.732, 3.732\}\\
        $E_6$ & \{1.000, 1.000, 2.732\} & \{1.000, 1.366, 1.366\}\\
        $E_7$ & \{1.000, 1.879, 2.532, 2.879\} & \{1.000, 1.532, 2.879\}\\
        $E_8$ & \{1.000, 1.618, 2.956, 4.783\} & \{1.000, 1.209, 1.618, 1.956\}\\
        $F_4$ & \{1.000, \textbf{1.366}\} & \{\textbf{1.000}, 2.732\}\\
        & \{1.000, \textbf{2.732}\} & \{\textbf{1.000}, 1.366\}\\
        $G_2$ & \{1.000\} & \{\textbf{1.000}\}\\
        & \{1.000\} & \{\textbf{1.000}\}\\
        \hline

    \end{tabular}
    \caption{Color set projections for dominant eigenvectors of Dynkin diagrams. For non-ADE types, the first row is the left eigenvector and the second row is the right eigenvector. In each non-ADE type, the boldface numbers are chiral vertices.}
    \label{table:2}
    \end{table}

    \noindent
    \textbf{Case 1: $\{B_3(C_3), A_5\}$}
    \begin{adjustwidth}{2.5em}{0pt}
        The following are the matrices of color sets for the left and right eigenvectors:
        \begin{align*}
            \begin{bmatrix}
                \{1, 1\} & \{1, 1, 2\}\\
                \{1\} & \{1, 1\}
            \end{bmatrix},\hspace{1cm}\begin{bmatrix}
                \{1, 2\} & \{1, 1, 2\}\\
                \{1\} & \{1, 1\}
            \end{bmatrix}.
        \end{align*}
        By Lemma \ref{non-ADEDoubleBindingComponent}, $\Delta_1, \Delta_2$ must contain a non-ADE component and therefore its color type must be in our list above.
        The only possibility is for $\Delta_1, \Delta_2$ to also be of type $\{B_3(C_3), A_5\}$.

        In the right matrix, the cardinality of color sets forces the pairing $\{1, 2\}$ and $\{1, 1, 2\}$ for $A_5$.
        However, $\{1, 2\}$ is not a color set for $A_5$, so these types are incompatible.
    \end{adjustwidth}
    \newpage
    \noindent
    \textbf{Case 2: $\{B_4(C_4), A_7\}$}
    \begin{adjustwidth}{2.5em}{0pt}
    Below is the matrix of color sets for both the left and right eigenvectors:
        \begin{align*}
            \begin{bmatrix}
                \{1, 2.414\} & \{1, 1, 2.414, 2.414\}\\
                \{1, 1.414\} & \{1, 1, 1.414\}
            \end{bmatrix}.
        \end{align*} 
        By Lemma \ref{non-ADEDoubleBindingComponent}, $\Delta_1, \Delta_2$ must contain a non-ADE component and therefore its color type must be in our list above.
        By color type, $\Delta_1, \Delta_2$ must be of type $\{B_5(C_5), D_6\}$.
        However, in order to match the cardinality of the color type for $D_6$, the set $\{1, 1, 2.414, 2.414\}$ above must be a color set for $D_6$, which it is not.
    \end{adjustwidth}
    \noindent
    \textbf{Case 3: $\{B_6(C_6), E_6\}$}
    \begin{adjustwidth}{2.5em}{0pt}
        Below are the matrices of color sets for the left and right eigenvectors:
        \begin{align*}
            \begin{bmatrix}
                \{1, 2.732, 3.732\} & \{1, 1, 2.732\}\\
                \{1, 1, 1.732\} & \{1, 1.366, 1.366\}
            \end{bmatrix},\hspace{1cm}\begin{bmatrix}
                \{1, 2.732, 3.732\} & \{1, 1, 2.732\}\\
                \{1, 1.732, 2\} & \{1, 1.366, 1.366\}
            \end{bmatrix}.
        \end{align*}
        By Lemma \ref{non-ADEDoubleBindingComponent}, $\Delta_1, \Delta_2$ cannot both be ADE.
        By color type, if $\Delta_1\not\cong\Delta_2$, then $\Delta_1, \Delta_2$ must be of type $\{B_6(C_6), E_6\}$ or $\{B_6, C_6\}$. But neither $\{1, 1, 2.732\}$ nor $\{1, 1.366, 1.366\}$ is a color set for $B_6(C_6)$, so both cases are incompatible.
        
        If $\Delta_1\cong \Delta_2$, then $\Delta_1\cong \Delta_2\in \{B_6, C_6\}$.
        By the same argument above, these are both incompatible types.
    \end{adjustwidth}
    \noindent
    \textbf{Case 4: $\{B_9(C_9), E_7\}$}
    \begin{adjustwidth}{2.5em}{0pt}
        Consider the right eigenvector color sets below:
        \begin{align*}
            \begin{bmatrix}
                \{1, 2.879, 4.411, 5.411, 5.759\} & \{1, 1.879, 2.532, 2.879\}\\
                \{1, 1.879, 2.532, 2.879\} & \{1, 1.532, 2.879\}
            \end{bmatrix}.
        \end{align*}
        One of $\Delta_1, \Delta_2$ must be non-ADE, so by color type, $\Delta_1, \Delta_2$ must be of type $\{B_9(C_9), E_7\}$. 
        The color set cardinalities force an exact pairing, yielding the biagram depicted in Figure \ref{fig:BCE}. However, this is not admissible.
    \end{adjustwidth}
\begin{figure}
\scalebox{0.5}{
\begin{tikzpicture}
\coordinate (v0x0) at (0.00, 0.00);
\coordinate (v0x1) at (0.00, 2.00);
\coordinate (v0x2) at (0.00, 4.00);
\coordinate (v0x3) at (0.00, 6.00);
\coordinate (v0x4) at (2.00, 7.00);
\coordinate (v0x5) at (3.50, 5.00);
\coordinate (v0x6) at (3.50, 3.00);
\coordinate (v0x7) at (3.50, 1.00);
\coordinate (v0x8) at (3.50, -1.00);
\coordinate (v1x0) at (6.25, 0.50 + 2.00);
\coordinate (v1x1) at (8.00, 0.50 + 7.00);
\coordinate (v1x2) at (6.25, 0.50 + 4.00);
\coordinate (v1x3) at (8.00, 0.50 + 5.00);
\coordinate (v1x4) at (9.50, 0.50 + 3.00);
\coordinate (v1x5) at (9.50, 0.50 + 1.00);
\coordinate (v1x6) at (9.50, 0.50 + -1.00);
\foreach \i in {0, 1, 2, 3, 4, 5, 6}
{
    \pgfmathtruncatemacro{\j}{\i + 1};
    \draw[color=red, line width=0.75mm] (v0x\i) to[] (v0x\j);
}
\draw[color=red, line width=0.75mm] (3.5, -0.11) to[] (3.72, 0.11) ;
\draw[color=red, line width=0.75mm] (3.5, -0.11) to[] (3.28, 0.11) ;
\draw[color=red, line width=0.75mm] (3.57, 1.0) to[] (3.57, -1.0) ;
\draw[color=red, line width=0.75mm] (3.43, 1.0) to[] (3.43, -1.0) ;
\draw[color=red, line width=0.75mm] (v1x0) to[] (v1x2);
\draw[color=red, line width=0.75mm] (v1x2) to[] (v1x3);
\draw[color=red, line width=0.75mm] (v1x3) to[] (v1x4);
\draw[color=red, line width=0.75mm] (v1x4) to[] (v1x5);
\draw[color=red, line width=0.75mm] (v1x5) to[] (v1x6);
\draw[color=red, line width=0.75mm] (v1x1) to[] (v1x3);
\draw[color=blue, line width=0.75mm] (v0x0) to[] (v1x6);
\draw[color=blue, line width=0.75mm] (v0x1) to[] (v1x5);
\draw[color=blue, line width=0.75mm] (v0x2) to[] (v1x6);
\draw[color=blue, line width=0.75mm] (v0x2) to[] (v1x1);
\draw[color=blue, line width=0.75mm] (v0x3) to[] (v1x3);
\draw[color=blue, line width=0.75mm] (v0x4) to[] (v1x1);
\draw[color=blue, line width=0.75mm] (v0x4) to[] (v1x2);
\draw[color=blue, line width=0.75mm] (v0x5) to[] (v1x3);
\draw[color=blue, line width=0.75mm] (v0x5) to[] (v1x0);
\draw[color=blue, line width=0.75mm] (v0x6) to[] (v1x2);
\draw[color=blue, line width=0.75mm] (v0x6) to[] (v1x4);
\draw[color=blue, line width=0.75mm] (v0x7) to[] (v1x3);
\draw[color=blue, line width=0.75mm] (v0x7) to[] (v1x5);
\draw[color=blue, line width=0.75mm] (5.412, 0.434) to[] (5.72, 0.39000000000000007) ;
\draw[color=blue, line width=0.75mm] (5.412, 0.434) to[] (5.456, 0.742) ;
\draw[color=blue, line width=0.75mm] (9.542, 3.444) to[] (3.542, -1.056) ;
\draw[color=blue, line width=0.75mm] (9.458, 3.556) to[] (3.458, -0.944) ;
\foreach \i in {0, 2, 4, 6, 8}
{
    \draw[fill=white] (v0x\i.center) circle (0.2);
}

\foreach \i in {1, 5, 7}
{
    \draw[fill=black!75!white] (v0x\i.center) circle (0.2);
}
\foreach \i in {1, 2, 4}
{
    \draw[fill=black!75!white] (v1x\i.center) circle (0.2);
}
\draw[fill=white] (v1x0.center) circle (0.2);
\draw[fill=white] (v1x3.center) circle (0.2);
\draw[fill=white] (v1x5.center) circle (0.2);
\draw[fill=black!75!white] (v0x3.center) circle (0.2);
\draw (-1.25, 6.25) node [anchor=north west][inner sep=0.75pt] {\LARGE{$v_1$}};
\draw[fill=black!75!white] (v1x6.center) circle (0.2);
\draw (10.00, -0.25) node [anchor=north west][inner sep=0.75pt] {\LARGE{$v_2$}};
\end{tikzpicture}}
\caption{\label{fig:BCE} The only possible pairing of vertices for the case $\{B_9(C_9), E_7\}$. The vertices $v_1, v_2$ form a nonadmissible pair.}
\end{figure}
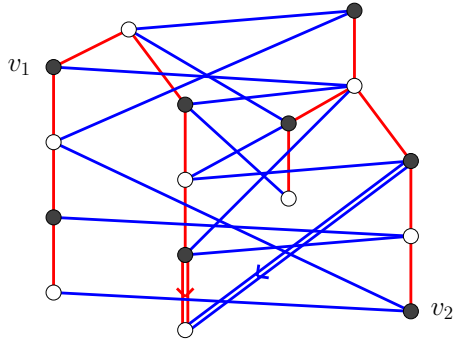
    \noindent
    \textbf{Case 5: $\{B_6(C_6), F_4\}$}
    \begin{adjustwidth}{2.5em}{0pt}
        The following are the matrices of color sets for the left and right eigenvectors:
        \begin{align*}
            \begin{bmatrix}
                \{1, 2.732, 3.732\} & \{1, 1.366\}\\
                \{1, 1, 1.732\} & \{1, 2.732\}
            \end{bmatrix}, \hspace{1cm}\begin{bmatrix}
                \{1, 2.732, 3.732\} & \{1, 1.366\}\\
                \{1, 1.732, 2\} & \{1, 2.732\}
            \end{bmatrix}.
        \end{align*}
        First we consider when $\Delta_1\not\cong\Delta_2$.
        There are no compatible color types if both $\Delta_1, \Delta_2$ are ADE (proof of \cite[Lemma~4.7]{stembridge}).
        If $\Delta_1, \Delta_2$ are not both ADE, then by color type they must be types $\{B_5, C_5\}$.
        However, $B_5(C_5)$ includes the color set $\{1, 2.618, 3.236\}$, which is not one of the color sets above.

        Now suppose $\Delta_1\cong\Delta_2$.
        Then $\Delta_1\cong\Delta_2\in \{A_5, B_5, C_5, D_5\}$. By the argument above, they cannot be $B_5$ or $C_5$. The eigenvector for $A_5$ includes $\{1, 1, 2\}$ as one of its color sets, which is not one of the sets above. Similarly, the eigenvector for $D_5$ includes the color set $\{1, 1, 1.414\}$, which is not one of the sets above. Thus, this case is incompatible.
    \end{adjustwidth}
    \noindent
    \textbf{Case 6: $\{F_4, D_7\}$}
    \begin{adjustwidth}{2.5em}{0pt}
        The following is the matrix of color sets for both the left and right eigenvectors:
        \begin{align*}
            \begin{bmatrix}
                \{1, 1.366\} & \{1, 1, 1, 1.732\}\\
                \{1, 2.732\} & \{1, 2.732, 3.732\}
            \end{bmatrix}.
        \end{align*}
        One of $\Delta_1, \Delta_2$ must be non-ADE, so by color type, $\Delta_1, \Delta_2$ must be of type $\{B_5(C_5), D_6\}$. However $\{1, 1, 1, 1.732\}$ is not a color set for $D_6$, so this is incompatible.
    \end{adjustwidth}
    \noindent
    \textbf{Case 7: $\{F_4, E_6\}$}
    \begin{adjustwidth}{2.5em}{0pt}
        The following is the matrix of color sets for both the left and right eigenvectors:
        \begin{align*}
            \begin{bmatrix}
                \{1, 1.366\} & \{1, 1, 2.732\}\\
                \{1, 2.732\} & \{1, 1.366, 1.366\}
            \end{bmatrix}.
        \end{align*}
        By Lemma \ref{non-ADEDoubleBindingComponent}, $\Delta_1, \Delta_2$ cannot both be ADE.
        By color type, if $\Delta_1\not\cong\Delta_2$, then $\Delta_1, \Delta_2$ must be of type $\{B_5, C_5\}$. However, \{1, 1.366, 1.366\} is not a color set for $B_5$ or $C_5$, so this is incompatible.

        If $\Delta_1\cong\Delta_2$, then $\Delta_1\cong\Delta_2\in \{B_5, C_5\}$. By the same argument above, this is not possible.
    \end{adjustwidth}

    \noindent
    \textbf{Case 8: $\{G_2, A_5\}$}
    \begin{adjustwidth}{2.5em}{0pt}
        The following is the matrix of color sets for both the left and right eigenvectors:
        \begin{align*}
            \begin{bmatrix}
                \{1\} & \{1, 1, 2\}\\
                \{1\} & \{1, 1\}
            \end{bmatrix}.
        \end{align*}
        One of $\Delta_1, \Delta_2$ must be non-ADE, so by color type, $\Delta_1, \Delta_2$ must be of type $\{B_3(C_3), D_4\}$. However, $D_4$ has color set $\{1, 1, 1\}$, which is not one of the sets above.
    \end{adjustwidth}
    So, we have successfully eliminated all potential type pairings not listed in Theorem \ref{doubleBindingsThm}.
    We now show that the remaining type pairings $\{B_n, C_n\}, \{B_n(C_n), D_{n+1}\}, \{B_3(C_3), G_2\}$, and $\{G_2, D_4\}$ result in the Dynkin biagrams listed in Theorem \ref{doubleBindingsThm}.

    Recall that every non-ADE Dynkin diagram has a chiral eigenvector. 
    We say a color set is \textit{chiral} if it changes values between the left and right eigenvectors.
    For this next section, we use a stronger assumption of the pairing of color sets that takes chiral eigenvectors into account. 
    Because of this, we use the notation $v_{\ell}(v_r)$ to indicate the values for the left and right eigenvector, so that one can more easily identify the chiral color sets.
    We call a chiral color set \textit{distinguishable} if its left and right counterparts are not scalar multiples of each other under any permutation of the entries.

    Notice that a double binding is not just a pairing of color sets for $\Gamma$ and $\Delta$; it is a pairing that must be consistent with both the left and right eigenvectors.
    Because of this, if $\Delta$ has a distinguishable chiral color set, it cannot be dealt a pairing of color sets that are both nondistinguishable chiral color sets.

    For example, $B_n$ and $C_n$ have a distinguishable chiral color set for all $n\neq 2, 4$, and as a result $B_n$ and $C_n$ are distinguishable from each other by their color sets in these cases.
    To see this is true, consider the color sets of $C_n$ and $B_n$ from Proposition \ref{eigenvectors},
    \begin{align*}
        \begin{bmatrix}
            \{v_1, \dots, v_k(v_k / 2)\} & \{v_1, \dots, v_k / 2 (v_k)\}\\
            \{r_1, \dots, r_s\} & \{r_1, \dots, r_s\}
        \end{bmatrix},
    \end{align*}
    where $v_1 < v_2 < \dots < v_k$.
    The chiral color sets $\{v_1, \dots, v_k(v_k / 2)\}$ are distinguishable whenever $\frac{v_k}{2} \geq v_1$, as changing sides preserves the minimum element of the set.
    
    When $n$ is odd, $v_1 = \cos\left(\frac{(n-1)\pi}{2n}\right)$ and $v_k = 1$ are the minimum and maximum elements of the chiral color set.
    Then,
    \[\frac{1}{2} \geq \cos\left(\frac{(n-1)\pi}{2n}\right) \iff \frac{(n-1)\pi}{2n} \geq \frac{\pi}{3} \iff n \geq 3,\]
    which is always true.
    
    When $n$ is even, $v_1 = \cos\left(\frac{(n-2)\pi}{2n}\right)$ and $v_k = 1$ are the minimum and maximum elements of the chiral color set.
    Then,
    \[\frac{1}{2} \geq \cos\left(\frac{(n-2)\pi}{2n}\right) \iff \frac{(n-2)\pi}{2n} \geq \frac{\pi}{3} \iff n \geq 6,\]
    which leaves only the case $n = 2, 4$.

    Now we are ready to show the remaining type pairings are those listed in Theorem \ref{doubleBindingsThm}.
    
    \noindent
    \textbf{Case 1: $\{B_n, C_n\}$, $n \geq 3$}
    \begin{adjustwidth}{2.5em}{0pt}
        First consider when $\Delta_1\cong\Delta_2$, $n\neq 4$.
        Then, $\Delta_1\cong\Delta_2\in \{A_n, B_n, C_n, D_{2k+1}, E_6, E_7, E_8\}$.

        \noindent
        Since $B_n$ and $C_n$ have distinguishable chiral color sets, $\Delta$ must also have distinguishable chiral color sets, so $\Delta_1\cong\Delta_2\in \{B_n, C_n\}$.
        So we have the following pairing:
        \begin{itemize}
            \item 
                $\Delta_1: \{v_1, \dots, v_k(v_k / 2)\}\sqcup \{r_1, \dots, r_s\}$,
            \item
                $\Delta_2: \{v_1, \dots, v_k / 2 (v_k)\}\sqcup \{r_1, \dots, r_s\}$.
        \end{itemize}
        However, $B_n$ and $C_n$ have distinct chiral color sets, so it is impossible for $\Delta_1\cong \Delta_2$.

        When $n = 4$, we have the possibilities $\Delta_1\cong\Delta_2\in \{A_4, B_4, C_4, F_4\}$.
        $A_4$ has the color set $\{1, 1.618\}$ and $F_4$ has the color set $\{1, 1.366\}$, neither of which is a color set of $B_4$ or $C_4$.
        Assigning $\Delta_1\cong\Delta_2\cong B_4$ or $\Delta_1\cong\Delta_2\cong C_4$ yields two Dynkin biagrams which are not admissible (See Figure \ref{fig:BBCCnonadmissible}).
        So we must have $\Delta_1\not\cong\Delta_2$.
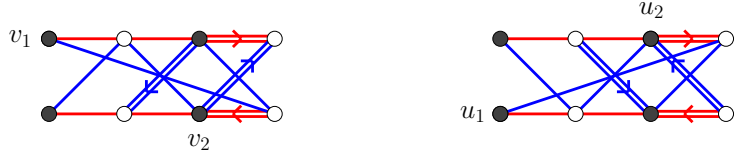
\begin{figure}
\scalebox{0.5}{
\begin{tikzpicture}
\foreach \k in {0, 1}
{
    \coordinate (v\k x0) at (12 * \k + 0.00, 0.00);
    \coordinate (v\k x1) at (12 * \k + 2.00, 0.00);
    \coordinate (v\k x2) at (12 * \k + 4.00, 0.00);
    \coordinate (v\k x3) at (12 * \k + 6.00, 0.00);
    \coordinate (v\k x4) at (12 * \k + 0.00, 2.00);
    \coordinate (v\k x5) at (12 * \k + 2.00, 2.00);
    \coordinate (v\k x6) at (12 * \k + 4.00, 2.00);
    \coordinate (v\k x7) at (12 * \k + 6.00, 2.00);
    \draw[color=red, line width=0.75mm] (v\k x0) to[] (v\k x1);
    \draw[color=red, line width=0.75mm] (v\k x1) to[] (v\k x2);
    \draw[color=red, line width=0.75mm] (v\k x4) to[] (v\k x5);
    \draw[color=red, line width=0.75mm] (v\k x5) to[] (v\k x6);
    \draw[color=red, line width=0.75mm] (12 * \k +  5.11, 2.0) to[] (12 * \k +  4.89, 2.22) ;
    \draw[color=red, line width=0.75mm] (12 * \k +  5.11, 2.0) to[] (12 * \k +  4.89, 1.78) ;
    \draw[color=red, line width=0.75mm] (12 * \k +  4.0, 2.07) to[] (12 * \k +  6.0, 2.07) ;
    \draw[color=red, line width=0.75mm] (12 * \k +  4.0, 1.93) to[] (12 * \k +  6.0, 1.93) ;
    \draw[color=red, line width=0.75mm] (12 * \k +  4.89, 0.0) to[] (12 * \k +  5.11, -0.22) ;
    \draw[color=red, line width=0.75mm] (12 * \k +  4.89, 0.0) to[] (12 * \k +  5.11, 0.22) ;
    \draw[color=red, line width=0.75mm] (12 * \k +  6.0, -0.07) to[] (12 * \k +  4.0, -0.07) ;
    \draw[color=red, line width=0.75mm] (12 * \k +  6.0, 0.07) to[] (12 * \k +  4.0, 0.07) ;
}
\draw[color=blue, line width=0.75mm] (v0x0) to[] (v0x5);
\draw[color=blue, line width=0.75mm] (v0x2) to[] (v0x5);
\draw[color=blue, line width=0.75mm] (v0x3) to[] (v0x4);
\draw[color=blue, line width=0.75mm] (v0x3) to[] (v0x6);
\draw[color=blue, line width=0.75mm] (5.411115079263853, 1.4111150792638534) to[] (5.0999880955417725, 1.4111150792638536) ;
\draw[color=blue, line width=0.75mm] (5.411115079263853, 1.4111150792638534) to[] (5.411115079263854, 1.0999880955417727) ;
\draw[color=blue, line width=0.75mm] (3.9505025253169417, 0.049497474683058325) to[] (5.950502525316941, 2.0494974746830583) ;
\draw[color=blue, line width=0.75mm] (4.049497474683059, -0.049497474683058325) to[] (6.049497474683059, 1.9505025253169417) ;
\draw[color=blue, line width=0.75mm] (2.588884920736146, 0.5888849207361464) to[] (2.900011904458227, 0.5888849207361464) ;
\draw[color=blue, line width=0.75mm] (2.588884920736146, 0.5888849207361464) to[] (2.5888849207361466, 0.9000119044582273) ;
\draw[color=blue, line width=0.75mm] (4.049497474683059, 1.9505025253169417) to[] (2.0494974746830583, -0.049497474683058325) ;
\draw[color=blue, line width=0.75mm] (3.9505025253169417, 2.0494974746830583) to[] (1.9505025253169417, 0.049497474683058325) ;

\draw[color=blue, line width=0.75mm] (v1x0) to[] (v1x7);
\draw[color=blue, line width=0.75mm] (v1x2) to[] (v1x7);
\draw[color=blue, line width=0.75mm] (v1x1) to[] (v1x4);
\draw[color=blue, line width=0.75mm] (v1x1) to[] (v1x6);
\draw[color=blue, line width=0.75mm] (15.411115079263855, 0.5888849207361464) to[] (15.411115079263853, 0.9000119044582273) ;
\draw[color=blue, line width=0.75mm] (15.411115079263855, 0.5888849207361464) to[] (15.099988095541773, 0.5888849207361464) ;
\draw[color=blue, line width=0.75mm] (14.049497474683058, 2.0494974746830583) to[] (16.049497474683058, 0.049497474683058325) ;
\draw[color=blue, line width=0.75mm] (13.950502525316942, 1.9505025253169417) to[] (15.950502525316942, -0.049497474683058325) ;
\draw[color=blue, line width=0.75mm] (16.58888492073615, 1.4111150792638534) to[] (16.588884920736145, 1.0999880955417727) ;
\draw[color=blue, line width=0.75mm] (16.58888492073615, 1.4111150792638534) to[] (16.90001190445823, 1.4111150792638536) ;
\draw[color=blue, line width=0.75mm] (17.950502525316942, -0.049497474683058325) to[] (15.950502525316942, 1.9505025253169417) ;
\draw[color=blue, line width=0.75mm] (18.049497474683058, 0.049497474683058325) to[] (16.049497474683058, 2.0494974746830583) ;

\foreach \i in {0,1}
{
    \foreach \j in {0, 2, 4, 6}
    {
        \pgfmathtruncatemacro{\k}{\j + 1};
        \draw[fill=black!75!white] (v\i x\j.center) circle (0.2);
        \draw[fill=white] (v\i x\k.center) circle (0.2);
    }
}
\draw[fill=black!75!white] (v0x4.center) circle (0.2);
\draw (-1.10, 2.25) node [anchor=north west][inner sep=0.75pt] {\LARGE{$v_1$}};
\draw[fill=black!75!white] (v0x2.center) circle (0.2);
\draw (3.65, -0.50) node [anchor=north west][inner sep=0.75pt] {\LARGE{$v_2$}};
\draw[fill=black!75!white] (v1x0.center) circle (0.2);
\draw (10.90, 0.25) node [anchor=north west][inner sep=0.75pt] {\LARGE{$u_1$}};
\draw[fill=black!75!white] (v1x6.center) circle (0.2);
\draw (15.65, 3.00) node [anchor=north west][inner sep=0.75pt] {\LARGE{$u_2$}};
\end{tikzpicture}}
\caption{\label{fig:BBCCnonadmissible} The potential vertex pairings when $\Gamma_1\cong B_4, \Gamma_2\cong C_4$, and $\Delta_1,\Delta_2\cong B_4$ (left) or $\Delta_1,\Delta_2\cong C_4$ (right). The vertices $v_1, v_2$ and $u_1, u_2$ form nonadmissible pairs.}
\end{figure}
        
        If $\Delta_1\not\cong\Delta_2$, we first note there are no potential non-isomorphic ADE color types that match (proof of \cite[Lemma~4.7]{stembridge}), so $\Delta_1$ and $\Delta_2$ cannot both be ADE.
        Thus they must be in the list above as $\{B_6(C_6), E_6\}$, $\{F_4, E_6\}$, $\{B_6(C_6), F_4\}$ or $\{B_n, C_n\}$, of which only $\{F_4, E_6\}$ and $\{B_n, C_n\}$ have not yet been eliminated as incompatible types.
        $E_6$ has the color set $\{1, 1, 2.732\}$, which is not a color set of $B_6(C_6)$, and $F_4$ has the color set $\{1, 1.366\}$, which is not a color set of $B_5(C_5)$.
        Thus the only possibility for $\Delta$ is $\{B_n, C_n\}$.
        Since each eigenvector of $B_n, C_n$ has a distinguished chiral color set, this forces the same pairing above, resulting in the admissible Dynkin biagram $B_n\bowtie C_n$, depicted in Figure \ref{fig:BC}. 
        When $n = 4$, there is an alternate pairing, yielding the exceptional double binding $B_4\boxtimes C_4$ (see Figure \ref{fig:BCBCexceptional}).
    \end{adjustwidth}
        
    \noindent
    \textbf{Case 2: $\{B_n(C_n), D_{n+1}\}$}
    \begin{adjustwidth}{2.5em}{0pt}
        By Proposition \ref{eigenvectors}, the eigenvectors of $C_n$ and $D_{n+1}$ have the following color sets:
        \begin{align*}
            \begin{bmatrix}
                \{v_1, \dots, v_k(v_k / 2)\} & \{v_1, \dots, v_{k-1}, v_k / 2, v_k / 2\}\\
                \{r_1, \dots, r_s\} & \{r_1, \dots, r_s\}
            \end{bmatrix},
        \end{align*}
        where $v_1 < v_2 < \dots < v_k$.
        By Lemma \ref{non-ADEDoubleBindingComponent}, $\Delta_1$ and $\Delta_2$ must be a listed non-ADE pair.
        The only possibilities for compatible types are $\{B_n(C_n), D_{n+1}\}$ and $\{B_3(C_3), G_2\}$.
        If $\Delta_1, \Delta_2$ are $\{B_3(C_3), G_2\}$, then $\Gamma_1, \Gamma_2$ must be $\{B_2, A_3\}$.
        However, $A_3$ has the color set $\{1, 1\}$, which is not a left color set of $C_3$.
        So $\Delta_1, \Delta_2$ must be $\{B_n(C_n), D_{n+1}\}$.

        Consider first when $n\neq 4$.
        Since $D_{n+1}$ does not have a distinguishable chiral color set, we have the following pairing of color sets:
        \begin{itemize}
            \item 
                $D_{n+1}: \{v_1, \dots, v_{k-1}, v_k / 2, v_k / 2\}\sqcup \{r_1, \dots, r_s\}$,
            \item
                $B_n(C_n): \{v_1, \dots, v_k / 2 (v_k)\}\sqcup \{r_1, \dots, r_s\}$.
        \end{itemize}
        As $B_n$ and $C_n$ have distinct chiral color sets for $n\neq 4$, $\Delta_1, \Delta_2$ must be $\{C_n, D_{n+1}\}$.
        This uniquely determines the biagram to be $C_n\ltimes D_{n+1}$. Similarly, when we start with $\{B_n, D_{n+1}\}$, we get $B_n\ltimes D_{n+1}$. Both of these biagrams are depicted in Figure \ref{fig:BDCD}.

        When $n = 4$, the values of the color sets determine the same pairing as above.
        However, there is an alternate assignment of vertices for $\Gamma_1\cong C_n, \Delta_1\cong B_n$, which results a nonadmissible Dynkin biagram.
        Similarly, when we have $\Gamma_1\cong B_n, \Delta_1\cong C_n$, we get a nonadmissible Dynkin biagram depicted in Figure \ref{fig:BCDnonadmissible}.
    \end{adjustwidth}

        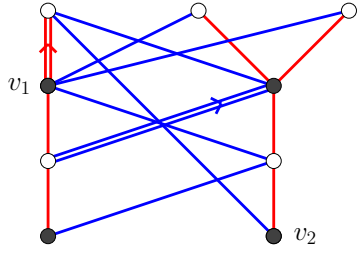
\begin{figure}
        \scalebox{0.5}{
        \begin{tikzpicture}
            \coordinate (v0x3) at (0.00, 6.00);
            \coordinate (v1x3) at (4.00, 6.00);
            \coordinate (v1x4) at (8.00, 6.00);
            \foreach \i in {0, 1}
            {
            \coordinate (v\i x0) at (6 * \i, 0.00);
            \coordinate (v\i x1) at (6 * \i, 2.00);
            \coordinate (v\i x2) at (6 * \i, 4.00);
            \draw[color=red,line width=0.75mm] (v\i x0) to[] (v\i x1);
            \draw[color=red,line width=0.75mm] (v\i x1) to[] (v\i x2);
            }
            \draw[color=red, line width=0.75mm] (0.0, 5.11) to[] (-0.22, 4.89) ;
            \draw[color=red, line width=0.75mm] (0.0, 5.11) to[] (0.22, 4.89) ;
            \draw[color=red, line width=0.75mm] (-0.07, 4.0) to[] (-0.07, 6.0) ;
            \draw[color=red, line width=0.75mm] (0.07, 4.0) to[] (0.07, 6.0) ;
            \draw[color=red,line width=0.75mm] (v1x2) to[] (v1x3);
            \draw[color=red,line width=0.75mm] (v1x2) to[] (v1x4);
            \draw[color=blue,line width=0.75mm] (v0x0) to[] (v1x1);
            \draw[color=blue,line width=0.75mm] (v0x2) to[] (v1x1);
            \draw[color=blue,line width=0.75mm] (v0x2) to[] (v1x3);
            \draw[color=blue,line width=0.75mm] (v0x2) to[] (v1x4);
            \draw[color=blue,line width=0.75mm] (v0x3) to[] (v1x2);
            \draw[color=blue,line width=0.75mm] (v0x3) to[] (v1x0);
            \draw[color=blue, line width=0.75mm] (4.604355162785557, 3.534785054261852) to[] (4.326074728690739, 3.673925271309261) ;
            \draw[color=blue, line width=0.75mm] (4.604355162785557, 3.534785054261852) to[] (4.465214945738148, 3.2565046201670347) ;
            \draw[color=blue, line width=0.75mm] (-0.02213594362117866, 2.066407830863536) to[] (5.977864056378821, 4.066407830863536) ;
            \draw[color=blue, line width=0.75mm] (0.02213594362117866, 1.933592169136464) to[] (6.022135943621179, 3.933592169136464) ;
            \foreach \i in {0, 1}
            {
                \draw[fill=black!75!white] (v\i x0.center) circle (0.2);
                \draw[fill=white] (v\i x1.center) circle (0.2);
                \draw[fill=black!75!white] (v\i x2.center) circle (0.2);
                \draw[fill=white] (v\i x3.center) circle (0.2);
            }
            \draw[fill=white] (v1x4.center) circle (0.2);
            \draw[fill=black!75!white] (v0x2.center) circle (0.2);
            \draw (-1.1, 4.25) node [anchor=north west][inner sep=0.75pt] {\LARGE{$v_1$}};
            \draw[fill=black!75!white] (v1x0.center) circle (0.2);
            \draw (6.50, 0.25) node [anchor=north west][inner sep=0.75pt] {\LARGE{$v_2$}};
        \end{tikzpicture}}
        \caption{\label{fig:BCDnonadmissible} The potential pairing of vertices for $\Gamma_1\cong B_4, \Gamma_2\cong D_5, \Delta_1\cong C_4, \Delta_2\cong D_5$. The vertices $v_1, v_2$ form a nonadmissible pair.}
        \end{figure}
        
    \noindent
    \textbf{Case 3: $\{B_3(C_3), G_2\}$}
    \begin{adjustwidth}{2.5em}{0pt}        
        The eigenvectors of $C_3$ and $G_2$ have the following color sets:
        \begin{align*}
            \begin{bmatrix}
                \{1, 2(1)\} & \{1\}\\
                \{1\} & \{1\}
            \end{bmatrix}.
        \end{align*}
        We first note there are no potential non-isomorphic ADE color types that match (proof of \cite[Lemma~4.7]{stembridge}), so $\Delta_1$ and $\Delta_2$ cannot both be ADE.
        Thus they must be in the list above as either $\{B_3(C_3), G_2\}$ or $\{B_2, A_3\}$.
        Notice that $A_3$ has the color set \{1, 1\}, which is not a left color set of $C_3$.
        So $\Delta_1, \Delta_2$ must be of type $\{B_3(C_3), G_2\}$.
        Since $B_3$ and $C_3$ have distinct chiral color sets, $\Delta_1, \Delta_2$ must be of type $\{C_3, G_2\}$.
        There are two choices for the pairings of sets, so with Lemma \ref{flippingLem}, this yields two different biagrams for $\{C_3, G_2\}$, $C_3\bowtie_{1, 2} G_2$.
        Similarly, when we start with $\{B_3, G_2\}$, we get $B_3\bowtie_{1, 2} G_2$ (see Figure \ref{fig:BGCG}).
    \end{adjustwidth}
        
    \noindent
    \textbf{Case 4: $\{G_2, D_4\}$}
    \begin{adjustwidth}{2.5em}{0pt}
        The following are the color sets for both left and right eigenvectors:
        \begin{align*}
            \begin{bmatrix}
                \{1\} & \{1, 1, 1\}\\
                \{1\} & \{1\}
            \end{bmatrix}.
        \end{align*}
        This cardinality of color sets is unique to $D_4$, so $\Delta_1, \Delta_2$ must also be $\{G_2, D_4\}$.
        There are two different pairings, producing $G_2\ltimes_1 D_4$ and $G_2\ltimes_2 D_4$ (see Figure \ref{fig:GD}).
    \end{adjustwidth}

\end{proof}

The following lemma, combined with Lemma \ref{nonisom}, completes the proof of Theorem \ref{doubleBindingsThm}.
\begin{lemma}
    If $(\Gamma, \Delta)$ is an admissible double binding such that the two components of $\Gamma$ are isomorphic, then $(\Gamma, \Delta)$ is either a parallel binding or a twist.
\end{lemma}

\begin{proof}
    We are given $\Gamma_1\cong \Gamma_2$, and also only consider $\Delta_1\cong \Delta_2$, otherwise the biagram would be dual to one already classified in Lemma \ref{nonisom}.
    Let the color type of $\Gamma_1, \Gamma_2$ be $\{a, b\}$. Then the color type of $\Delta_1\cong \Delta_2$ must also be $\{a, b\}$ by Lemma \ref{colorType}.

    First consider when $\Gamma \not\cong \Delta$.
    The following is a list of all possible distinct non-ADE pairings $\{\Gamma, \Delta\}$ with the same color type:

    \begin{tabular}{m{4cm} m{4cm} m{4cm} m{4cm}}
        $\{B_n, A_n\},$ & $\{C_n, A_n\},$ & $\{B_n, C_n\},$ & $\{B_{2n+1}, D_{2n+1}\},$\\
        $\{C_{2n+1}, D_{2n+1}\},$ & $\{B_6(C_6), E_6\},$ & $\{B_7(C_7), E_7\},$ & $\{B_8(C_8), E_8\},$\\
        $\{F_4, A_4\},$ & $\{F_4, B_4(C_4)\},$ & $\{G_2, A_2\},$ & $\{G_2, B_2\}.$
    \end{tabular}\\

    We first consider the infinite cases.
    When $\Gamma_1\cong \Gamma_2\cong B_n (C_n)$, there is an eigenvector with strictly increasing values. Since $D_n$ and $A_{2k + 1}$ each have a color set with repeated values, both $\{B_{2k+1}(C_{2k+1}), A_{2k+1}\}$ and $\{B_{2k+1}(C_{2k+1}), D_{2k+1}\}$ are incompatible types. 

    When $n = 2k$, from Proposition \ref{eigenvectors} the two color sets of $A_n$ are identical and have all distinct values, so there exists an ordering of the vertices of each color set that is strictly increasing. However, even though $B_n$ and $C_n$ have increasing color sets, when $n > 2$, the relative ordering of vertices changes between sides, so there is no assignment of values to vertices that produces increasing sets in both the left and right dominant eigenvectors. So when $n > 2$, $\{B_n(C_n), A_n\}$ are incompatible types.
    When $n = 2$, $\{B_2, A_2\}$ has two potential pairings, one of which is admissible, $B_2\otimes A_2$.

    For the case of $\{B_n, C_n\}$, $n\geq 3$, we know the chiral color set for $B_n$ is distinct from the chiral color set for $C_n$ for all $n\neq 4$, so this is an incompatible pairing.
    When $n = 4$, the assignment of color set values to vertices results in the nonadmissible Dynkin biagram depicted in Figure \ref{fig:BCnonadmissible}.

\begin{figure}
\scalebox{0.5}{
\begin{tikzpicture}
\foreach \i in {0, 1}
{
\coordinate (v\i x0) at (6 * \i, 0.00);
\coordinate (v\i x1) at (6 * \i, 2.00);
\coordinate (v\i x2) at (6 * \i, 4.00);
\coordinate (v\i x3) at (6 * \i, 6.00);
\draw[color=red,line width=0.75mm] (v\i x0) to[] (v\i x1);
\draw[color=red,line width=0.75mm] (v\i x1) to[] (v\i x2);
\draw[color=red, line width=0.75mm] (6 * \i + 0.0, 5.11) to[] (6 * \i + -0.22, 4.89) ;
\draw[color=red, line width=0.75mm] (6 * \i + 0.0, 5.11) to[] (6 * \i + 0.22, 4.89) ;
\draw[color=red, line width=0.75mm] (6 * \i + -0.07, 4.0) to[] (6 * \i + -0.07, 6.0) ;
\draw[color=red, line width=0.75mm] (6 * \i + 0.07, 4.0) to[] (6 * \i + 0.07, 6.0) ;
}
\draw[color=blue,line width=0.75mm] (v0x0) to[] (v1x3);
\draw[color=blue,line width=0.75mm] (v0x2) to[] (v1x3);
\draw[color=blue,line width=0.75mm] (v0x3) to[] (v1x2);
\draw[color=blue,line width=0.75mm] (v0x3) to[] (v1x0);
\draw[color=blue, line width=0.75mm] (4.604355162785557, 3.534785054261852) to[] (4.326074728690739, 3.673925271309261) ;
\draw[color=blue, line width=0.75mm] (4.604355162785557, 3.534785054261852) to[] (4.465214945738148, 3.2565046201670347) ;
\draw[color=blue, line width=0.75mm] (-0.02213594362117866, 2.066407830863536) to[] (5.977864056378821, 4.066407830863536) ;
\draw[color=blue, line width=0.75mm] (0.02213594362117866, 1.933592169136464) to[] (6.022135943621179, 3.933592169136464) ;
\draw[color=blue, line width=0.75mm] (1.3956448372144434, 3.534785054261852) to[] (1.5347850542618522, 3.2565046201670347) ;
\draw[color=blue, line width=0.75mm] (1.3956448372144434, 3.534785054261852) to[] (1.673925271309261, 3.673925271309261) ;
\draw[color=blue, line width=0.75mm] (5.977864056378821, 1.933592169136464) to[] (-0.02213594362117866, 3.933592169136464) ;
\draw[color=blue, line width=0.75mm] (6.022135943621179, 2.066407830863536) to[] (0.02213594362117866, 4.066407830863536) ;
\foreach \i in {0, 1}
{
    \draw[fill=black!75!white] (v\i x0.center) circle (0.2);
    \draw[fill=white] (v\i x1.center) circle (0.2);
    \draw[fill=black!75!white] (v\i x2.center) circle (0.2);
    \draw[fill=white] (v\i x3.center) circle (0.2);
}
\draw[fill=black!75!white] (v0x0.center) circle (0.2);
\draw (-1.10, 0.25) node [anchor=north west][inner sep=0.75pt] {\LARGE{$v_1$}};
\draw[fill=black!75!white] (v1x2.center) circle (0.2);
\draw (6.5, 4.25) node [anchor=north west][inner sep=0.75pt] {\LARGE{$v_2$}};
\end{tikzpicture}}
\caption{\label{fig:BCnonadmissible} A pairing of vertices for $\Gamma_1\cong \Gamma_2\cong B_4$ and $\Delta_1\cong\Delta_2\cong C_4$. The vertices $v_1, v_2$ form a nonadmissible pair.}
\end{figure}
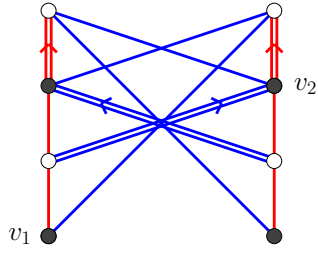

    In each of the pairings $\{B_n(C_n), E_n\}$ for $n \in \{6, 7, 8\},$ there is a color set of $E_n$ that is not a color set of $B_n(C_n)$.

    In the pairings $\{F_4, A_4\}$ and $\{F_4, B_4(C_4)\}$, $A_4$ has the color set $\{1, 1.618\}$ and $B_4(C_4)$ have the color set $\{1, 2.414\}$, neither of which is a color set of $F_4$. 
    
    For each of $\{G_2, A_2\}$ and $\{G_2, B_2\}$, there are two possible pairings to create a biagram, one of which is admissible: $G_2\otimes A_2$ and $G_2\otimes B_2$.
    \vspace{3mm}\\
    Now suppose $\Gamma \cong \Delta$. We show that this produces only twists or tensor products.
    From Stembridge's paper, we know this is true for ADE $\Gamma$ and $\Delta$.

    When $\Gamma, \Delta\cong B_2$, both assignments are admissible: $B_2\times B_2$ and $B_2\otimes B_2$.
    
    When $\Gamma, \Delta\cong B_n$ or $C_n$ for $n > 2$, the color set pairings must give one distinguishable chiral set to each $\Delta$-component, yielding $B_n\times B_n$ and $C_n\times C_n$.

    When $\Gamma, \Delta\cong F_4$, the color set values are distinct, so the color set pairings and vertex assignments are uniquely determined to produce $F_4\times F_4$.

    When $\Gamma, \Delta\cong G_2$, both assignments are admissible: $G_2\times G_2$ and $G_2\otimes G_2$.
\end{proof}

\begin{proposition}
\label{selfBindings}
    The only other unit bindings besides the double bindings are $A_2\otimes A_1\cong A_2, B_2\otimes A_1\cong B_2$, and $G_2\otimes A_1\cong G_2$.
\end{proposition}

\begin{proof}
    Any smaller unit binding can be written as $(\Gamma, \Delta)$, where $\Gamma$ is connected and $\Delta$ has at most two components. 
    If $\Delta = \Delta_1\sqcup \Delta_2$, then every edge of $\Gamma$ connects $\Delta_1$ and $\Delta_2$. 
    Moreover, every vertex of $\Delta_1$ is adjacent along a $\Gamma$-edge to a vertex in $\Delta_2$ (\cite[Lemma~2.5]{stembridge}), so $\Delta_1\sqcup \Delta_2$ is the bipartite coloring of the diagram.
    However, this means $\Delta_1, \Delta_2$ are each all one color, so we have $\Delta_1, \Delta_2\cong A_1$.

    If $\Delta$ has only one component, then by Lemma \ref{commonEigVec}, $\Gamma$ and $\Delta$ share a common dominant eigenvector. 
    By Corollary \ref{uniqueEigenvector}, $\Gamma\cong \Delta$.

    In each of the cases $\Gamma \in \{B_n$, $C_n$, $F_4$, $G_2$, $E_7$, $E_8$\}, $\Gamma$ has an eigenvector with distinct values, which uniquely determines the assignment of vertices so that $\Delta = \Gamma$. 
    However, this means that $\Delta$ shares edges with $\Gamma$, a contradiction. 

    All of the ADE types $A_n$, $D_n$ and $E_6$ have nontrivial bigraph automorphisms that fix the dominant eigenvector, but each one results in a diagram that shares edges with the original. 
    Thus there are no other unit self-bindings. 
\end{proof}

\subsection{Proof of the classification}

The idea of the proof of Theorem \ref{classThm} uses Proposition \ref{StembridgeAdmissibleBindings} and is based on a basic argument adapted from \cite{stembridge}.
Let $\Gamma$ have $k$ components, $\Gamma_1, \dots, \Gamma_{k}$. 
We define $\mathcal{C}$ to be the component graph on the vertex set $[k]$ that has an edge $\{i, j\}$ if there exists an edge of $\Delta$ that connects vertices of $\Gamma_i$ and $\Gamma_j$.
Such an edge then denotes the binding of $\Gamma_i$ and $\Gamma_j$.

Since every admissible Dynkin biagram can be formed through a series of taking duals and gluing together bindings, the class of connected admissible Dynkin biagrams is the smallest class $\mathcal{G}$ that contains
\begin{itemize}
    \item [(i)] 
        $A_1\otimes A_1$, $A_1\otimes A_2$, $A_1\otimes B_2$, $A_1\otimes G_2$ and all admissible Dynkin double bindings,
    \item [(ii)]
        all admissible Dynkin biagrams obtained by gluing together bindings in $\mathcal{G}$ along some connected component graph $\mathcal{C}$, and
    \item [(iii)]
        the dual of every biagram $(\Gamma, \Delta)$ in $\mathcal{G}$ such that $\Delta$ has at most 2 components.
\end{itemize}
The list in Theorem \ref{classThm} clearly satisfies (i). It remains to prove (ii) and (iii).

\begin{proof}[Proof of (ii)]
    Let $(\Gamma, \Delta)$ be a connected admissible Dynkin biagram whose bindings are listed in Theorem \ref{classThm} or Stembridge's classification:
    \begin{enumerate}
        \item the parallel and twisted bindings,
        \item $A_{2n-1}\ast B_n$, $A_{2n-1}\ast C_n$, $B_n\ast C_n$, $B_n\ast D_{n+1}$, $C_n\ast D_{n+1}$, $E_6\ast_{1,2} F_4$, $F_4\ast F_4$,
        \item $A_{2n-1}\ast D_{n+1}$, $E_6\ast E_6$, $D_6\ast D_6$, $E_8\ast E_8$,
        \item the double bindings $B_n\ltimes D_{n+1}$, $C_n\ltimes D_{n+1}$, $B_n\bowtie C_n$, $G_2\ltimes_{1, 2} D_4$, $B_3\bowtie_{1,2} G_2$, and $C_3\bowtie_{1,2} G_2$,
        \item the exceptional bindings $B_4\boxtimes C_4$, $D_5\boxtimes A_7,$ and $E_7\boxtimes D_{10}$.
    \end{enumerate}
    Since Stembridge has already classified the connected admissible ADE bigraphs, we may assume there is at least one non-ADE binding in $(\Gamma, \Delta)$. We only need to show that $(\Gamma, \Delta)$ is listed in Theorem \ref{classThm}.
    
    If all bindings in $\mathcal{C}$ are parallel, $(\Gamma, \Delta)$ is necessarily a tensor product. So, assume there is some nonparallel binding.
    We first note that $\mathcal{C}$ must be acyclic, as all Dynkin diagrams are acyclic (\cite[Lemma~2.5(b)]{stembridge}). So, $\mathcal{C}$ is a tree. We show $\mathcal{C}$ is necessarily a path. 
    
    Suppose $\mathcal{C}$ has a node of degree at least 3.
    Then, every vertex in the corresponding component $\Gamma_0$ has $\deg_{\Delta} \geq 3$ (\cite[Lemma~2.5(a)]{stembridge}).
    If one of the nonparallel bindings is ADE, then each $\Gamma$-component has a vertex with $\deg_{\Delta}\geq 2$. 
    So, the vertex in the nonparallel binding has two different paths to $\Gamma_0$, forming a $\Delta$-component containing at least two points of $\Gamma_0$.
    However, there is no Dynkin diagram with more than one vertex of degree 3.
    If one of the nonparallel bindings is not ADE, then there exists a vertex adjacent to a nonsimple $\Delta$-edge.
    Then, this yields a non-ADE $\Delta$-component connected to $\Gamma_0$ which therefore contains a degree 3 vertex.
    However, no non-ADE Dynkin diagrams have any degree $\geq 3$ vertices. So there must be no nodes of degree 3, and $\mathcal{C}$ is just a path.

    Now suppose there is more than one nonparallel binding in $\mathcal{C}$. Since subdiagrams of Dynkin diagrams are Dynkin diagrams, we can prune our path to be a component graph $\mathcal{C}$ with exactly two nonparallel bindings at the ends.
    Since $\Delta$ is acyclic, a maximal $\Delta$-path through distinct vertices necessarily terminates at vertices $v$ such that $\deg_{\Delta}(v) = 1$.
    These end vertices must occur in the components of $\Gamma$ at the ends of $\mathcal{C}$, since the vertices in all other components necessarily have $\deg_{\Delta}(v)\geq 2$.
    In every nonparallel binding listed above, every $\Delta$-component either contains at least two edges or is a nonsimple edge.
    So, the vertices one step away from the endpoints of a maximal $\Delta$-path either have $\deg_{\Delta}(v)\geq 3$ or are adjacent to nonsimple edges.
    Dynkin diagrams have at most one vertex of degree $\geq 3$, and non-ADE Dynkin diagrams have at most one nonsimple edge and no vertices of degree $\geq 3$.
    So if at least one nonparallel binding is non-ADE, there is at most one binding in $\mathcal{C}$, contradicting our assumption of two nonparallel bindings.
    If both nonparallel bindings are ADE, then all maximal $\Delta$-paths are of length $\leq 2$ and hence all $\Delta$-components must be isomorphic to either $A_3, B_3, C_3, D_4,$ or $G_2$.
    This is a contradiction, since every $\Delta$-component is a gluing of $\Delta$-components from the two nonparallel bindings, and therefore has at least four edges.
    
    Thus, $\mathcal{C}$ is a path with exactly one nonparallel binding.

    If our nonparallel binding is ADE, then every $\Gamma$-component contains a vertex with $\deg_{\Delta} \geq 2$. So, extending the path with bindings on either side will produce a vertex of degree 3. As no non-ADE Dynkin diagrams have vertices of degree 3, $\mathcal{C}$ cannot have any nonsimple parallel bindings.
    Therefore in this case, $\mathcal{C}$ is an ADE bigraph, which has been classified by Stembridge in \cite{stembridge}.

    Now suppose our nonparallel binding is non-ADE. 
    If in this nonparallel binding, $\Gamma_0$ contains one of the following, then $\Gamma_0$ must be an end node of $\mathcal{C}$.
    \begin{itemize}
        \item 
            a vertex $v$ with $\deg_{\Delta}(v) = 3$. Extending this path with a binding would create a degree 4 vertex.
        \item 
            a vertex $v$ such that $\deg_{\Delta} = 2$ and $v$ is adjacent to a nonsimple $\Delta$-edge. Extending this path with a binding would create a non-ADE Dynkin diagram with a degree 3 vertex.
        \item 
            a vertex $v$ adjacent to a $G_2$ $\Delta$-component. $G_2$ is not a subdiagram of any other Dynkin diagram.
    \end{itemize}
    Each of the non-ADE double bindings and twists besides $B_2\times B_2$ satisfy one of these criterion in both $\Gamma$ components. So if the nonparallel binding is a double binding or twist, then $\mathcal{C}$ consists of only one binding.

    We can explicitly give all component graphs of other non-ADE bindings by observing two things.
    First, no Dynkin diagrams have more than one nonsimple edge, so all parallel bindings in $\mathcal{C}$ must be simple parallel bindings.
    Second, in a Dynkin diagram, a nonsimple edge either occurs at the end of a path ($B_n, C_n$) or in the middle of a length 3 path ($F_4$). So, starting from a non-ADE nonparallel binding, we may either extend with simple parallel bindings arbitrarily in one direction or once in each direction.

    Below are the potential component paths for each of the remaining nonparallel non-ADE bindings:
    \begin{align*}
        B_2\times B_2 &\implies B_2\times B_2\equiv \cdots \equiv B_2, \hspace{2mm} B_2\equiv B_2\times B_2\equiv B_2,\\
        B_n\ast A_{2n-1}&\implies (BA^{m-1})_n,\hspace{2mm} (B^{m-1}A)_n,\hspace{2mm} B_n\equiv B_n \ast A_{2n-1} \equiv A_{2n-1},\\
        C_n \ast A_{2n-1}&\implies (CA^{m-1})_n, \hspace{2mm} (C^{m-1}A)_n, \hspace{2mm} C_n\equiv C_n \ast A_{2n-1} \equiv A_{2n-1},\\
        B_n\ast C_n&\implies (BC^{m-1})_n, \hspace{2mm} (B^{m-1}C)_n, \hspace{2mm} B_n\equiv B_n\ast C_n\equiv C_n,\\
        B_n\ast D_{n+1}&\implies (BD^{m-1})_n, \hspace{2mm} (B^{m-1}D)_n, \hspace{2mm} B_n\equiv B_n \ast D_{n+1} \equiv D_{n+1},\\
        C_n \ast D_{n + 1}&\implies (CD^{m-1})_n, \hspace{2mm} (C^{m-1}D)_n, \hspace{2mm} C_n\equiv C_n \ast D_{n+1} \equiv D_{n+1},\\
        F_4 \ast_1 E_6&\implies (F_4E_6^{m-1})_{1}, \hspace{2mm} (F_4^{m-1}E_6)_{1}, \hspace{2mm} 
        F_4\equiv F_4 \ast_{1} E_6 \equiv E_6,\\
        F_4 \ast_2 E_6&\implies (F_4E_6^{m-1})_{2}, \hspace{2mm} (F_4^{m-1}E_6)_{2}, \hspace{2mm} 
        F_4\equiv F_4 \ast_{2} E_6 \equiv E_6,\\
        F_4\ast F_4&\implies F_4\ast F_4\equiv \dots \equiv F_4, \hspace{2mm} F_4\equiv F_4\ast F_4\equiv F_4.
    \end{align*}
    These are all listed in Theorem \ref{classThm}.
\end{proof}

\begin{proof}[Proof of (iii)]
    Let $(\Gamma, \Delta)$ be a biagram listed in Theorem \ref{classThm} such that $\Delta$ has at most two components. 
    Since tensor products are dual closed, and all non-ADE double bindings are self-dual, we can assume $(\Gamma, \Delta)$ is one of the following biagrams:
    \begin{align*}
        B_2\times B_2 \equiv \cdots \equiv B_2 &\longleftrightarrow B_n\ast C_n,\\
        B_2\equiv B_2\times B_2\equiv B_2 &\longleftrightarrow F_4\ast F_4,\\
        B_2\ast A_3 &\longleftrightarrow B_2\ast A_3,\\
        (BA^{m-1})_2 &\longleftrightarrow B_m\ast A_{2m-1},\\
        (B^{m-1}A)_2 &\longleftrightarrow C_m\ast D_{m+1},
    \end{align*}
    \begin{align*}
        B_2\equiv B_2\ast A_3\equiv A_3 &\longleftrightarrow F_4\ast_1 E_6,\\
        B_2\ast D_3 &\longleftrightarrow B_2\ast D_3,\\
        (BD^{m-1})_2 &\longleftrightarrow C_m\ast A_{2m-1},\\
        (B^{m-1}D)_2 &\longleftrightarrow B_m\ast D_{m+1},\\
        B_2\equiv B_2\ast D_3\equiv D_3 &\longleftrightarrow F_4\ast_2 E_6.
    \end{align*}
    The right hand side lists the dual for each Dynkin biagram.
    These are all listed in Theorem \ref{classThm}.
\end{proof}

To finish the proof of the classification, it remains to check that all biagrams listed in Theorem \ref{classThm} are admissible Dynkin biagrams. 
\begin{proof}
    It is clear that all components are Dynkin diagrams. By Theorem \ref{doubleBindingsThm} and Corollary \ref{admissibleBindings}, all bindings are admissible. Every family listed in Theorem \ref{classThm} is made up of these bindings, so by Proposition \ref{StembridgeAdmissibleBindings}, all of the biagrams listed in Theorem \ref{classThm} are admissible. This concludes the proof of Theorem \ref{classThm}.
\end{proof}

\subsection{A key observation}

\begin{proposition}
\label{DynkinFromADE}
    Any admissible Dynkin biagram can be obtained from an ADE bigraph through a sequence of global flips and folds along bicolored automorphisms.
\end{proposition}

\begin{proof}
    Notice that if $\Gamma$ folds to $\Gamma'$, the tensor product $\Gamma\otimes \Delta$ folds to $\Gamma'\otimes \Delta$.
    Similarly, the twist $\Gamma\times \Gamma$ folds to $\Gamma'\times \Gamma'$.

    We have already shown in the proof of Lemma \ref{admissibleDouble} that each of the double bindings can be obtained through a series of global flips and folds along bicolored automorphisms.

    Each of the remaining bindings can also be obtained from ADE bindings through sequences of global flips and folds as follows:
    \begin{align*}
        D_{n+1}\ast A_{2n-1} &\mapsto B_n\ast A_{2n - 1},\\
        (B_n\ast A_{2n-1})^{\top} &= C_n\ast A_{2n - 1},\\
        A_{2n-1} \ast D_{n+1} &\mapsto C_n\ast D_{n + 1},\\
        (C_n\ast D_{n + 1})^{\top} &= B_n\ast D_{n + 1},\\
        C_n\ast D_{n + 1} &\mapsto B_n\ast C_n,\\
        E_6\ast E_6 &\mapsto F_4\ast_1 E_6,\\
        (F_4\ast_1 E_6)^{\top} &= F_4\ast_2 E_6,\\
        F_4\ast_1 E_6 &\mapsto F_4\ast F_4.
    \end{align*}

    From part (ii) of the proof of Theorem \ref{classThm}, if the non-ADE Dynkin biagram is not a tensor product, twist, or double binding, then it can be written as a path of parallel bindings with one nonparallel binding:
    \begin{align*}
        \Gamma_1\equiv\Gamma_1\equiv \dots \equiv \Gamma_1 \ast \Gamma_2 \equiv \Gamma_2\equiv\dots\equiv\Gamma_2.
    \end{align*}
    For the rest of the proof, we notate the above Dynkin biagram as $\Gamma_1^n \ast\Gamma_2^m.$
    There is some ADE binding that yields $\Gamma_1\ast \Gamma_2$ through some sequence $\gamma = \gamma_1\gamma_2\dots\gamma_{\ell}$ of global flips and folds. We can obtain an admissible ADE bigraph that will yield the Dynkin biagram $\Gamma_1^n \ast\Gamma_2^m$ through a sequence of global flips and folds by globally taking the inverse of this sequence $\gamma$ as follows:

    \begin{itemize}
        \item [(1)]
            Let $B_{\ell}$ be the Dynkin biagram $\Gamma_1^n \ast\Gamma_2^m$. Set $k = \ell$.
        \item [(2)]
            Given $B_k = \Lambda_1^n\ast \Lambda_2^m$,
            \begin{itemize}
                \item [(i)]
                    if $\gamma_k$ is a global flip, let $B_{k-1} = (B_k)^{\top}.$
                \item [(ii)]
                    if $\gamma_k$ is a folding from $\Lambda_1'\ast \Lambda_2'$ to $\Lambda_1\ast \Lambda_2$, then let $B_{k - 1} = (\Lambda_1')^n\ast(\Lambda_2')^m.$
            \end{itemize}
        \item [(3)]
            If $k = 1$, we are done, and $B_0$ is the desired ADE Dynkin biagram. Otherwise, let $k = k - 1$ and repeat step (2).\qedhere
    \end{itemize}
\end{proof}

\section{Classification of Zamolodchikov periodic cluster algebras}
\label{ZP}
In this section, we prove that Zamolodchikov periodic $B$-matrices are in bijection with admissible Dynkin biagrams.
To do so, we introduce the following two notions:
\subsection{Strictly subadditive labeling}

Let $B = \tilde{\Gamma} + \tilde{\Delta}$ be a bipartite recurrent $B$-matrix, and let $\Gamma = \lvert \tilde{\Gamma} \rvert, \Delta = \lvert \tilde{\Delta} \rvert$. 
We say a labeling $\rho\in\mathbb{R}_{>0}^n$ of the indices $[n]$ is \textit{strictly subadditive} if for all $k\in [n]$,
$$2\rho_k > \sum_i \Gamma_{ik} \rho_i,\qquad 2\rho_k > \sum_j \Delta_{jk} \rho_j.$$
If this inequality is not strict, this is called a subadditive labeling.
An example of a strictly subadditive labeling of $D_4\otimes A_2$ is given in Figure \ref{fig:subadd}.

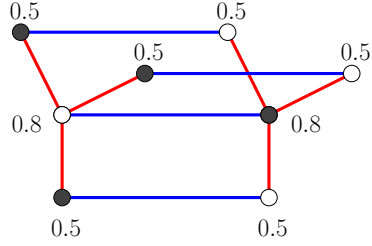
\begin{figure}
\scalebox{0.55}{
\begin{tikzpicture}
\draw (-1.25, 2) node [anchor=north west][inner sep=0.75pt]   [align=left] {\Large{$0.8$}};
\draw (5.50, 2) node [anchor=north west][inner sep=0.75pt]   [align=left] {\Large{$0.8$}};
\foreach \i in {0, 1}
{
    \coordinate (v\i x0) at (5 * \i + 0.00, 0.00);
    \coordinate (v\i x1) at (5 * \i + 0.00, 2.00);
    \coordinate (v\i x2) at (5 * \i + -1.00, 4.00);
    \coordinate (v\i x3) at (5 * \i + 2, 3.00);
    \draw[color=red, line width=0.75mm] (v\i x0) to[] (v\i x1);
    \draw[color=red, line width=0.75mm] (v\i x1) to[] (v\i x2);
    \draw[color=red, line width=0.75mm] (v\i x1) to[] (v\i x3);
    \draw (5 * \i + -0.3, -0.50) node [anchor=north west][inner sep=0.75pt]   [align=left] {\Large{$0.5$}};
    \draw (5 * \i + -1.30, 4.75) node [anchor=north west][inner sep=0.75pt]   [align=left] {\Large{$0.5$}};
    \draw (5 * \i + 1.70, 3.75) node [anchor=north west][inner sep=0.75pt]   [align=left] {\Large{$0.5$}};
}
\foreach \i in {0, 1, 2, 3}
{
    \draw[color=blue, line width=0.75mm] (v0x\i) to[] (v1x\i);
    \draw[fill=black!75!white] (v0x\i.center) circle (0.2);
    \draw[fill=white] (v1x\i.center) circle (0.2);
}
\draw[fill=black!75!white] (v1x1.center) circle (0.2);
\draw[fill=white] (v0x1.center) circle (0.2);
\end{tikzpicture}}
\caption{\label{fig:subadd} A strictly subadditive labeling for $D_4\otimes A_2$.}
\end{figure}

\subsection{Fixed point labeling}
Let $B = \tilde{\Gamma} + \tilde{\Delta}$ be a bipartite recurrent $B$-matrix, and let $\Gamma = \lvert \tilde{\Gamma} \rvert, \Delta = \lvert \tilde{\Delta} \rvert$. 
We say a labeling $\rho\in\mathbb{R}_{>1}^n$ of the indices $[n]$ is a \textit{fixed point labeling} of $B$ if for all $k\in [n]$,
$$\rho_k^2 = \prod_i \rho_i^{\Gamma_{ik}} + \prod_j \rho_j^{\Delta_{jk}}.$$
An example of a fixed point labeling of $A_3\times A_2$ is given in Figure \ref{fig:fixedpt}.

\begin{figure}
\scalebox{0.55}{
\begin{tikzpicture}
\draw (-1.25, -0.50) node [anchor=north west][inner sep=0.75pt] {\Large{$2\sqrt{2}$}};
\draw (-1.00, 2.20) node [anchor=north west][inner sep=0.75pt] {\Large{$4$}};
\draw (-1.25, 4.95) node [anchor=north west][inner sep=0.75pt] {\Large{$2\sqrt{2}$}};
\draw (5.25 + -1.00, -0.50) node [anchor=north west][inner sep=0.75pt] {\Large{$2\sqrt{2}$}};
\draw (5.30 + -0.80, 2.20) node [anchor=north west][inner sep=0.75pt] {\Large{$4$}};
\draw (5.25 + -1.00, 4.95) node [anchor=north west][inner sep=0.75pt] {\Large{$2\sqrt{2}$}};
\foreach \i in {0, 1}
{
    \coordinate (v\i x0) at (4 * \i + 0.00, 0.00);
    \coordinate (v\i x1) at (4 * \i + 0.00, 2.00);
    \coordinate (v\i x2) at (4 * \i + 0.00, 4.00);
    \draw[color=red, line width=0.75mm] (v\i x0) to[] (v\i x1);
    \draw[color=red, line width=0.75mm] (v\i x1) to[] (v\i x2);    
}
\draw[color=blue, line width=0.75mm] (v0x0) to[] (v1x1);
\draw[color=blue, line width=0.75mm] (v0x1) to[] (v1x0);
\draw[color=blue, line width=0.75mm] (v0x1) to[] (v1x2);
\draw[color=blue, line width=0.75mm] (v0x2) to[] (v1x1);
\foreach \i in {0, 1}
{
    \draw[fill=black!75!white] (v\i x0.center) circle (0.2);
    \draw[fill=white] (v\i x1.center) circle (0.2);
    \draw[fill=black!75!white] (v\i x2.center) circle (0.2);
}
\end{tikzpicture}}
\caption{\label{fig:fixedpt} A fixed point labeling for $A_3\times A_2$.}
\end{figure}
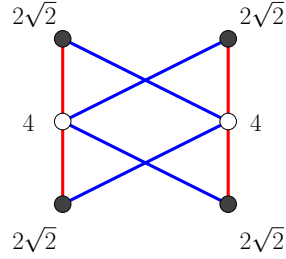

\newpage
\begin{theorem}
\label{5-way}
    Let $B = \tilde{\Gamma} + \tilde{\Delta}$ be a bipartite recurrent $B$-matrix. The following are equivalent:
    \begin{itemize}
        \item [(1)]
            $(\Gamma, \Delta)$ is an admissible Dynkin biagram.
        \item [(2)]
            $B$ has a strictly subadditive labeling.
        \item [(3)]
            $B$ has a fixed point labeling.
        \item [(4)]
            The tropical $T$-system associated with $B$ is periodic for any initial labeling $\lambda\in\mathbb{R}^n$.
        \item [(5)]
            The $T$-system associated with $B$ is periodic.
    \end{itemize}
\end{theorem}

\tikzset{every picture/.style={line width=0.75pt}} 
\begin{figure}
\begin{tikzpicture}[x=0.75pt,y=0.75pt,yscale=-1,xscale=1]

\draw  [line width=0.75]  (212.33,27) -- (374.67,27) -- (374.67,52.33) -- (212.33,52.33) -- cycle ;
\draw   (325,100) -- (483.67,100) -- (483.67,125.33) -- (325,125.33) -- cycle ;
\draw   (298.33,167) -- (508.67,167) -- (508.67,192.33) -- (298.33,192.33) -- cycle ;
\draw  [color={rgb, 255:red, 0; green, 0; blue, 0 }  ,draw opacity=1 ][line width=0.75]  (28.33,167) -- (238.67,167) -- (238.67,192.33) -- (28.33,192.33) -- cycle ;
\draw [color={rgb, 255:red, 0; green, 0; blue, 0 }  ,draw opacity=1 ][line width=0.75]    (135.33,134.33) -- (135.33,160.33) ;
\draw [color={rgb, 255:red, 0; green, 0; blue, 0 }  ,draw opacity=1 ][line width=0.75]    (138.33,134.33) -- (138.33,160.33) ;
\draw  [color={rgb, 255:red, 0; green, 0; blue, 0 }  ,draw opacity=1 ][line width=0.75]  (132.79,135.38) .. controls (135.11,133.72) and (136.49,132.05) .. (136.96,130.37) .. controls (137.41,132.06) and (138.77,133.75) .. (141.05,135.45) ;
\draw    (412.67,134.33) -- (412.67,160.33) ;
\draw    (415.67,134.33) -- (415.67,160.33) ;
\draw   (410.13,135.38) .. controls (412.44,133.72) and (413.82,132.04) .. (414.29,130.37) .. controls (414.74,132.05) and (416.11,133.75) .. (418.38,135.45) ;
\draw   (418.33,159.25) .. controls (416.08,160.99) and (414.76,162.72) .. (414.35,164.42) .. controls (413.84,162.75) and (412.41,161.1) .. (410.08,159.49) ;
\draw [color={rgb, 255:red, 0; green, 0; blue, 0 }  ,draw opacity=1 ][line width=0.75]    (288.38,177.65) -- (247,178) ;
\draw [color={rgb, 255:red, 0; green, 0; blue, 0 }  ,draw opacity=1 ][line width=0.75]    (288.44,180.65) -- (247.07,181) ;
\draw   (288.51,175.12) .. controls (290.21,177.41) and (291.9,178.77) .. (293.59,179.22) .. controls (291.9,179.68) and (290.23,181.08) .. (288.56,183.38) ;
\draw  [color={rgb, 255:red, 0; green, 0; blue, 0 }  ,draw opacity=1 ][line width=0.75]  (247.1,183.65) .. controls (245.31,181.43) and (243.58,180.13) .. (241.87,179.74) .. controls (243.54,179.22) and (245.16,177.77) .. (246.75,175.4) ;
\draw   (31.33,99) -- (249.67,99) -- (249.67,124.33) -- (31.33,124.33) -- cycle ;
\draw    (201.12,45.07) -- (133.87,86.31) ;
\draw    (202.78,47.57) -- (135.54,88.81) ;
\draw   (138.14,90.33) .. controls (134.94,89.67) and (132.27,89.77) .. (130.12,90.63) .. controls (131.75,89.02) and (132.86,86.65) .. (133.45,83.52) ;
\draw    (282.67,60.6) -- (282.67,82.83) ;
\draw    (285.67,60.6) -- (285.67,82.83) ;
\draw   (280.13,61.12) .. controls (282.44,59.46) and (283.83,57.78) .. (284.29,56.1) .. controls (284.74,57.78) and (286.1,59.48) .. (288.38,61.18) ;
\draw    (282.51,82.46) -- (253.16,108.64) ;
\draw    (284.51,84.7) -- (255.15,110.87) ;
\draw    (315.9,107.61) -- (285.69,82.42) ;
\draw    (313.98,109.92) -- (283.77,84.72) ;
\draw    (343.64,158.98) -- (207.74,126.75) ;
\draw    (342.95,161.9) -- (207.05,129.67) ;
\draw   (343.2,156.26) .. controls (344.29,158.89) and (345.59,160.63) .. (347.13,161.48) .. controls (345.38,161.52) and (343.42,162.46) .. (341.24,164.28) ;
\draw (214.33,30) node [anchor=north west][inner sep=0.75pt]   [align=left] {{\fontfamily{ptm}\selectfont (5) $T$-system is periodic}};
\draw (33.67,102.67) node [anchor=north west][inner sep=0.75pt]   [align=left] {{\fontfamily{ptm}\selectfont (4) Tropical $T$-system is periodic}};
\draw (327,103) node [anchor=north west][inner sep=0.75pt]   [align=left] {{\fontfamily{ptm}\selectfont (3) Fixed point labeling}};
\draw (300.33,170) node [anchor=north west][inner sep=0.75pt]   [align=left] {{\fontfamily{ptm}\selectfont (2) Strictly subadditive labeling}};
\draw (30.33,170) node [anchor=north west][inner sep=0.75pt]   [align=left] {{\fontfamily{ptm}\selectfont (1) Admissible Dynkin biagram}};

\end{tikzpicture}

\caption{The plan for the proof of Theorem \ref{5-way}.}
\label{figure:1}
\end{figure}
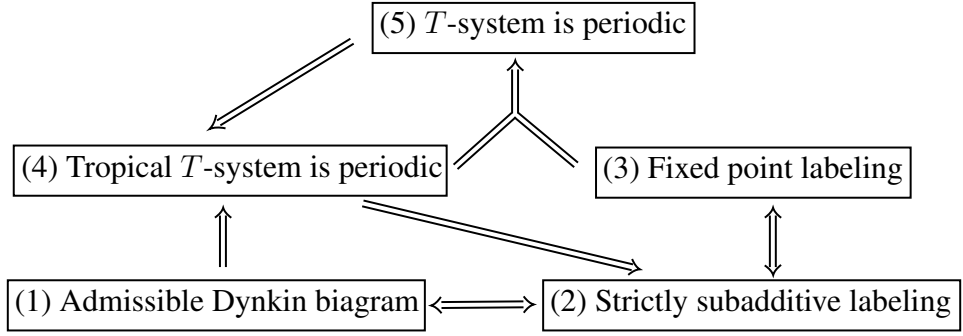

\subsection{Connection between periodicity and tropical periodicity}
\label{polytopes}
This section covers (5)$\implies$(4) as well as (3) + (4)$\implies (5)$.
For a linear function $\lambda: \mathbb{R}^d\rightarrow \mathbb{R}$ and a polytope $P\subset \mathbb{R}^d$, define 
\[\sup(\lambda, P) \coloneqq \sup\{\lambda(z) | z\in P\}.\]
For any Laurent polynomial $p\in \mathbb{Z}[\boldsymbol{x}^{\pm 1}]$, say,
\[p = \sum_{\alpha \in \mathbb{Z}^d} c_{\alpha} \boldsymbol{x}^{\alpha},\]
define its \textit{Newton polytope} $\newton(p)\subset \mathbb{R}^d$ to be $\newton(p) \coloneqq \text{Conv}\{\alpha\in\mathbb{Z}^d~|~c_{\alpha}\neq 0\}$.
Moreover, a Laurent polynomial $p$ is called \textit{Newton-positive} if for every $\alpha\in \mathbb{Z}^d$ that is a vertex of $\newton(p)$, we have $c_{\alpha} > 0$.
Observe that mutation preserves Newton-positivity, so all Laurent polynomials in the birational $T$-system $T(t)$ are Newton-positive.
\begin{lemma}[\cite{pasha}]
\label{newton}
    For a map $\lambda: [n]\rightarrow \mathbb{R}$, the tropical $T$-system $\boldsymbol{t}^{\lambda}$ is obtained from the birational $T$-system via the following transformation:
    \[\boldsymbol{t}_v^{\lambda}(t) = \sup(\lambda, \newton(T_v(t))).\]
\end{lemma}
Notice that if the $T$-system is periodic, it immediately follows that the tropical $T$-system $\boldsymbol{t}^{\lambda}(t)$ is also periodic.

For $p\in\mathbb{Z}[\boldsymbol{x}^{\pm1}]$ and $i\in [n]$, define $\deg_{\max}(i, p)\in \mathbb{Z}$ and $\deg_{\min}(i, p)\in \mathbb{Z}$ to be the maximal and minimal degree of $x_i$ in $p$ viewed as an element of $\mathbb{Z}[x_i^{\pm 1}]$ with all other indeterminates $\{x_j\}_{j\neq i}$ regarded as constants.

For $i\in [n]$, define $\delta_i: [n]\rightarrow \mathbb{R}$ to be
\[
\delta_i(j) \coloneqq \begin{cases}
    1 & \text{if }i = j,\\
    0 & \text{otherwise.}
\end{cases}
\]

Throughout this section, we use the following notation for every $j\in [n]$.
$$\eta_j \coloneqq \begin{cases}
    0 & \text{if }\epsilon_j = \circ,\\
    1 & \text{if }\epsilon_j = \bullet.
\end{cases}$$
Then, notice that the tropical and birational $T$-systems $\boldsymbol{t}_j^{\lambda}$ and $T_j$ are defined only for $t\in\mathbb{Z}$ such that $t + \eta_j$ is even.

Observe that the tropical $T$-system with initial condition $\delta_i$ tracks the maximal degrees of $x_i$ in Laurent polynomials appearing as entries of the birational $T$-system. 
Precisely, we have $\boldsymbol{t}_j^{\delta_i} = \deg_{\max}(i, T_j(t))$ for all $i, j\in [n]$ and all $t\in \mathbb{Z}$ with $t + \eta_j$ even.
The following lemma is an extension of Proposition 6.6 in \cite{pasha}, and shows that $\boldsymbol{t}_j^{\delta_i}$ tracks the minimal degrees of $x_i$ as well:

\begin{lemma}
\label{minmax}
    For all $t\in \mathbb{Z}$ and $i,j\in [n]$,
    \begin{align*}
        -\deg_{\min}(i, T_j(t)) &= \deg_{\max}(i, T_j(t + 2)) \quad  \text{for } \epsilon_i = \circ, \text{ and}\\
        -\deg_{\min}(i, T_j(t)) &= \deg_{\max}(i, T_j(t - 2)) \quad  \text{for }\epsilon_i = \bullet.
    \end{align*}
\end{lemma}
\begin{proof}
    The values $-\deg_{\min}(i, T_j(t))$ clearly satisfy the same recurrence as the values $\deg_{\max}(i, T_j(t))$.
    Therefore it suffices to check the initial case $t = \eta_j$. 
    Initially, we have
    \begin{align*}
        \deg_{\max}(i, T_j(\eta_j)) &= \delta_i(j)\quad  \text{and} \quad  -\deg_{\min}(i, T_j(\eta_j)) = -\delta_i(j).
    \end{align*}
    Assume for instance that $\epsilon_i = \circ$ and $\eta_i = 0$.
    Let us look at $t = 2$, when $T_j(t)$ is defined only for $\epsilon_j = \circ$.
    Since $\epsilon_i = \circ$, $b_{ij} = 0$ and
    \[\deg_{\max}(i, T_j(2)) = \max(0, 0) - \deg_{\max}(i, T_j(0)) = -\delta_i(j).\]
    For $t = 3$, we must have $\epsilon_j = \bullet$.
    \[\deg_{\max}(i, T_j(3)) = \max(-|b_{ij}|, 0) - \deg_{\max}(i, T_j(1)) = -\delta_i(j) = 0.\]
    To summarize, for all $j\in [n]$,
    \[\deg_{\max}(i, T_j(2 + \eta_j)) = -\delta_i(j) = -\deg_{\min}(i, T_j(\eta_j)).\]
    The case of $\epsilon_i = \bullet$ is treated similarly.
    When $t = 2$, we have $\epsilon_j = \circ$, giving us
    \[-\deg_{\min}(i, T_j(2)) = \max(-|b_{ij}|, 0) + \deg_{\min}(i, T_j(0)) = \delta_i(j) = 0.\]
    When $t = 3$, we have $\epsilon_j = \bullet$, so $b_{ij} = 0$ and
    \[-\deg_{\min}(i, T_j(3)) = \max(0, 0) + \deg_{\min}(i, T_j(1)) = \delta_i(j).\]
    So when $\epsilon_i = \bullet,$
    \[-\deg_{\min}(i, T_j(2 + \eta_j)) = \delta_i(j) = \deg_{\max}(\eta_j).\]
\end{proof}

\begin{proposition}
    Given an integer $N$ and a bipartite recurrent $B$-matrix, the following are equivalent:
    \begin{itemize}
        \item[(1)]
            $\boldsymbol{t}_j^{\delta_i}(t + 2N) = \boldsymbol{t}_j^{\delta_i}(t)$ for all $i, j\in [n], t\in \mathbb{Z}$.
        \item[(2)]
            $\boldsymbol{t}_j^{\lambda}(t + 2N) = \boldsymbol{t}_j^{\lambda}(t)$ for all $\lambda\in \mathbb{R}^n$ and $j\in [n], t\in\mathbb{Z}$.
        \item[(3)]
            There exists a labeling $c\in\mathbb{R}_{>0}^n$ such that $T_j(t + 2N) = c_jT_j(t)$ for all $j\in [n], t\in\mathbb{Z}$.
    \end{itemize}
    In any of these cases, if $B$ has a fixed point labeling, then $c_j = 1$ for all $j\in [n]$.
\end{proposition}
\begin{proof}
    Notice that (2)$\implies$(1) is trivial.
    Also, (3)$\implies$(2) follows directly from Lemma \ref{newton}, as $\newton(c_jT_j(t)) = \newton(T_j(t))$.

    To show (1)$\implies$(3), it is enough to consider $t = \eta_j$, because if $T_j(\eta_j + 2N) = c_jT_j(\eta_j)$ for all $j\in [n]$, then we can get the case of arbitrary $t$ via the substitution $x_j \coloneqq T_j(t + \eta_j)$ for all $j\in [n]$.
    By Lemma \ref{minmax}, (1) implies that for every $i, j\in [n]$,
    \begin{align*}
        \deg_{\max}(i, T_j(\eta_j + 2N)) &= \deg_{\max}(i, T_j(\eta_j)) \quad \text{and}\\
        \deg_{\min}(i, T_j(\eta_j + 2N)) &= \deg_{\min}(i, T_j(\eta_j)).
    \end{align*}
    Since $T_j(\eta_j) = x_j$, $\deg_{\max}(i, T_j(\eta_j)) = \deg_{\min}(i, T_j(\eta_j)) = \delta_j(i)$ and therefore $T_j(\eta_j + 2N)$ differs from $T_j(\eta_j)$ by a scalar multiple $c_j$. 
    Since they are both Newton-positive, $c_j$ is necessarily positive, so (3) follows.

    Finally if $\rho\in\mathbb{R}_{>1}^n$ is a fixed point labeling, then substituting $x_j \coloneqq \rho_j$ for all $j\in [n]$ into $T_j(t + 2N) = c_jT_j(t)$ yields $\rho_j = c_j\rho_j$ and thus $c_j = 1$.
\end{proof}
\begin{remark}
    This proof is identical to the proof of Proposition 6.8 in \cite{pasha}, but uses Lemma \ref{minmax} instead of Proposition 6.6.
\end{remark}

\subsection{Admissible Dynkin biagrams have periodic tropical $T$-systems}\hspace{1mm}\\
\noindent
Galashin--Pylyavskyy proved in \cite{pasha} that admissible ADE bigraphs have periodic tropical $T$-systems. 
Recall from Proposition \ref{DynkinFromADE} that every admissible Dynkin biagram can be obtained from an admissible ADE bigraph through a sequence of global flips and folds along bicolored automorphisms. 
In this section, we prove that both of these operations work well with tropical mutations in order to show that all admissible Dynkin biagrams inherit periodic tropical $T$-systems.

\begin{lemma}
\label{transposeMutation}
    Let $B = \tilde{\Gamma} + \tilde{\Delta}$ be a bipartite recurrent $B$-matrix. 
    Let $C$ be a diagonal matrix with positive diagonal entries $c\in\mathbb{R}_{>0}^n$ such that $CB$ is a skew-symmetric matrix.
    Let $\rho\in \mathbb{R}^n$ be a labeling of $B$.
    Define $\lambda\in \mathbb{R}^n$ as $\lambda_i := \frac{1}{c_i}\rho_i$, for all $i\in [n]$.
    Let
    \begin{align*}
        \rho_k' + \rho_k &= \max\left(\sum_i \Gamma_{ik}\rho_i, \sum_j \Delta_{jk} \rho_j\right) \quad \text{and} \quad \lambda_k' + \lambda_k = \max\left(\sum_i \Gamma_{ki}\lambda_i, \sum_j \Delta_{kj}\lambda_j\right)
    \end{align*}
    for $k\in [n]$.
    Then, $\lambda_k' = \frac{1}{c_k}\rho_k'$.
\end{lemma}
\begin{proof}
    Since $c$ scales the rows of $B$ to be skew-symmetric, $\forall i, j\in [n]$ $c_i|b_{ij}| = c_j|b_{ji}|$.
    Then we see that
    \begin{align*}
        \max\left(\sum_i \Gamma_{ik}\rho_i, \sum_j \Delta_{jk}\rho_j\right) &= \max\left(\sum_i c_i\Gamma_{ik}\lambda_i, \sum_j c_j\Delta_{jk}\lambda_j\right)\\
        &= \max\left(\sum_i c_k\Gamma_{ki}\lambda_i, \sum_j c_k \Delta_{kj}\lambda_j\right)\\
        &= c_k\cdot \max\left(\sum_i \Gamma_{ki}\lambda_i, \sum_j \Delta_{kj}\lambda_j\right).
    \end{align*}
    Since $\rho_k = c_k\lambda_k$, we have $\rho_k' = c_k\lambda_k'$ too.
\end{proof}

\begin{proposition}
\label{flipsCommute}
    The tropical $T$-system $\boldsymbol{t}^{\lambda}$ associated with $(\Gamma, \Delta)$ is periodic for all $\lambda\in \mathbb{R}^n$ with period $N$ if and only if the tropical $T$-system $\boldsymbol{t}'^{\lambda}$ associated with $(\Gamma^{\top}, \Delta^{\top})$ is periodic for all $\lambda\in\mathbb{R}^n$ with period $N$.
\end{proposition}
\begin{proof}
    Let $\boldsymbol{t}_i^{\rho}(t)$ be the tropical $T$-system for $B$, and let $\boldsymbol{t}_i'^{\lambda}(t)$ be the tropical $T$-system for $-B^{\top}$.
    Suppose $-B^{\top}$ is periodic with period $N$. Then for any labeling $\lambda\in \mathbb{R}^n$, 
    \begin{align*}
        \boldsymbol{t}_j'^{\lambda}(t + 2N) &= \boldsymbol{t}_j'^{\lambda}(t) \quad \text{for all } j\in [n], t\in \mathbb{Z}.
    \end{align*}
    Let $\rho\in\mathbb{R}^n$ be a labeling for $B$. Let $\lambda\in\mathbb{R}^n$ be defined as $\lambda_i := \frac{1}{c_i}\rho_i$, for all $i\in [n]$. Then by Lemma \ref{transposeMutation},
    \begin{align*}
        \boldsymbol{t}_k^{\rho}(t + 2N) = c_k\cdot \boldsymbol{t}_k'^{\lambda}(t + 2N) = c_k\cdot \boldsymbol{t}_k'^{\lambda}(t) = \boldsymbol{t}_k^{\rho}(t).
    \end{align*}
    Thus $B$ is tropical periodic with period at most $N$. Since $B = -(-B^{\top})^{\top}$, the periods of $B$ and $-B^{\top}$ are the same.
\end{proof}

\begin{definition}
    Let $f: [n] \rightarrow [n]$ be a bicolored automorphism of $B$. $\rho\in \mathbb{R}^n$ is an \textit{$f$-symmetric labeling} of $B$ if $\forall i\in [n], \rho_{f(i)} = \rho_i.$
\end{definition}

\begin{definition}
    Let $f$ be a bicolored automorphism of $B$, and let $\rho$ be an $f$-symmetric labeling. Then, define $f(\rho)\in \mathbb{R}^{|f(B)|}$ to be a labeling on the orbits of $f$ such that 
    $$f(\rho)_I \coloneqq \rho_i,$$
    for every $f$-orbit $I$, and any $i\in I$.
\end{definition}

\begin{remark}
    In \cite{fomin4}, Fomin--Williams--Zelevinsky define a quiver $Q$ to be \textit{globally foldable} under a group $G$ if mutating at all vertices in a $G$-orbit results in a quiver that is still foldable under $G$. Due to our restriction to bipartite dynamics, in this paper all bipartite recurrent quivers are globally foldable, as the quiver (and $B$-matrix) is fixed up to sign by $\mu_{\circ}$ and $\mu_{\bullet}$.
\end{remark}

\begin{lemma}
\label{foldingMutation}
    Let $(B, f, \rho)$ be a $B$-matrix, a bicolored automorphism, and an $f$-symmetric labeling. 
    Let $K$ be an orbit of $f$, and let $\mu_k$ denote tropical mutation at index $k$.
    Then, $\mu_K\circ f(\rho) = f\circ\prod_{k\in K}\mu_i(\rho)$.
\end{lemma}

\begin{proof}
    We can reformulate the statement as the following. 
    Let $\lambda\in \mathbb{R}^n$, $\lambda^*\in \mathbb{R}^{|f(B)|}$ such that
    \begin{itemize}
        \item[(i)] $\lambda_j = \rho_j$ $\forall j\not\in K$,
        \item[(ii)] $\lambda_J^* = \rho_J$ for all orbits $J\neq K$ of $f$,
        \item[(iii)] 
            $\lambda_k + \rho_k = \max\left(\sum_i \Gamma_{ik}\rho_i, \sum_j \Delta_{jk}\rho_j\right)$ $\forall k\in K$, and
        \item[(iv)]
            $\lambda_K^* + f(\rho)_K = \max \Big(\sum_If(\Gamma)_{IK}f(\rho)_I, \sum_Jf(\Delta)_{JK}f(\rho)_J \Big)$.
    \end{itemize}
    Then we want to show $\lambda^*_K = \lambda_k$ for any $k\in K$, and thus $(B, f, \lambda)$ folds to $(f(B), \lambda^*).$ 
    
    \noindent
    Notice that
    \begin{align*}
        \max\left(\sum_I f(\Gamma)_{IK}f(\rho)_I, \sum_J f(\Delta)_{JK} f(\rho)_J\right) &= \max\left(\sum_I f(\Gamma)_{IK}\rho_i, \sum_J f(\Delta)_{JK} \rho_j\right)\\
        &= \max\left(\sum_I \sum_{i\in I}\Gamma_{ik}, \sum_J \sum_{j\in J}\Delta_{jk}\rho_j\right)\\
        &= \max\left(\sum_i\Gamma_{ik}\rho_i, \sum_j\Delta_{jk}\rho_j\right).
    \end{align*}
    Since we already know $f(\rho)_K = \rho_k$ for $k\in K$, we have $\lambda^*_K = \lambda_k$ for any $k\in K$.

\end{proof}
\begin{proposition}
\label{foldingCommutes}
    If the tropical $T$-system $\boldsymbol{t}^{\lambda}$ associated to $(\Gamma, \Delta)$ is periodic for all $\lambda\in \mathbb{R}^n$ with period $N$, and $f$ is a bicolored automorphism, then the tropical $T$-system $\boldsymbol{t}'^{\lambda}$ associated to $(f(\Gamma), f(\Delta))$ is also periodic for all $\lambda\in \mathbb{R}^n$, with period dividing $N$.
\end{proposition}

\begin{proof}

    Let $\lambda\in\mathbb{R}^{|f(B)|}$ be any labeling of $(f(\Gamma), f(\Delta))$.
    Define $\rho\in \mathbb{R}^n$ as follows: 
    $$\rho_i := \lambda_I \quad \forall i\in I, \text{for every $f$-orbit }I.$$
    $\rho$ is clearly an $f$-symmetric labeling, and for every $f$-orbit $I$, $f(\rho)_I = \rho_i = \lambda_I.$
    So, $f$ folds $(B, \rho)$ into $(f(B), \lambda)$.
    Let $\boldsymbol{t}_i^{\rho}(t)$ be the tropical $T$-system for $(\Gamma, \Delta)$, and let $\boldsymbol{t}_I'^{\lambda}(t)$ be the tropical $T$-system for $(f(\Gamma), f(\Delta))$.
    Since $\boldsymbol{t}_i^{\rho}(t)$ is periodic, by Proposition \ref{foldingMutation},
    \begin{align*}
        \boldsymbol{t}_I'^{\lambda}(t + 2N) = \boldsymbol{t}_i^{\rho}(t + 2N) = \boldsymbol{t}_i^{\rho}(t) = \boldsymbol{t}_I'^{\lambda}(t)
    \end{align*}
    for every $f$-orbit $I$, $\forall t\in \mathbb{Z}$. 
    Thus the tropical $T$-system for $(f(\Gamma), f(\Delta))$ is periodic with period dividing $N$.
\end{proof}

\begin{corollary}
    The tropical $T$-system associated to an admissible Dynkin biagram $(\Gamma, \Delta)$ is periodic, with period dividing $h_{\Gamma} + h_{\Delta}$.
\end{corollary}

\begin{proof}
    From Proposition \ref{DynkinFromADE}, every admissible Dynkin biagram can be obtained from an admissible ADE bigraph through a series of folding and transpose operations. Admissible ADE bigraphs are shown in \cite{pasha} to have periodic tropical $T$-systems $\boldsymbol{t}^{\lambda}$ for all $\lambda\in \mathbb{R}^n$, with period dividing $h_{\Gamma} + h_{\Delta}$.
    By Proposition \ref{foldingCommutes} and Proposition \ref{flipsCommute}, each application of these operations produces another $B$-matrix whose tropical $T$-system is periodic for every $\lambda\in\mathbb{R}^n$.
\end{proof}

\begin{remark}
    It is worth considering the size of the period for each of the Zamolodchikov periodic families. 
    In \cite{keller}, Keller proves the period for tensor products $\Gamma\otimes \Delta$ divides $h_{\Gamma} + h_{\Delta}$.
    In \cite{pasha}, Galashin--Pylyavskyy show that the period for twists $\Gamma\times \Gamma$ is $h_{\Gamma}$ when $h_{\Gamma}$ is even. 
    In the remaining ADE families, they derive periodicity from that of tensor products, preserving the period to be a divisor of $h_{\Gamma} + h_{\Delta}$.
    In our construction, we derive all admissible Dynkin biagrams by folding ADE bigraphs and taking transposes in a way that preserves periodicity.
    As a result, all of the admissible Dynkin biagrams are Zamolodchikov periodic with period dividing $h_{\Gamma} + h_{\Delta}$.
    However, the actual period of each family is still unknown.
\end{remark}

\subsection{Strictly subadditive labeling $\iff$ admissible Dynkin biagram}

The following characterization of finite type Cartan matrices is due to Vinberg:
\begin{theorem}[\cite{vinberg}]
    Let $A$ be an $n\times n$ Cartan matrix. Then $A$ is of finite type if and only if there exists a vector $\alpha\in\mathbb{R}_{>0}^n$ such that $A\alpha\in \mathbb{R}_{>0}^n$.
\end{theorem}

The following proposition and its proof generalize Proposition 5.1 in \cite{pasha}.
\begin{proposition}
    Let $B = \tilde{\Gamma} + \tilde{\Delta}$ be a bipartite recurrent $B$-matrix. 
    $B$ has a strictly subadditive labeling if and only if $(\Gamma, \Delta)$ is an admissible Dynkin biagram.
\end{proposition}

\begin{proof}
    First, we claim that for a bipartite matrix $B$, $B$ is recurrent if and only if $\Gamma, \Delta$ commute.
    We can see this is true by observing the action of $\mu_{\circ}, \mu_{\bullet}$ on a bipartite matrix $B$.
    If $i, j$ are two indices of different color, then $B_{ij} = \mu_{\circ}(B)_{ji}$.
    If $i, j$ are both black, then $\mu_{\circ}(B)_{ij}$ is the number of directed length 2 paths from $i$ to $j$, which is $(\Delta\Gamma)_{ij} - (\Gamma\Delta)_{ij}$.
    Similarly if $i, j$ are both white, then $\mu_{\bullet}(B)_{ij}$ is the number of directed length 2 paths from $i$ to $j$.
    Then it is clear that $B$ is recurrent if and only if $\mu_{\circ}(B), \mu_{\bullet}(B)$ remain bipartite, which is exactly when $(\Delta\Gamma)_{ij} = (\Gamma\Delta)_{ij}$.
    
    It remains to show that $B$ has a strictly subadditive labeling if and only if the connected components of $\Gamma, \Delta$ are Dynkin diagrams.

    One direction follows directly from Vinberg's characterization.
    Recall that the Cartan matrices and Coxeter adjacency matrices are closely related:
    $$C_{\Gamma} = 2I - \Gamma, \qquad C_{\Delta} = 2I - \Delta.$$
    Given a strictly subadditive labeling of $B$, restricting to one irreducible component of $\Gamma$ or $\Delta$ remains strictly subadditive.
    $$2\rho_k > \sum_i \Gamma_{ik} \rho_i,\qquad 2\rho_k > \sum_j \Delta_{jk} \rho_j.$$
    Rearranging, we get
    $$(C_{\Gamma}^{\top}\rho)_k = 2\rho_k - \sum_i \Gamma_{ik} \rho_i > 0,\qquad (C_{\Delta}^{\top}\rho)_k = 2\rho_k - \sum_j \Delta_{jk} \rho_j> 0,$$
    and it is clear that by Vinberg's characterization, each of the components of $C_{\Gamma}^{\top}$ and $C_{\Delta}^{\top}$ are of finite type.
    Therefore, each of the components of $C_{\Gamma}$ and $C_{\Delta}$ are also of finite type and are Dynkin diagrams.

    For the other direction, let $(\Gamma, \Delta)$ be an admissible Dynkin biagram.
    Recall from Lemma \ref{commonEigVec} that $\Gamma, \Delta$ have a common dominant eigenvector $\vv{v}\in\mathbb{R}^n_{>0}$. Recall also that the dominant eigenvalues for $\Gamma$ and $\Delta$ are $2\cos(\pi/h_{\Gamma})$ and $2\cos(\pi/h_{\Delta})$, respectively. So we have
    \begin{align*}
        2v_k > 2\cos(\pi/h_{\Gamma})v_k &= (\vv{v}\Gamma)_k = \sum_i \Gamma_{ik}v_i \quad \text{and}\\
        2v_k > 2\cos(\pi/h_{\Delta})v_k &= (\vv{v}\Delta)_k = \sum_j \Delta_{jk}v_j
    \end{align*}
    for all $k\in [n]$. This is a strictly subadditive labeling of $B$.
\end{proof}

\subsection{Fixed point labeling $\iff$ strictly subadditive labeling}
\begin{proposition}
    Let $B$ be a bipartite recurrent $B$-matrix.
    There exists a fixed point labeling for $B$ if and only if there exists a strictly subadditive labeling for $B$.
\end{proposition}
The proof is identical to that of Proposition 4.1 in \cite{pasha}, but with matrix notation instead of graph notation to account for skew-symmetrizable matrices.

\begin{proof}
    For any labeling $\rho\in \mathbb{R}_{\geq 1}^n$, define another labeling $Z\rho\in \mathbb{R}_{\geq 1}^n$ by
    \begin{align*}
        (Z\rho)_k &= \sqrt{\prod_i \rho_i^{\Gamma_{ik}} + \prod_{j}\rho_j^{\Delta_{jk}}}.
    \end{align*}
    We may view $Z$ as a map $Z: \mathbb{R}_{\geq 1}^n\rightarrow \mathbb{R}_{\geq 1}^n$ whose fixed points in $\mathbb{R}_{>1}^n$ are exactly the fixed point labelings for $B$.

    For two labelings $\rho', \rho''\in \mathbb{R}_{\geq 1}^n$, we write $\rho' > \rho''$ (resp., $\rho' \geq \rho'')$ if and only if $\rho'-\rho''\in \mathbb{R}_{>0}^n$ (resp., $\rho'-\rho''\in \mathbb{R}_{\geq 0}^n$).
    Notice that $Z\vv{\boldsymbol{1}} \geq \vv{\boldsymbol{\sqrt{2}}} > \vv{\boldsymbol{1}}$.
    We introduce an intermediate condition and show the following are equivalent:
    \begin{itemize}
        \item[(i)]
            There exists a strictly subadditive labeling of $B$.
        \item[(ii)]
            There exists a labeling $\rho^0\in\mathbb{R}_{>1}^n$ satisfying $\vv{\boldsymbol{1}} < Z\rho^0 < \rho^0$.
        \item[(iii)]
            There exists a fixed point labeling for $B$.
    \end{itemize}

    We start with (ii) $\implies$ (iii).
    It is easy to see that if $\rho \geq \rho'$ then $Z\rho \geq Z\rho'$.
    Already, we have $Z\vv{\boldsymbol{1}} > \vv{\boldsymbol{1}}$ and $Z\rho^0 < \rho^0$.
    Therefore the convex compact set
    \begin{align*}
        [\vv{\boldsymbol{1}}, \rho^0] \coloneqq \{\rho\in \mathbb{R}_{\geq 1}^n | \vv{\boldsymbol{1}}\leq \rho \leq \rho^0\}
    \end{align*}
    is mapped by $Z$ to itself, so (iii) follows immediately from the Brouwer fixed point theorem.
    Since $Z\vv{\boldsymbol{1}} > \vv{\boldsymbol{1}}$, none of the coordinates of this fixed point can be equal to 1 and so our fixed point is in $\mathbb{R}_{> 1}^n$.

    Now we show (iii) $\implies$ (i).
    Assume $\rho\in \mathbb{R}_{> 1}^n$ is a fixed point of $Z$.
    Define $\nu\in\mathbb{R}_{>0}^n$ by $\nu_k \coloneqq \log{\rho_k}$ for all $k\in [n]$ and then we have
    \begin{align*}
        e^{\nu_k} &= \rho_k = (Z\rho)_k = \sqrt{\prod_i \rho_i^{\Gamma_{ik}} + \prod_{j}\rho_j^{\Delta_{jk}}} > \max\left\{\sqrt{\prod_i \rho_i^{\Gamma_{ik}}}, \sqrt{\prod_{j}\rho_j^{\Delta_{jk}}}\right\},
    \end{align*}
    so after taking the logarithm of both sides we see that $\nu$ is a strictly subadditive labeling of $B$.

    It remains to show (i) $\implies$ (ii).
    Assume $\nu\in\mathbb{R}_{>0}^n$ is a strictly subadditive labeling of $B$.
    Since $n$ is finite, there exists a number $q > 0$ such that
    \begin{align*}
        2\nu_k - q > \sum_{i}\Gamma_{ik}\nu_i, \hspace{5mm} 2\nu_k - q > \sum_j \Delta_{jk}\nu_j
    \end{align*}
    for all $k\in [n].$
    Let $\alpha > 1$ be a large enough real number so that $\alpha^q > 2$.
    Define $\rho_k \coloneqq \alpha^{\nu_k}$ for all $k\in [n]$.
    Then we have
    \begin{align*}
        (Z\rho)_k &= \sqrt{\prod_i \rho_i^{\Gamma_{ik}} + \prod_{j}\rho_j^{\Delta_{jk}}} = \sqrt{\alpha^{\sum_i \Gamma_{ik}\nu_i} + \alpha^{\sum_j \Delta_{jk}\nu_j}}\\
        &< \sqrt{2\alpha^{2\nu_k - q}} < \alpha^{\nu_k} = \rho_k.
    \end{align*}
    Clearly, we have $\vv{\boldsymbol{1}} < \rho$ and $Z\vv{\boldsymbol{1}} < Z\rho$, so
    \begin{align*}
        \vv{\boldsymbol{1}} < Z\vv{\boldsymbol{1}} < Z\rho < \rho,
    \end{align*}
    which concludes the proof.
\end{proof}

\subsection{Tropical periodicity $\implies$ strictly subadditive labeling}
\begin{proposition}
\label{4->2}
    Let $B$ be a bipartite recurrent $B$-matrix and assume the tropical $T$-system $\boldsymbol{t}^{\lambda}$ associated with $B$ is periodic for all $\lambda\in \mathbb{R}^n$.
    Then there exists a strictly subadditive labeling for $B$.
\end{proposition}
The proof is identical to that of Proposition 7.1 in \cite{pasha}, but with matrix notation instead of graph notation to account for skew-symmetrizable matrices\footnote{Proposition \ref{4->2} is also a special case of \cite[Theorem~5.5]{mizuno}.}.
\begin{proof}
    Let $N\in \mathbb{N}$ be the period of the tropical $T$-system so that for all $i, j\in [n]$, $\boldsymbol{t}_j^{\delta_i}(t + 2N) = \boldsymbol{t}_j^{\delta_i}(t)$.
    We wish to find a strictly subadditive labeling $\nu\in\mathbb{R}_{>0}^n$.
    To do so, we construct intermediate labelings $a^i\in\mathbb{R}^n$ for every $i\in [n]$.
    Define $a_j^i$ as the sum of $\deg_{\text{max}}(i, T_j(t))$ over one period:
    \begin{align*}
        a_j^i \coloneqq \sum_{k = 0}^{N-1} \boldsymbol{t}_j^{\delta_i}(2k + \eta_j), \quad  \text{where }\eta_j  = \begin{cases}
            0 & \text{if }\epsilon_j = \circ,\\
            1 & \text{if }\epsilon_j = \bullet.
        \end{cases}
    \end{align*}

    We first show that $a^i\geq 0$ for all $i\in [n]$. 
    By Lemma \ref{minmax}, the average of $\deg_{\max}(i, T_j(t))$ over one period equals the average of $-\deg_{\min}(i, T_j(t))$.
    Since $\deg_{\max}(i, T_j(t))\geq \deg_{\min}(i, T_j(t))$ for every $t\in \mathbb{Z}$, we get $a_j^i \geq -a_j^i$.
    This implies $a_j^i\geq 0$ for all $i, j\in [n]$.

    In addition, if $b_{ij}\neq 0$ and $\epsilon_i = \bullet, \epsilon_j = \circ$, then $\deg_{\max}(i, T_j(2)) = |b_{ij}| > 0 = \deg_{\min}(i, T_j(2))$ and in this case, $a_j^i > -a_j^i$.
    Similarly if $\epsilon_i = \circ, \epsilon_j = \bullet$, then $\deg_{\max}(i, T_j(-1)) = |b_{ij}| > 0 = \deg_{\min}(i, T_j(-1))$.
    Hence $a_j^i > 0$ when $b_{ij}\neq 0$.

    Next, we claim the labeling $a^i$ is (not necessarily strictly) subadditive:
    \begin{equation}
    \label{eq:1}
        2a_j^i \geq \max\left(\sum_{\ell} \Gamma_{\ell j}a_{\ell}^i, \sum_{\ell} \Delta_{\ell j}a_{\ell}^i\right).
    \end{equation}
    Indeed, using the periodicity of $\boldsymbol{t}_j^{\delta_i}$, we can write
    \begin{align*}
        2a_j^i &= \sum_{k = 0}^{N-1}\boldsymbol{t}_j^{\delta_i}(2k + \eta_j) + \sum_{k = 1}^{N}\boldsymbol{t}_j^{\delta_i}(2k + \eta_j)\\
        &= \sum_{k = 0}^{N-1}\big(\boldsymbol{t}_j^{\delta_i}(2k + \eta_j) + \boldsymbol{t}_j^{\delta_i}(2k + 2 + \eta_j)\big).
    \end{align*}
    This is the form for tropical mutation, and is equal to
    \begin{align*}
        &= \sum_{k = 0}^{N-1}\max\left(\sum_{\ell} \Gamma_{\ell j} \boldsymbol{t}_{\ell}^{\delta_i}(2k + 1 + \eta_j), \sum_{\ell} \Delta_{\ell j} \boldsymbol{t}_{\ell}^{\delta_i}(2k + 1 + \eta_j)\right)\\
        &\geq \max\left(\sum_{k = 0}^{N-1}\sum_{\ell} \Gamma_{\ell j} \boldsymbol{t}_{\ell}^{\delta_i}(2k + 1 + \eta_j), \sum_{k = 0}^{N-1}\sum_{\ell} \Delta_{\ell j} \boldsymbol{t}_{\ell}^{\delta_i}(2k + 1 + \eta_j)\right)\\
        &= \max \left(\sum_{\ell}\Gamma_{\ell j} a_{\ell}^i, \sum_{\ell}\Delta_{\ell j} a_{\ell}^i\right).
    \end{align*}
    Thus $a^i$ is a subadditive labeling.
    The only way (\ref{eq:1}) can be an equality for $j\in [n]$ is when one of the following holds:
    \begin{itemize}
        \item for all $k\in \{0, 1, \dots, N-1\},$
            \begin{equation}
            \label{eq:2}
                \sum_{\ell}\Gamma_{\ell j} \boldsymbol{t}_{\ell}^{\delta_i}(2k + 1 + \eta_j) \geq \sum_{\ell}\Delta_{\ell j} \boldsymbol{t}_{\ell}^{\delta_i}(2k + 1 + \eta_j).
            \end{equation}
        \item for all $k\in \{0, 1, \dots, N-1\},$
            \begin{equation}
            \label{eq:3}
                \sum_{\ell}\Gamma_{\ell j} \boldsymbol{t}_{\ell}^{\delta_i}(2k + 1 + \eta_j) \leq \sum_{\ell}\Delta_{\ell j} \boldsymbol{t}_{\ell}^{\delta_i}(2k + 1 + \eta_j).
            \end{equation}
    \end{itemize}
    As noted in the proof of Lemma \ref{minmax}, there are two integers $k_1$ and $k_2$ such that for all $\ell\in [n]$,
    \begin{align*}
        \boldsymbol{t}_{\ell}^{\delta_i}(2k_1 + 1 + \eta_j) &= \delta_i(\ell) \quad \text{and} \quad \boldsymbol{t}_{\ell}^{\delta_i}(2k_2 + 1 + \eta_j) = -\delta_i(\ell).
    \end{align*}

    Thus if $b_{ij} \neq 0$, then we get nonzero summations in (\ref{eq:2}) and (\ref{eq:3}) for $k = k_1, k_2$ and neither of these inequalities hold for all values of $k$ simultaneously.
    It follows that the inequality (\ref{eq:1}) is strict whenever $b_{ij}\neq 0$. 
    This allows us to define $\nu$ as a sum over all $a^i$:
    \begin{align*}
        \nu_j &\coloneqq \sum_{i = 1}^n a_j^i.
    \end{align*}
    Clearly, $\nu$ inherits subadditivity from $a^i$, and since each vertex $j$ is a neighbor of some vertex $i$, $\nu$ is strictly subadditive.
    By the same reasoning, $\nu_j > 0$ for all $j\in [n]$.
\end{proof}

This concludes the proof of Theorem \ref{5-way} and Theorem \ref{mainThm}.
\section{Connections to $W$-graphs}
\label{$W$-graphs}

In 2010, Stembridge studied admissible ADE bigraphs, as they helped classify the admissible $I_2(p) \times I_2(q)$-cells, which encode the action of the standard generators inherited from $I_2(p)\times I_2(q)$ on the Kazhdan--Lusztig basis of its associated Iwahori--Hecke algebra $\mathcal{H}(W)$. In this section, we explore which $W$-graphs are associated with admissible Dynkin biagrams. We adopt most of our notation from \cite{stembridge2} and \cite{stembridge}.

Let $I$ be a finite index set. An \textit{I-labeled graph} is a triple $\Lambda = (V, m, \tau)$ such that
\begin{itemize}
    \item[(i)]
        $V$ is a vertex set of size $n\in\mathbb{N}$,
    \item[(ii)]
        $m\in \text{Mat}_{n\times n}(\mathbb{Z}[q^{\pm 1/2}])$ is a weighted adjacency matrix, and
    \item[(iii)]
        $\tau: V\rightarrow \{\text{subsets of } I\}$ is a labeling of the vertices, called the \textit{$\tau$-invariant}.
\end{itemize}

Now, let $W$ be a Coxeter group relative to a generating set $S = \{s_i~|~ i\in I\}$.
Let $\{T_i ~|~ i\in I\}$ be the corresponding set of generators of the associated Iwahori--Hecke algebra $\mathcal{H}$ over the ground ring $\mathbb{Z}[q^{\pm 1/2}]$.
Recall that the defining relations for $\mathcal{H}$ are the quadratic relations $(T_i - q)(T_i + 1) = 0$ for all $i\in I$, and the braid relations inherited from $(W, S)$.

\begin{definition}[$W$-graph]
    An $I$-labeled graph $\Lambda$ is a \textit{$W$-graph} if the $\mathbb{Z}[q^{\pm 1/2}]$-module $M_{\Lambda}$ freely generated by vertex set $V$ may be given the $\mathcal{H}$-module structure such that for all $u\in V$,
    \begin{align*}
        T_i(u) = \begin{cases}
            qu & \text{if } i\not\in \tau(u),\\
            -u + q^{1/2}\sum_{v : i\not\in \tau(v)} m_{uv}v & \text{if } i\in \tau(u).
        \end{cases}
    \end{align*}
\end{definition}

In particular, these operators always satisfy the quadratic relations $(T_i - q)(T_i + 1) = 0$, so the only nontrivial condition to satisfy is maintaining the braid relations.

In addition, if $m_{uv} = 0$ whenever $\tau(u)\subseteq \tau(v)$, then we say the I-labeled graph is \textit{reduced}. 
Notice that when $\tau(u)\subseteq \tau(v)$, the value of $m_{uv}$ never plays a role in the operators defined above. 
So without loss of generality, we assume our $W$-graphs are always reduced.
A $W$-graph is a \textit{$W$-cell} if its matrix of edge weights $m$ is strongly connected.
Since the vertex set of any directed graph can be uniquely partitioned into strongly connected components, it can be shown that a $W$-graph $\Lambda$ can be broken into $W$-cells.
If $\Lambda$ is a reduced $W$-graph, then a vertex $v$ with $\tau(v) = \emptyset$ or $\tau(v) = [n]$ necessarily has outdegree 0 or indegree 0.
So in this case, $\{v\}$ forms its own cell of $\Lambda$.
Such cells are called \textit{trivial}.

Given a Coxeter group $W$, one can construct a $W$-graph $(V, m, \tau)$ by taking the matrices representing the action of the standard generators $T_i$ on the Kazhdan--Lusztig basis of $\mathcal{H}$.
In particular, the matrix $m$ of edge weights is exactly a function of the coefficients in the Kazhdan--Lusztig polynomial $P_{u, v}(q)$.
For more details on the construction, we refer the reader to \cite{kazhdan}.

\begin{remark}
    In this section, we focus on classifying $W$-cells that are bipartite with nonnegative edge weights.
    This is a generalization of \textit{admissible} $W$-cells as defined in \cite{stembridge2}, which captures several combinatorial features of Kazhdan--Lusztig $W$-graphs.
    Indeed, since $P_{u, v}(q)$ only has integral powers of q, it follows that $m$ is bipartite.
    Additionally, for a finite Coxeter group $W$, its associated Kazhdan--Lusztig polynomial $P_{u,v}(q)$ has nonnegative integer coefficients, so these $W$-graphs have nonnegative edge weights.
\end{remark}

If $\Gamma = (U, m, \tau_1)$ is a $W_1$-graph and $\Delta = (V, m', \tau_2)$ is a $W_2$-graph, the \textit{outer product} $\Gamma\otimes \Delta$ is the $W_1\times W_2$-graph with vertex set $U\times V$ and
\begin{align*}
    \tau(uv) &\coloneq \tau_1(u)\cup \tau_2(v),\\
    m''_{uv, u'v'} &\coloneq m_{u, u'}\delta_{v, v'} + m'_{v, v'}\delta_{u, u'},
\end{align*}
for all $u,u'\in U$ and $v,v'\in V$.
Notice that the outer product $\Gamma\otimes \Delta$ gives the same matrix as the tensor product $\Gamma\otimes \Delta$, so this notation is consistent.

The following two lemmas were proven in \cite{stembridge2}.
\begin{lemma}
\label{stembridgeProp}
    Let $\Lambda$ be a reduced $\{1, 2\}$-labeled graph such that every vertex has $\tau$-invariant $\{1\}$ or $\{2\}$. If $m$ is the matrix of edge weights of $\Lambda$, then $\Lambda$ is an $I_2(p)$-graph if and only if $\phi_p(m) = 0$, where $\phi_p(t)$ is the degree $p-1$ polynomial defined by the recurrence
    \begin{align*}
        \phi_{r+1}(t) = t\phi_r(t) - \phi_{r-1}(t)
    \end{align*}
    with initial conditions $\phi_0(t) = 0$, $\phi_1(t) = 1$.
\end{lemma}

\begin{lemma}
\label{compatibilityRule}
    Let $(W,S)$ be a Coxeter group, and let $\Lambda$ be a $W$-graph.
    If $m_{uv}\neq 0$, then for every $i\in \tau(u) - \tau(v)$ and $j\in \tau(v) - \tau(u)$, $s_i$ is bonded to $s_j$ in the Coxeter diagram of $(W,S)$.
\end{lemma}

In \cite{stembridge2}, the nontrivial admissible $I_2(p)$-cells are shown to be in correspondence with 2-colorings of ADE Dynkin diagrams with Coxeter number $h$ dividing $p$.
The following result is a generalization of this result to Dynkin diagrams.

\begin{theorem}
\label{I_2(p)-cells}
    A bipartite strongly connected multigraph with vertex color labels $\{1\}, \{2\}$ is a nontrivial $I_2(p)$-cell with nonnegative edge weights $m$ if and only if it is a Dynkin diagram of finite type whose Coxeter number divides $p$.
\end{theorem}

The proof of Theorem \ref{I_2(p)-cells} is a modification of the proof in \cite{stembridge2}.

\begin{proof}
    Let $\Lambda$ be an appropriately labeled multigraph, viewed as an admissible $\{1,2\}$-labeled graph with a symmetric edge-weight marix $m$.
    From Lemma \ref{stembridgeProp}, we know $\Lambda$ is an $I_2(p)$-cell if it is strongly connected and satisfies $\phi_p(m) = 0.$
    Notice that
    \begin{align*}
        \sin(r+1)\theta &= 2\cos{\theta}\sin{r\theta} - \sin(r-1)\theta,
    \end{align*}
    which is the same recurrence as $\phi_r(t)$ if we let $t = 2\cos{\theta}$. The only difference is in the initial conditions. Where $\phi_0(t) = 0$ and $\phi_1(t) = 1$, we now have $\sin(0\cdot \theta) = 0$ and $\sin(1\cdot \theta) = \sin(\theta).$ It is straightforward to see that for all $k\in\mathbb{N}$,
    $$\sin{k\theta} = \sin{\theta}\cdot \phi_k(2\cos{\theta}).$$
    It follows that $\phi_p(2\cos{\theta}) = \frac{\sin{p\theta}}{\sin{\theta}}$ and therefore the roots of $\phi_p(t)$ are exactly
    $$2\cos\left(\frac{k\pi}{p}\right), \forall k\in \{1, \dots, p-1\}.$$
    If $\phi_p(m) = 0$, then all eigenvalues of $m$ are roots of $\phi_p$, which are strictly less than 2.
    In particular, $2-m$ is a Cartan matrix whose eigenvalues are strictly positive.

    Since $m$ is strongly connected, it is irreducible.
    By the Perron--Frobenius Theorem for irreducible nonnegative matrices, there exists a positive eigenvalue $\lambda$ with eigenvector $\vv{v}\in\mathbb{R}_{>0}^n$.
    The matrices $2-m$ and $m$ have the same eigenvectors, so $\vv{v}\in\mathbb{R}_{>0}^n$ and $(2-m)\vv{v}\in \mathbb{R}_{>0}^n$ too. 
    By Vinberg's characterization, $2-m$ is a Cartan matrix of finite type, so $m$ is the Coxeter adjacency matrix of a Dynkin diagram.

    Conversely, suppose we have a Dynkin diagram with Coxeter adjacency matrix $A$ and bipartite labeling $\tau$.
    We want to see for which Dynkin diagrams is $\Lambda = (V, A, \tau)$ an $I_2(p)$-cell.
    We already know if $\phi_p(A) = 0$, then all eigenvalues of $A$ are roots of $\phi_p$.
    Recall from \cite{eigenvalues} that the eigenvalues of $A$ are of the form $2\cos\left(\frac{e_j \pi}{h}\right)$, where the $e_j$ are exponents of $A$ and $h$ is the Coxeter number.
    In particular, $A$ has $n$ distinct eigenvalues $\lambda_1, \dots, \lambda_n$, so $A$ is diagonalizable.
    If $\lambda_i$ is a root of $\phi_p$ for every $i$, it follows that $\phi_p(A) = 0$.
    This happens exactly when $h$ divides $p$.
\end{proof}

Given the form of the nonnegative $I_2(p)$-cells, the following theorem presents the connection between nonnegative $I_2(p)\times I_2(q)$-cells and connected admissible Dynkin biagrams. 
This is a generalization of the result in \cite{stembridge} classifying the admissible $I_2(p)\times I_2(q)$-cells.
\begin{theorem}
\label{I_2(p)xI_2(q)-cells}
    The $I_2(p)\times I_2(q)$-cells with nonnegative edge weights are in correspondence with certain $\{1, 2, 3, 4\}$-labeled connected admissible Dynkin biagrams $(\Gamma, \Delta)$ such that $h_{\Gamma}$ divides $p$ and $h_{\Delta}$ divides $q$.
    In particular, each such Dynkin biagram has four distinct labelings that produce an $I_2(p)\times I_2(q)$-cell.
\end{theorem}

The following proof is adapted from \cite{stembridge}.
\begin{proof}
    We will show that $I_2(p)\times I_2(q)$-cells with nonnegative edge weights will be either
    \begin{itemize}
        \item[(1)]
            outer products of nonnegative $I_2(p)$-cells with trivial $I_2(q)$-cells (or vice versa), or
        \item[(2)]
            other connected admissible Dynkin biagrams with respective Coxeter numbers $h_{\Delta} | p$ and $h_{\Gamma} | q$ and vertex labels $V(i,j) = \{v | \tau(v) = \{i+1, j+3\}\}$ forming a 2x2 partition.
    \end{itemize}

    Let $\Lambda = (V, m, \tau)$ be a nonnegative $I_2(p)\times I_2(q)$-cell. 
    Let $T_1, T_2$ be the associated generators for $I_2(p)$, and $T_3, T_4$ the generators for $I_2(q)$.
    By Lemma \ref{compatibilityRule}, if $m_{uv}\neq 0$, then for every $i\in \tau(u) - \tau(v)$ and $j\in \tau(v) - \tau(u)$, $s_i$ and $s_j$ do not commute.
    Since $\Lambda$ is reduced, $\tau(u)\not\subseteq \tau(v)$, so one of the following must hold:
    \begin{itemize}
        \item[(i)] 
            $\tau(v)$ is a proper subset of $\tau(u)$.
        \item[(ii)] 
            $\tau(u) - \tau(v) = \{1\}$ and $\tau(v) - \tau(u) = \{2\}$ (or vice versa).
        \item[(iii)] 
            $\tau(u) - \tau(v) = \{3\}$ and $\tau(v) - \tau(u) = \{4\}$ (or vice versa).
    \end{itemize}
    If $3, 4\not\in \tau(u)$, then in each of the cases above the same must be true for $\tau(v)$ and therefore all vertices, since $\Lambda$ is strongly connected.
    So, the restriction of $\Lambda$ to $I_2(q)$ results in a disjoint union of copies of a trivial $I_2(q)$-cell $\Delta$, and the restriction to $I_2(p)$ results in a single nonnegative $I_2(p)$-cell. This is the outer product of the nonnegative $I_2(p)$-cell and the trivial cell $\Delta$.
    The same applies if $1, 2\not\in \tau(u)$, leaving the outer product of a nonnegative $I_2(q)$-cell and a trivial $I_2(p)$-cell.
    
    Similarly, if $3, 4\in \tau(v)$, then the same is true for $\tau(u)$ and all other vertices, as $\Lambda$ is strongly connected.
    This leaves the outer product of a nonnegative $I_2(p)$-cell and a trivial $I_2(q)$-cell.
    The same is true if $1, 2\in \tau(v)$, leaving the outer product of a nonnegative $I_2(q)$-cell and a trivial $I_2(p)$-cell.

    From Theorem \ref{I_2(p)-cells}, each of these trivial outer products are in correspondence with Dynkin diagrams with suitable Coxeter numbers and a bipartite labeling.

    The remaining possibility is when for each vertex $w$ of $\Lambda$, $\tau(w)$ contains one element of $\{1, 2\}$ and one element of $\{3, 4\}$.
    Then there are no edges of type (i), so $\Lambda$ has no one-directional edges ($u, v$ such that $m_{uv} = 0$ and $m_{vu}\neq 0$).
    $\Lambda$ has the structure of a Dynkin biagram $(\Gamma, \Delta)$ if we let $\Gamma$ be edges of type (ii) and $\Delta$ be edges of type (iii).
    Indeed, restriction of $\Lambda$ to $I_2(p)$ or $I_2(q)$ yields a disjoint union of $I_2(p)$-cells or $I_2(q)$-cells, which by Theorem \ref{I_2(p)-cells} are Dynkin diagrams with Coxeter numbers $h_{\Gamma}$ dividing $p$ and $h_{\Delta}$ dividing $q$. 
    Moreover, $\Gamma$ and $\Delta$ share no edges and $\Lambda$ is bipartite (vertices with labeling $\{1, 3\}$ or $\{2, 4\}$ are one color, and vertices with labeling $\{1, 4\}$ or $\{2, 3\}$ are the other color).

    To show that $(\Gamma, \Delta)$ is admissible, notice that $\Gamma$ exactly represents the action of the element $q^{-1/2}(T_1 + T_2) - (q^{1/2} - q^{-1/2})$ in the $\mathcal{H}$-module $M_{\Lambda}$ on the elements $v\in V$ corresponding to the vertices of $\Lambda$.
    Similarly, $\Delta$ represents the action of the element $q^{-1/2}(T_3 + T_4) - (q^{1/2} - q^{-1/2})$ in $M_{\Lambda}$ on the elements $v\in V$.
    Since $T_1, T_2$ commute with $T_3, T_4$ in $M_{\Lambda}$, these two elements also commute, so $\Gamma\Delta = \Delta\Gamma$.
    Thus, $(\Gamma, \Delta)$ is an admissible Dynkin biagram.

    Conversely, suppose that $(\Gamma, \Delta)$ is a connected admissible Dynkin biagram such that the Coxeter numbers $h_{\Gamma}, h_{\Delta}$ divide $p$ and $q$ respectively.
    There are four potential $W$-cell structures $(V, m, \tau)$.
    First, let $V$ be the vertices in $\Gamma, \Delta$ and let $m = \Gamma + \Delta$.
    Choose a vertex $v_0$ in $(\Gamma, \Delta)$.
    If both $\Gamma$ and $\Delta$ are nontrivial, then let $v_0$ be assigned a $\tau$-labeling of $\{1, 3\}, \{2, 3\}, \{1, 4\}$ or $\{2, 4\}$.
    Then, $\tau$ is determined for all other vertices by letting $\Gamma$-edges be of type (ii) and $\Delta$-edges be of type (iii).

    If $\Delta$ is trivial, then let $v_0$ be assigned a $\tau$-labeling of $\{1\}, \{2\}, \{1, 3, 4\}$ or $\{2, 3, 4\}$.
    Then, $\tau$ is determined for all other vertices by letting all edges be of type (ii).
    Similarly if $\Gamma$ is trivial, we let $v_0$ be assigned $\{3\}, \{4\}, \{1, 2, 3\}$ or $\{1, 2, 4\}$ and let all edges be of type (iii).

    By Theorem \ref{I_2(p)-cells}, restrictions of $(\Gamma, \Delta)$ to either $\Gamma$ or $\Delta$ yields admissible $I_2(p)$-graphs or $I_2(q)$-graphs.
    To show this is a nonnegative $I_2(p)\times I_2(q)$-cell, the only remaining relation to check is that $T_1, T_2$ commutes with $T_3, T_4$.

    A straightforward calculation shows that when $\tau(u) = \{1, 3\}$ and $\tau(v) = \{2, 4\}$, the coefficient of $v$ in $q^{-1}(T_1T_3 - T_3T_1)(u)$ is the difference between the number of blue-red paths and red-blue paths from $u$ to $v$. 
    Otherwise, the coefficient vanishes. Since $(\Gamma, \Delta)$ is admissible, the coefficients always vanish, so $T_1T_3 - T_3T_1 = 0$.
    Similar arguments show that $T_1$ commutes with $T_4$ and $T_2$ commutes with $T_3, T_4$.
\end{proof}

\begin{remark}
    The proof of Theorem \ref{I_2(p)xI_2(q)-cells} is identical to Stembridge's proof (Theorem 6.5), with the exception of using Theorem \ref{I_2(p)-cells} instead of the analogous statement for admissible $I_2(p)$-cells.
    As a result, we are left with a Dynkin biagram structure, which can have nonsymmetric edge weights instead of a bigraph structure.
\end{remark}

\section{Conjectures}
\label{conjectures}
Theorem \ref{mainThm} has further implications that hint at a strong connection between root systems and bipartite cluster algebras.
Let $B$ be a Zamolodchikov periodic $B$-matrix, and let $\boldsymbol{t}_i^{\lambda}$ be its corresponding tropical $T$-system.
Recall the form for tropical mutation at index $k\in [n]$:
\begin{align*}
    \boldsymbol{t}_k^{\lambda}(t + 1) + \boldsymbol{t}_k^{\lambda}(t - 1) = \max\left(\sum_i\Gamma_{ik}\boldsymbol{t}_i^{\lambda}(t), \sum_j \Delta_{jk}\boldsymbol{t}_j^{\lambda}(t)\right).
\end{align*}
For a generic initial labeling $\lambda\in\mathbb{R}^n$, we can color this mutation either red or blue, depending on the evaluation of the max function. 
If the maximum is $\sum_i\Gamma_{ik}\boldsymbol{t}_i^{\lambda}(t)$, then we call this a $\Gamma$-mutation, and color it red.
Otherwise if the maximum is $\sum_j\Delta_{jk}\boldsymbol{t}_j^{\lambda}(t)$, then we call this a $\Delta$-mutation, and color it blue.

Recall that the number of roots in a root system is $h\cdot r$, where $h$ is the Coxeter number and $r$ is the rank.
If we color mutations for one period $h_{\Gamma} + h_{\Delta}$ in the tropical $T$-system, we observe an interesting pattern:

\begin{conjecture}
\label{conj:1}
    Let $(\Gamma, \Delta)$ be an admissible Dynkin biagram with $r$ vertices.
    The number of $\Gamma$-mutations in one period is the number of roots in the root systems associated to each $\Gamma$-component, totaling to $h_{\Gamma}\cdot r$.
    Similarly, the number of $\Delta$-mutations in one period is the total number of roots in the root systems associated to each $\Delta$-component, $h_{\Delta}\cdot r$.
\end{conjecture}

This conjecture has been verified by SageMath on all non-ADE families in Theorem \ref{classThm} for small $n$, as well as all infinite ADE families for $n\leq 10$.
\begin{example}
    Let us consider the Dynkin biagram given by $A_2\otimes A_3$, with $\Gamma$ consisting of two $A_3$ components, and $\Delta$ consisting of three $A_2$ components.
    The Coxeter numbers of each are $h_{\Gamma} = 4$ and $h_{\Delta} = 3$, respectively.
    So, the length of one period is $h_{\Gamma} + h_{\Delta} = 7$.
    In the tropical $T$-system, this corresponds to iterating $t$ by 14 time steps, as each time step only mutates half of the vertices.
    
    Suppose our initial labeling is $\lambda = [2, 0, -0.9, 1, -1, 4]\in\mathbb{R}^6$.
    In Table \ref{table:evolution}, we can see the evolution of the $T$-system $\boldsymbol{t}^{\lambda}(t)$ over half a period of mutations.
    In this case, after half a period we reach the same state as $t = 0$, but with the rows flipped.
    The second half of the period behaves identically.
    The total number of $\Gamma$-mutations in one period is $6h_{\Gamma} = 24$, which is two times the number of roots in the $A_3$ root system.
    Similarly, the total number of $\Delta$-mutations in one period is $6h_{\Delta} = 18$.
 \begin{table} 
  \begin{tabular}{|c|c|c|c|c|c|c|c|c|}\hline
      $t$&
      $0$&
      $1$&
      $2$&
      $3$&
      $4$&
      $5$&
      $6$&
      $7$\\\hline
     $\boldsymbol{t}^{\lambda}(t)$&
     \sm{2}{0}{-0.9}{1}{-1}{4}&
     \sm{\textcolor{blue}{\textbf{-1}}}{}{\textcolor{blue}{\textbf{4.9}}}{}{\textcolor{red}{\textbf{6}}}{}&
     \sm{}{\textcolor{blue}{\textbf{6}}}{}{\textcolor{red}{\textbf{5}}}{}{\textcolor{red}{\textbf{2}}}&
     \sm{\textcolor{red}{\textbf{7}}}{}{\textcolor{red}{\textbf{1.1}}}{}{\textcolor{red}{\textbf{1}}}{}&
     \sm{}{\textcolor{red}{\textbf{2.1}}}{}{\textcolor{blue}{\textbf{2}}}{}{\textcolor{blue}{\textbf{-0.9}}}&
     \sm{\textcolor{red}{\textbf{-4.9}}}{}{\textcolor{red}{\textbf{1}}}{}{\textcolor{blue}{\textbf{1.1}}}{}&
     \sm{}{\textcolor{blue}{\textbf{-1}}}{}{\textcolor{red}{\textbf{-0.9}}}{}{\textcolor{red}{\textbf{2}}}&
     \sm{\textcolor{blue}{\textbf{4}}}{}{\textcolor{blue}{\textbf{1}}}{}{\textcolor{red}{\textbf{0}}}{}
     \\\hline
   \end{tabular}
   \caption{\label{table:evolution}The evolution of the tropical $T$-system of type $A_2\otimes A_3$ for one half period. The red boldface numbers are the $\Gamma$-mutations, while the blue boldface numbers are the $\Delta$-mutations.}
 \end{table}
\end{example}

In fact, one can visualize the space of all initial labelings $\lambda\in\mathbb{R}^r$ as partitioned into chambers based on the sequence of colored mutations $\lambda$ produces in one period.
Going from one chamber to another involves crossing a linear wall which swaps some number of red and blue mutations in the period.
Let $\Lambda$ be a Dynkin diagram.
For the small case of $\Lambda\otimes A_1$, we observed the following, implying a potential bijection between chambers and clusters.

\begin{conjecture}
    Let $n \geq 2$. 
    Let $\Lambda \in \{A_{2n}, B_{2n}, C_{2n}, E_n, F_4, G_2\}$ be a Dynkin diagram of rank $k$.
    Consider the Dynkin biagram $\Lambda\otimes A_1$.
    The space of initial labelings $\mathbb{R}^{k}$ is partitioned into $c_{\Lambda}$ distinct chambers, where $c_{\Lambda}$ is the number of seeds in the cluster algebra of type $\Lambda$.
\end{conjecture}

The following conjecture generalizes Conjecture \ref{conj:1}, and applies to affine $\boxtimes$ finite ADE bigraphs as defined in \cite{subadd}.
\begin{conjecture}
    Let $(\Gamma, \Delta)$ be an admissible biagram with $r$ vertices where the components of $\Gamma$ are affine Dynkin diagrams and the components of $\Delta$ are Dynkin diagrams. 
    The number of $\Delta$-mutations is bounded above by the total number of $\Delta$-roots, $h_{\Delta}\cdot r$.
\end{conjecture}

This conjecture has also been verified by SageMath on all 19 families in the classification from \cite{subadd} of affine $\boxtimes$ finite ADE bigraphs for $n\leq 5$.

\bibliographystyle{alpha}
\bibliography{biblio}

\end{document}